\documentclass[12pt,notitlepage]{article}

\usepackage{amsthm}
\usepackage{psfrag}
\usepackage{cite}
\usepackage{amsmath}
\usepackage{amssymb}
\usepackage{latexsym}
\usepackage{geometry}
\usepackage{graphicx,caption}
\usepackage{xcolor}
\newtheorem{theorem}{Theorem}

\newtheorem{lemma}[theorem]{Lemma}

\newtheorem{problem}[theorem]{Main Problem}
\newtheorem{rem}[theorem]{Remark}

\newcommand{\calC}{{\mathcal C}}
\newcommand{\calM}{{\mathcal M}}
\newcommand{\calP}{{\mathcal P}}

\newcommand{\ivC}{\textbf{\textsl{c}}}

\newcommand{\ivP}{\textbf{\textsl{p}}}
\newcommand{\ivS}{\textbf{\textsl{s}}}
\newcommand{\ivT}{\textbf{\textsl{t}}}
\newcommand{\ivX}{\textbf{\textsl{x}}}
\newcommand{\ivY}{\textbf{\textsl{y}}}
\newcommand{\ivZ}{\textbf{\textsl{z}}}
\newcommand{\ivXX}{\textbf{\textsl{X}}}

\newcommand{\B}{\mathbb}

\newcommand{\R}{{\B R}}

\newcommand{\lo}[1]{\underline{#1}}
\newcommand{\hi}[1]{\overline{#1}}
\newcommand{\cc}{\bf\tt}
\newcommand{\comment}[1]{\mbox{}}

\title{The number of relative equilibria in the PCR4PB}

\author{
Jordi-Llu\'is Figueras\\
Department of Mathematics\\
Uppsala University\\
Box 480, Uppsala, Sweden \\
{\tt figueras}{\rm @}{\tt math.uu.se}
\and
Warwick Tucker\\
School of Mathematics\\
Monash University\\
Melbourne, Australia \\
{\tt warwick.tucker}{\rm @}{\tt monash.edu}
\and
Piotr Zgliczynski \\
Institute of Computer Science\\
Jagiellonian University\\
30--348 Krakow, Poland\\
{\tt umzglicz}{\rm @}{\tt cyf-kr.edu.pl}
}

\begin{document}
\maketitle

\abstract{The aim of this paper is to present a new, analytical, method for computing the exact number of relative equilibria in the planar, circular, restricted 4-body problem of celestial mechanics. The new approach allows for a very efficient computer-aided proof, and opens a potential pathway to proving harder instances of the $n$-body problem.}

\setlength{\parindent}{0cm}
\setlength{\parskip}{1.5ex}

\section{Introduction}\label{sec:introduction}

The relative equilibria of the $3$--body problem have been known for centuries. In terms of equivalence classes, there are -- irrespectively of the masses -- exactly five relative equilibria. Three of these are the collinear configurations discovered by Euler \cite{Euler1767}; the remaining two are Lagrange's equilateral triangles \cite{Lagrange1772}, see Figure~\ref{fig:relEq1-2}.

\begin{figure}[ht]
\begin{center}
\includegraphics[width=0.25\linewidth]{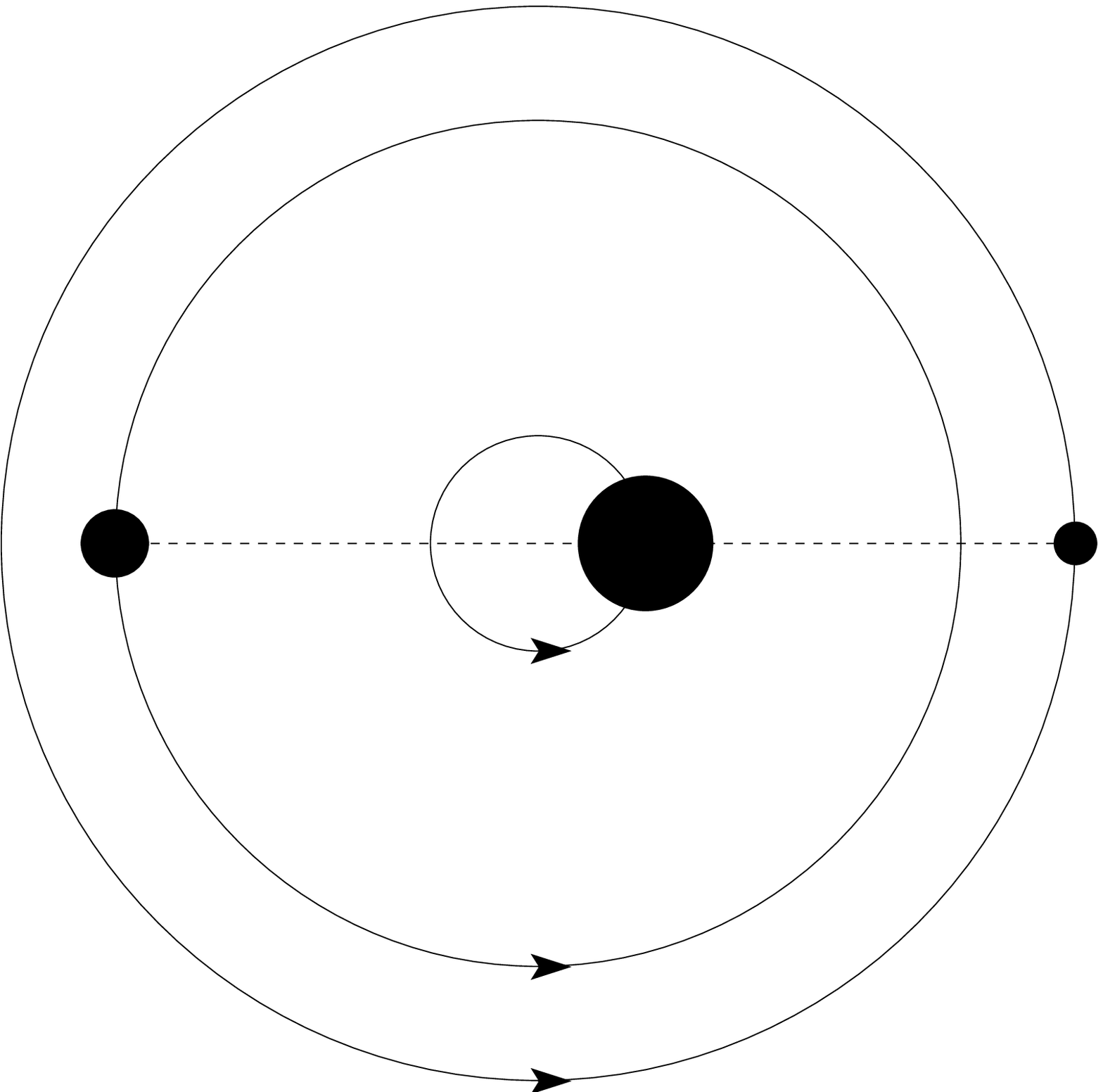}
\hspace{0.15\linewidth}
\includegraphics[width=0.25\linewidth]{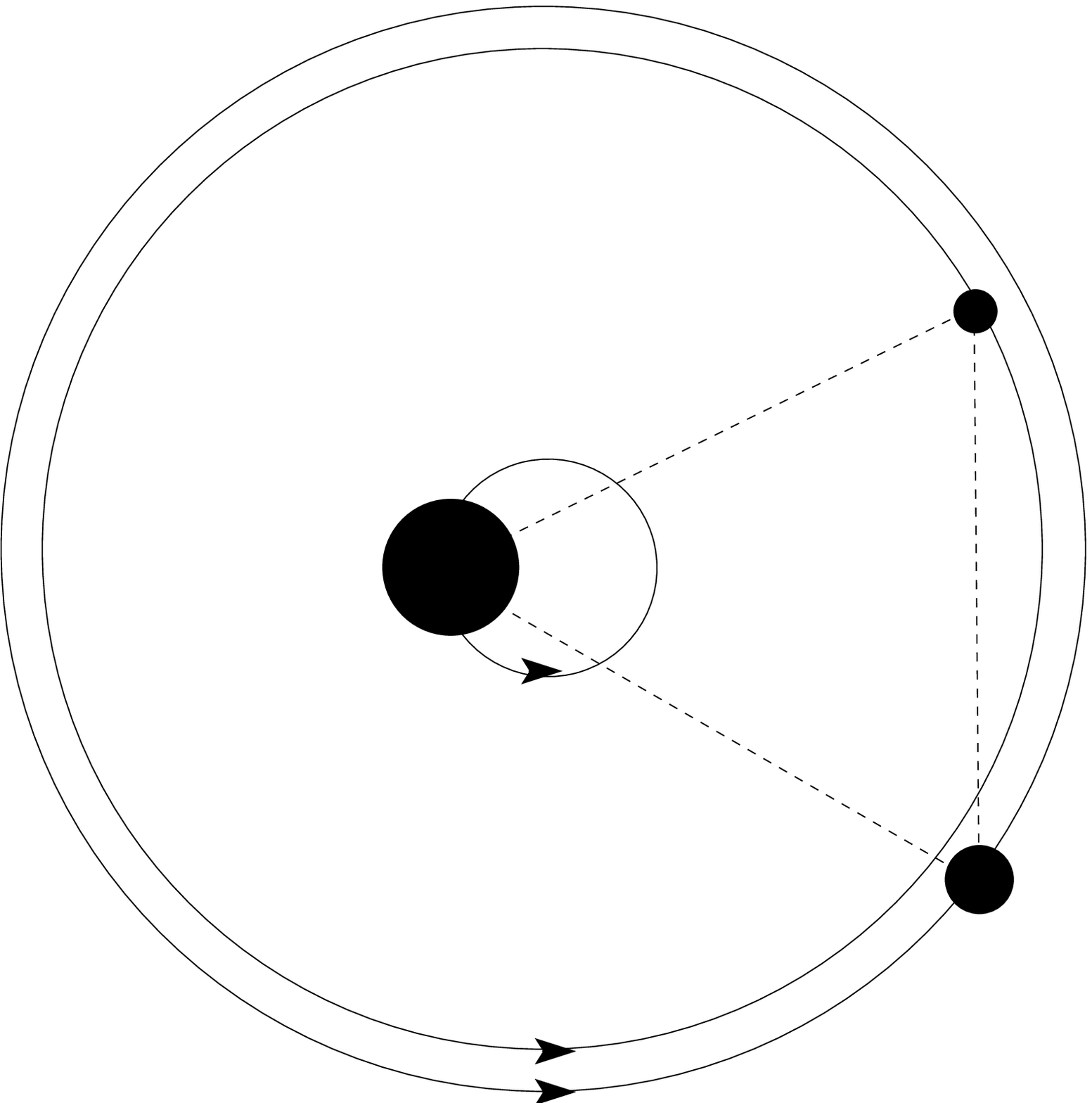}
\end{center}
\captionsetup{width=0.8\linewidth}
\caption{Relative equilibria for the $3$--body problem. (a) The collinear case of Euler; (b) The equilateral triangle case of Lagrange.}
\label{fig:relEq1-2}
\end{figure}

The collinear configurations found by Euler have been generalized for $n$ bodies by Moulton \cite{Moulton1910}. There are exactly $n!/2$ such collinear equivalence classes.

In 2006, Hamilton and Moeckel \cite{HamtonMoeckel06} proved that the number of relative equilibria of the Newtonian $4$--body problem is finite (always between 32 and 8472). Their computer-aided proof is based on symbolic and exact integer computations. The upper bound 8472 is believed to be a large overestimation; numerical simulations suggest that no more than 50 equilibria exists, see e.g. \cite{Simo1978}.

Albouy and Kaloshin \cite{AlbouyKaloshin12} almost settled the question of finiteness for $n = 5$ bodies. They proved that there are finitely many relative equilibria in the Newtonian $5$--body problem, except perhaps if the $5$--tuple of positive masses belongs to a given co-dimension--2 subvariety of the mass space. By B\`ezout's theorem, an upper bound on the number of relative equilibria is obtained (outside the exceptional subvariety), but the authors conclude \textit{``However, the bound is so bad that we avoid writing it explicitly''}.

Relaxing the positivity of the masses can produce a continuum of relative equilibria. In \cite{Roberts1999} Roberts demonstrated this for the $5$--body problem with one negative mass.

Looking at the restricted $4$--body problem (i.e. when one of the planets has an infinitesimally small mass), Kulevich et al \cite{KulevichEtAl10} proved finiteness with an upper bound of 196 relative equilibria. This, however, was assumed to be a great overestimation, and later Barros and Leandro \cite{BarrosLeandro11, BarrosLeandro14} proved that there could only be 8, 9 or 10 relative equilibria -- depending on the three primary masses. The proof is based on techniques from algebraic (or tropical) geometry, used to count solutions to systems of very large polyomial equations.

In this paper, we present a new approach to counting relative equilibria in various settings. We use techniques from real analysis rather than (complex) algebraic geometry, an approach we believe will generalize better to more complicated settings such as the full 4--body problem, which still remains unresolved. Moreover, the techniques presented here do not use algebraic properties of the system, only differentiability. This quality could play a role in more general contexts, like in curved spaces or in physical systems where the potential is not given by Newton's laws of gravitation.

In what follows, we will focus on the planar, circular, restricted 4--body problem, and give a new proof of the results of Barros and Leandro. In this setting, the three primary bodies form an equilateral triangle as in Figure~\ref{fig:relEq1-2} (b).

\section{Formulating the problem}

Let $m_1, m_2, m_3$ denote the positive masses of the three primaries, and let $p_1, p_2, p_3$ denote their positions in $\mathbb{R}^2$, which
form an equilateral triangle. Also let $z$ be the position of the fourth (weightless) body. In this setting, the gravitational pull on $z$ is described by the \textit{amended potential}:
\begin{equation}\label{eq:potential_function}
 V(z;m) = \frac{1}{2}\|z - c\|^2 + \sum_{i=1}^3 m_i\|z - p_i\|^{-1}.
\end{equation}
Here $c$ denotes the center of mass of the primaries. It follows that the locations of the relative equilibria are given by the critical points of $V$. Thus, the challenge of counting the number of relative equilibria can be translated into the task of counting the critical points of $V$, given the appropriate search region $\calC\times\calM$. Here $\calC \subset \mathbb{R}^2$ is the set of positions the of weightless body, and $\calM$ is the set of masses, discussed in Sections~\ref{subsec:the_configuration_space} and~\ref{subsec:the_mass_space}, respectively.

We are now prepared to formulate the main problem of this study:

\begin{problem}\label{pro:original}
How many solutions can the critical equation $\nabla_z V(z;m) = 0$ have (in $\calC$) when $m \in\calM$?
\end{problem}

\begin{figure}[ht]
\begin{center}
\includegraphics[width=0.49\linewidth]{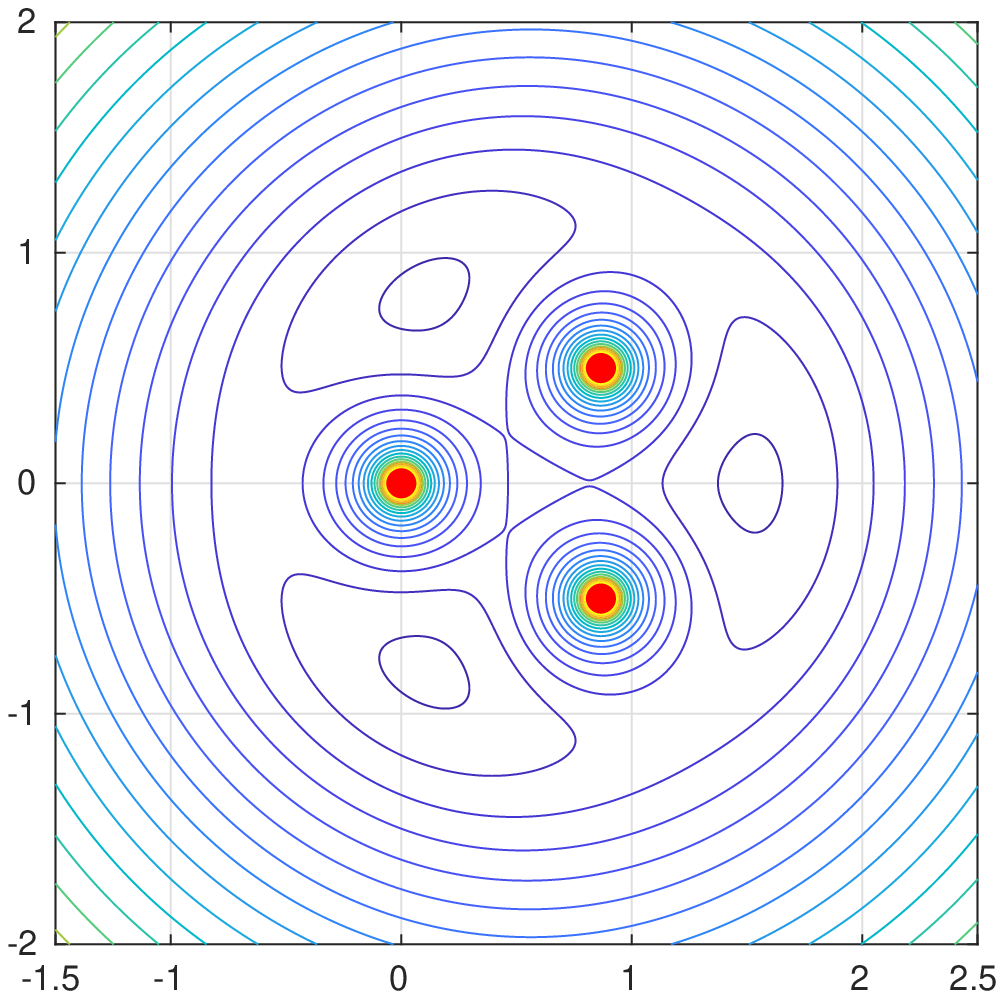}
\hspace{-0.1cm}
\includegraphics[width=0.49\linewidth]{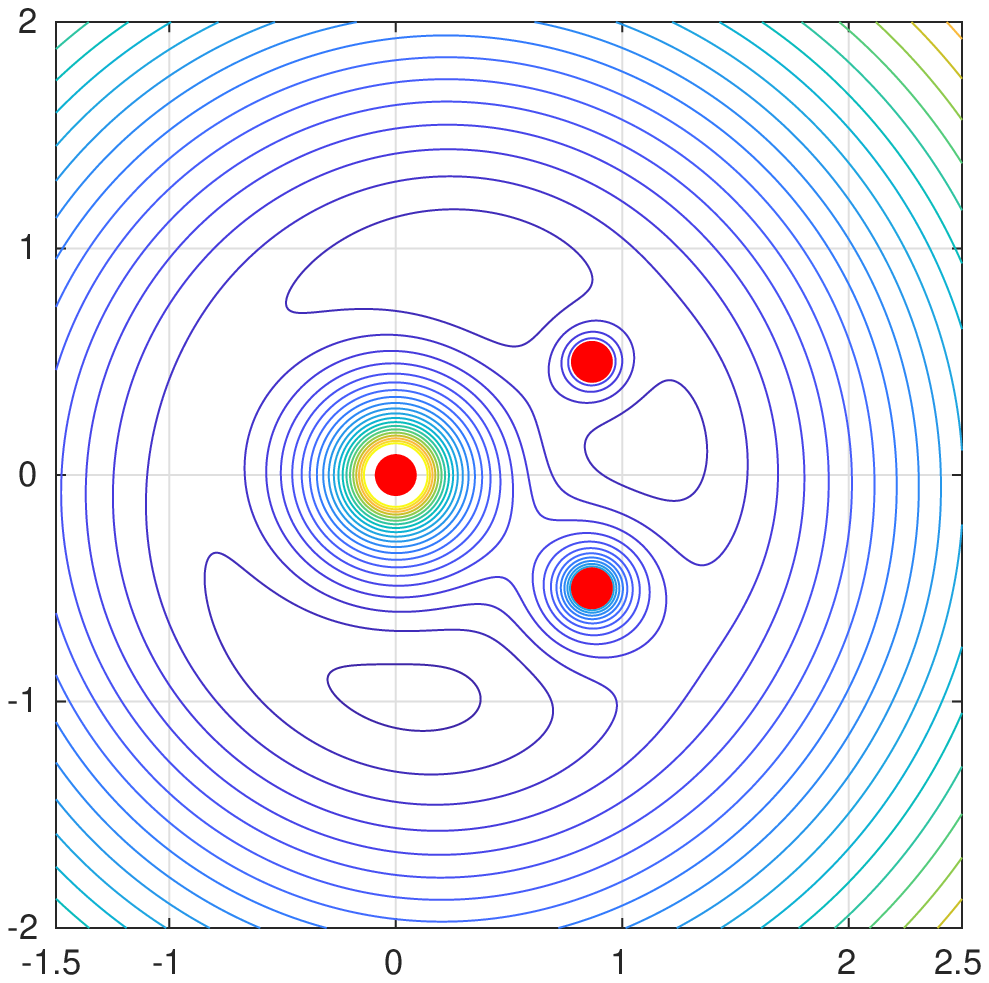}
\end{center}
\captionsetup{width=0.8\linewidth}
\caption{Level curves of the amended potential (\ref{eq:potential_function}) with the three primaries forming an equilateral triangle. (a) Equal masses: $m_1 = m_2 = m_3 = 1/3$; (b) Different masses: $m_1 = 1/16$, $m_2 = 3/16$, $m_3 = 12/16$ with the heaviest primary located at the origin.}
\label{fig:amendedPotential}
\end{figure}

With our approach, there are three difficulties that must be resolved in order to be successful:
\begin{enumerate}
\item The potential $V$ (and its gradient) is singular at the three primaries.
\item When the mass of a primary tends to zero, several critical points of $V$ tend to that primary.
\item As masses (not close to zero) are varied, the number of critical points of $V$ may change due to bifurcations taking place.
\end{enumerate}

When all masses are uniformly bounded away from zero, the singularities at the primaries (case 1) can be handled by proving that no critical points can reside in certain small disks centered at the primaries. This is explained further in section~\ref{subsec:the_configuration_space}.

For masses approaching zero (case 2), we end up with a multi-restricted\footnote{A restricted $n$-body problem is said to be of type $m+k$ is there are $m$ bodies with positive mass and $k$ weightless bodies, and $n = m+k$.} problem of type $2 + 2$ ($m_1\to 0$) or of type $1 + 3$ ($m_1\to 0$ and $m_2\to 0$). Both scenarios can be resolved by desingularizing the potential $V$ along the circle $\|z - p_3\| = 1$. More about this in sections~\ref{sec:polar_coordinates} and ~\ref{subsubsec:small_masses}.

Out of these difficulties, the bifurcations (case 3) are the easiest to resolve. In principle, this only requires verifying the sign of certain combinations of partial derivatives of the potential $V$.  We will address this in section~\ref{subsec:bifurcation_analysis}.

\subsection{The mass space $\calM$}\label{subsec:the_mass_space}

Without loss of generality, we may assume that the masses of the primaries are normalized $m_1 + m_2 + m_3 = 1$ and ordered $0 < m_1 \le m_2\le m_3$. Call this set $\calM$:

\begin{equation}\label{eq:mass_space}
\calM = \{m\in\mathbb{R}^3 \colon m_1 + m_2 + m_3 = 1 \textrm{ and } 0 < m_1 \le m_2\le m_3\}.
\end{equation}

We illustrate the unordered and ordered mass space in Figure~\ref{fig:mass_space}. Note that the midpoint of the large triangle corresponds to all masses being equal: $m = (\tfrac{1}{3},\tfrac{1}{3},\tfrac{1}{3})$.

\begin{figure}[ht]
\psfrag{l1}{$m = (1,0,0)$}
\psfrag{L2}{$m = (0,\tfrac{1}{2},\tfrac{1}{2})$}
\psfrag{L3}{$m = (0,0,1)$}
\psfrag{L4}{$m^\star$}
\psfrag{L5}{$m = (\tfrac{1}{3},\tfrac{1}{3},\tfrac{1}{3})$}
\begin{center}
\hspace*{-25mm}
\includegraphics[width=0.4\linewidth]{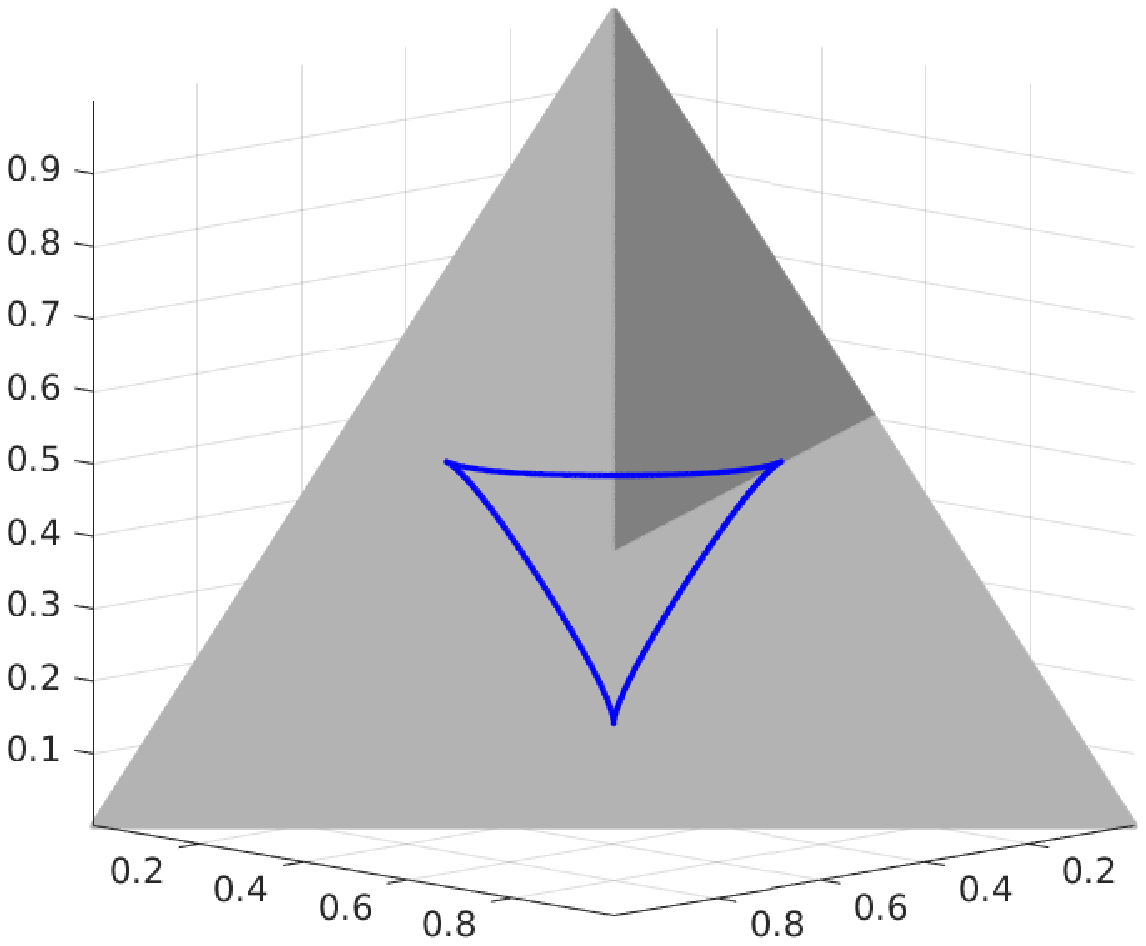}
\hspace*{15mm}
\includegraphics[width=0.125\linewidth]{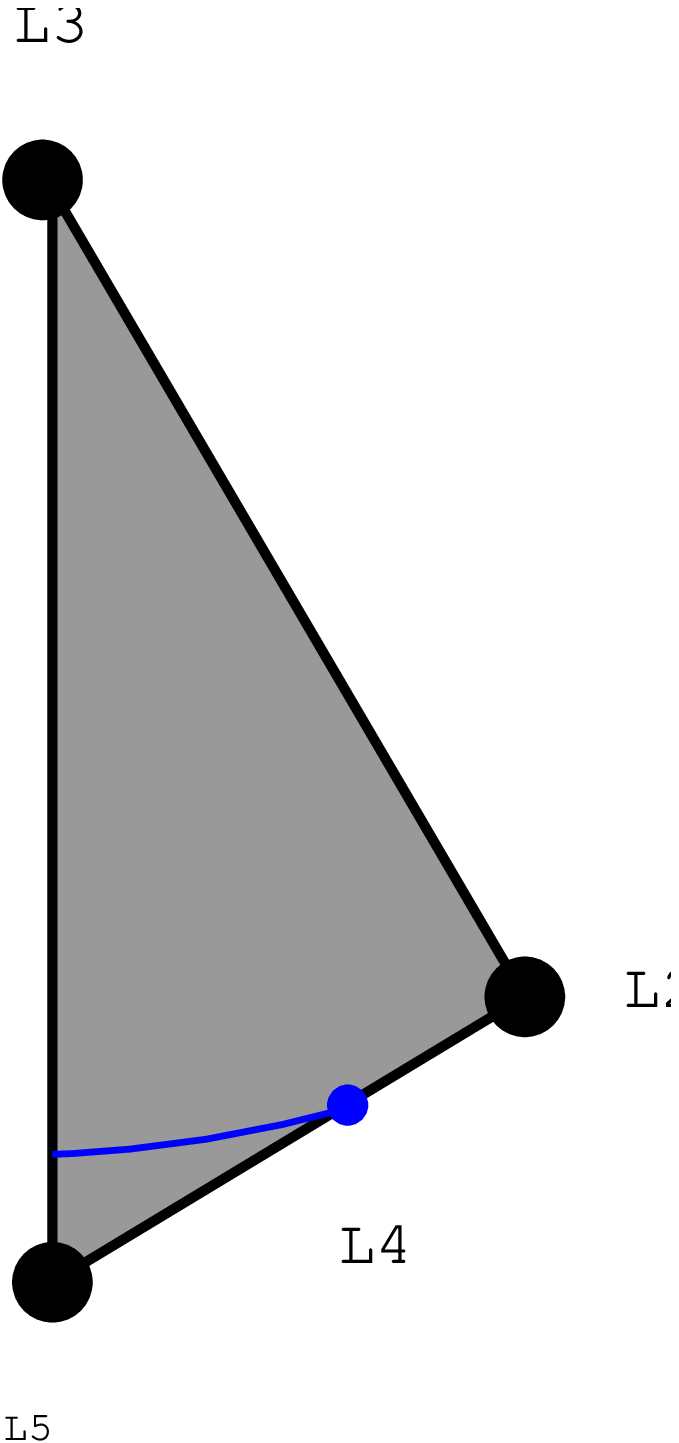}
\end{center}
\captionsetup{width=0.8\linewidth}
\caption{(a) The normalised mass space $m_1 + m_2 + m_3 = 1$ in barycentric coordinates, with the region of ordered masses ($0 < m_1 \le m_2 \le m_3$) highlighted in a darker shade. The blue curve illustrates the set on which bifurcations take place. (b) $1/6$ of the normalized mass space corresponding to the space $\calM$. The mass $m^\star$ corresponds to a cubic-type bifurcation.}
\label{fig:mass_space}
\end{figure}

The bifurcations taking place are of two kinds: quadratic and cubic. The latter are rare, and come from the inherent 6-fold symmetry of the normalized mass space, see Figure~\ref{fig:mass_space}. The quadratic bifucations are of a saddle-node type, in which two unique solutions approach each other, merge at the bifucration, and no longer exists beyond the point of bifurcation, see Figure~\ref{fig:bifurc}.

\begin{figure}[h]
\begin{center}
\includegraphics[width=0.2\linewidth]{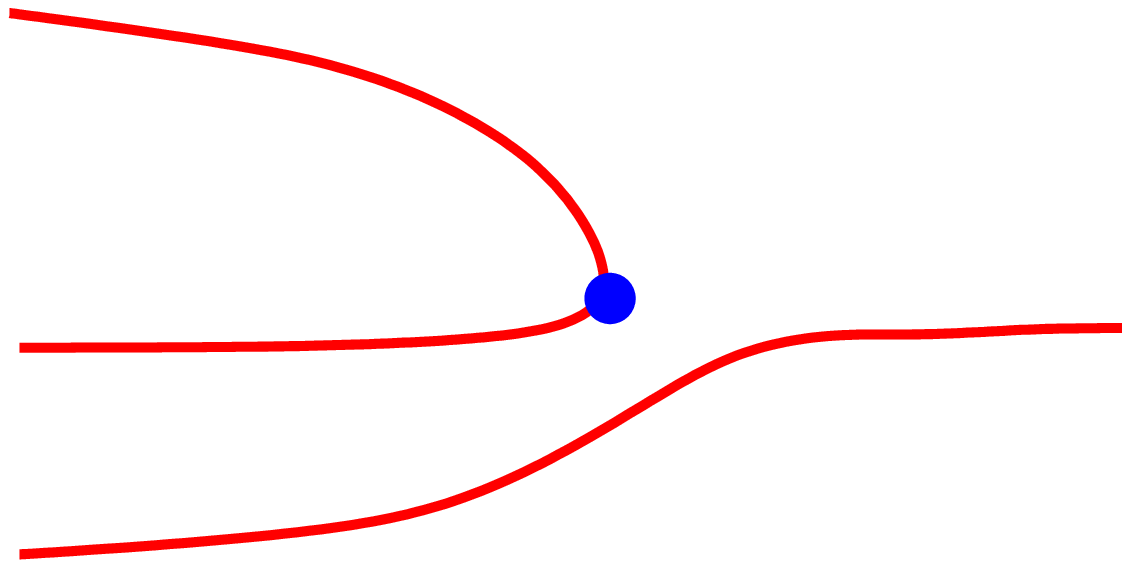}
\hspace*{9mm}
\includegraphics[width=0.2\linewidth]{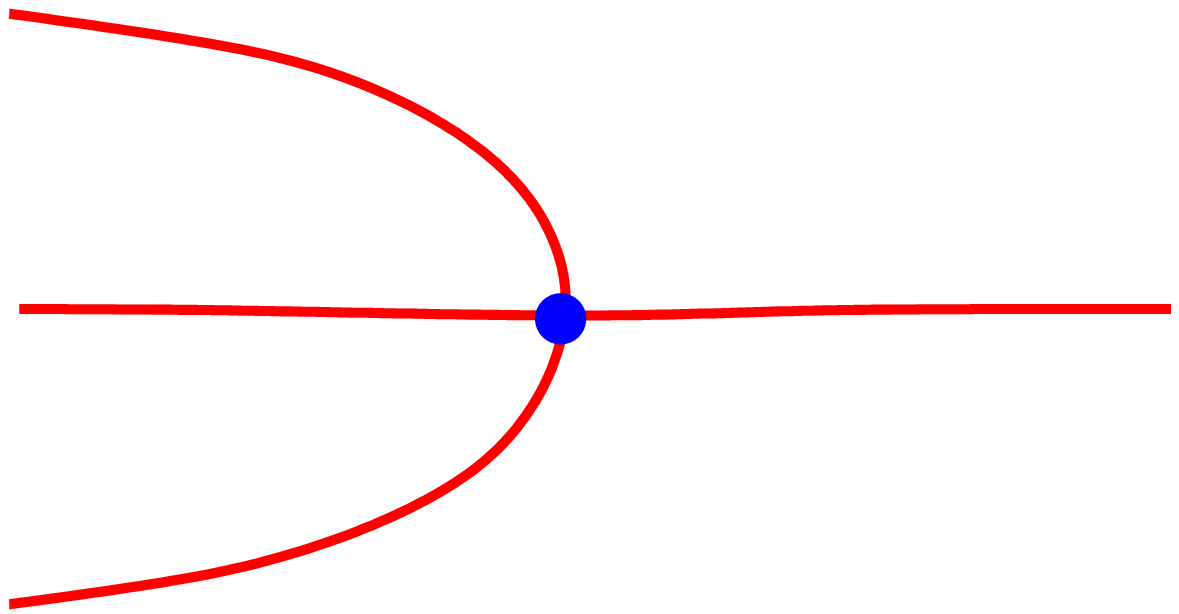}
\hspace*{9mm}
\includegraphics[width=0.2\linewidth]{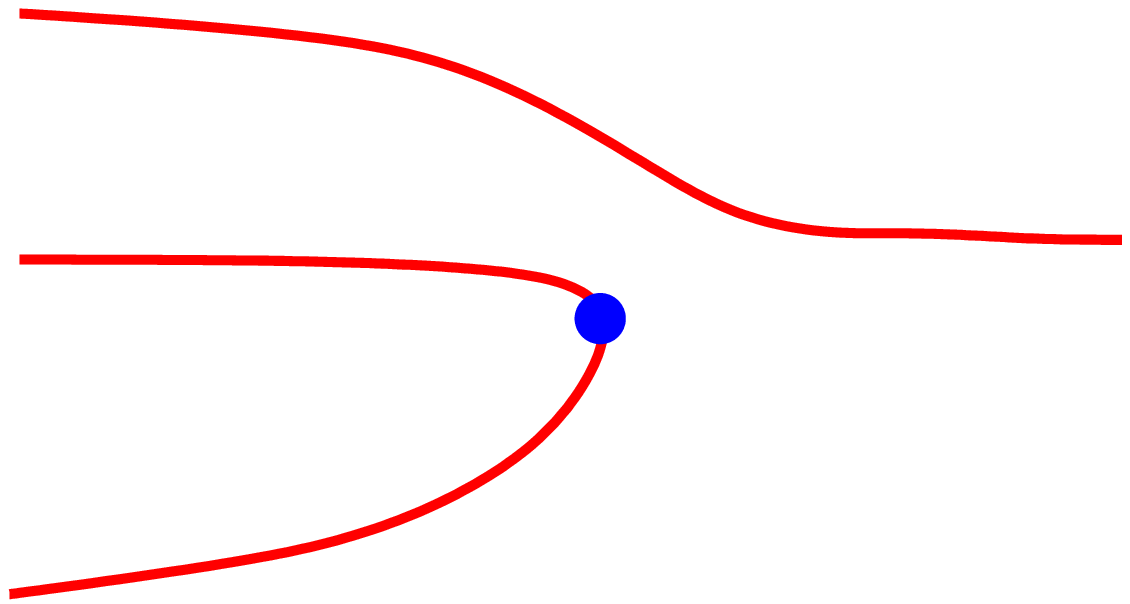}
\end{center}
\captionsetup{width=0.8\linewidth}
\caption{The typical quadratic bifurcation (blue dot), in which two solutions merge and vanish, becomes a cubic-type bifurcation at $m^\star$ due to symmetry.}
\label{fig:bifurc}
\end{figure}

In section~\ref{subsec:bifurcation_analysis} we will present the mathematics needed to resolve the two types of bifurcations taking place.

\subsection{The configuration space $\calC$}\label{subsec:the_configuration_space}

Without loss of generality, we may fix the positions of the three primaries: $p_1 = (\tfrac{\sqrt{3}}{2}, +\tfrac{1}{2})$, $p_2 = (\tfrac{\sqrt{3}}{2}, -\tfrac{1}{2},)$, and $p_3 = (0,0)$, thus forming an equilateral triangle with unit length sides, as in Figure~\ref{fig:relEq1-2} b.

We begin by deriving some basic results that will be used later on. Let $\nabla_z V(z;m)$ denote the gradient (with respect to $z$) of the potential $V$. A relative equilibria is then simply a solution to the equation $\nabla_z V(z;m) = 0$, i.e., a critical point of $V(z;m)$.

In what follows, it will be convenient to adopt the following notation:
\begin{equation}\label{eq:r_i}
r_i = r_i(z) = \|z - p_i\| \quad\quad (i = 1,2,3).
\end{equation}

In determining the relevant configuration space, we will use the following two exlusion results (Lemma~\ref{lemma:near_distance_test_2} and Lemma~\ref{lemma:near_distance_test_4}, respectively):
\begin{itemize}
\item Assume that $m_3 \geq 1/3$. If $r_3(z) \leq 1/3$, then $z$ is not a critical point of the potential $V$.
\item If $r_i(z) \ge 2$ for some $i =1,2,3$, then $z$ is not a critical point of the potential $V$.
\end{itemize}

Combining these two results, we see that, for $m\in\calM$, all relative equilibria must satisfy $1/3\le r_3\le 2$. We will take this to be our global search region $\calC$ in configuration space:
\begin{equation}\label{eq:configuration_space}
\calC = \{z\in\mathbb{R}^2 \colon 1/3 \le \|z\| \le 2\}.
\end{equation}

A more detailed analysis reveals that the relative eqililibria all reside in an even smaller subset of $\calC$, see Figure~\ref{fig:configuration_space}(a). We will, however, not use this level of detail in our computations.

\begin{figure}[h]
\begin{center}
\includegraphics[width=0.45\linewidth]{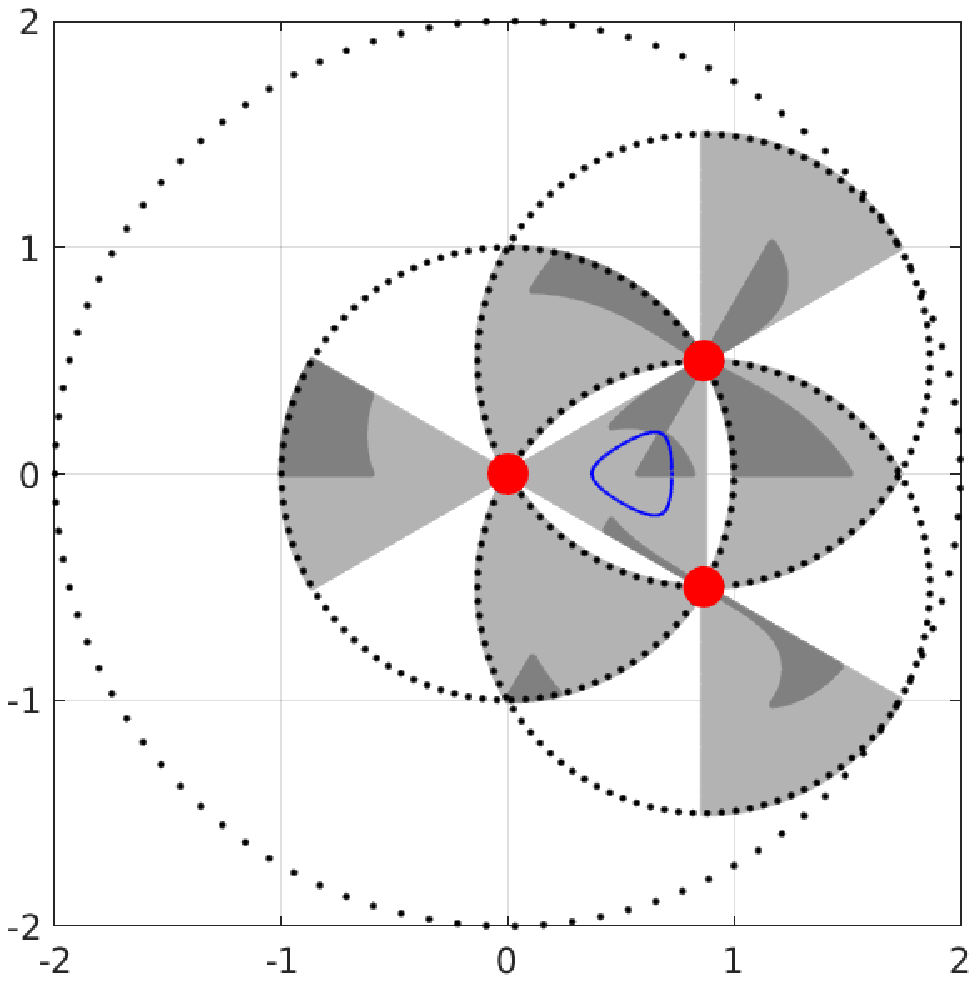}
\hspace{-5mm}
\includegraphics[width=0.45\linewidth]{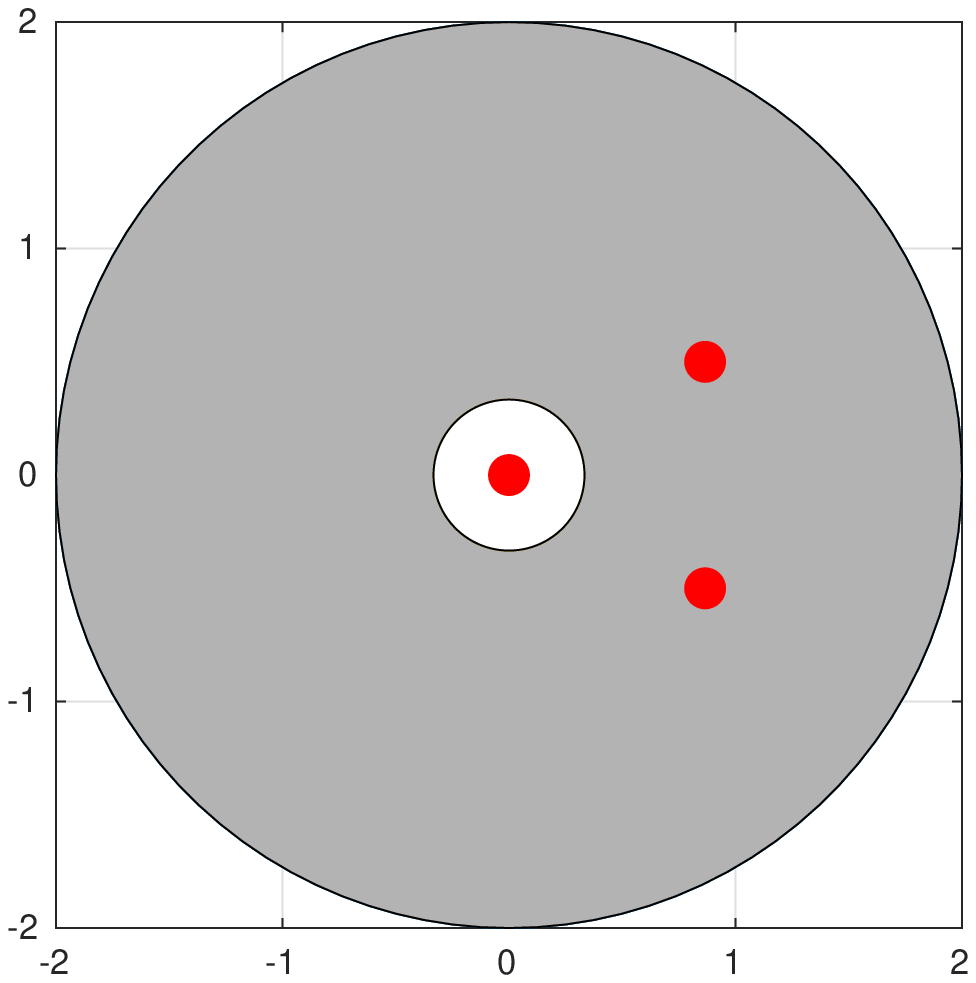}
\end{center}
\captionsetup{width=0.8\linewidth}
\caption{The configuration space with the three primaries $p_1, p_2, p_3$ (red points) spanning an equilateral triange. (a) The seven regions (colored light grey) in phase space where all relative equilibria must reside. When the masses are ordered, all relative equilibria are restricted to a smaller subset (colored dark grey). The blue curve illustrates the set on which bifurcations take place. (b) The global search region $\calC = \{z\in\mathbb{R}^2 \colon 1/3 \le \|z\| \le 2\}$ of the configuration space. Note that the heaviest primary $p_3 = (0,0)$ does not belong to $\calC$.}
\label{fig:configuration_space}
\end{figure}

We end this section by deriving the two exclusion results used above. The critical equation $\nabla_z V(z;m) = 0$ can be written as
\begin{equation}\label{eq:crit-p-full}
 (z-c) +  m_1 \nabla \frac{1}{r_1} + m_2 \nabla \frac{1}{r_2} + m_3 \nabla \frac{1}{r_3}  = 0.
\end{equation}
Note that we have

\begin{equation}\label{eq:nabla-r}
\left\| \nabla \frac{1}{r_i} \right\| = \frac{1}{r_i^2} \qquad (i = 1,2,3).
\end{equation}

The following lemma provides an a priori bound on how close a solution of $\nabla_z V(z;m) = 0$ can be to one of the primaries.
\begin{lemma}\label{lemma:near_distance_test_1}
Let $z \in \mathbb{R}^2$ and set $r_i = r_i(z)$. If for some $i =1,2,3$ we have
\begin{eqnarray}
  \frac{m_i}{r_i^2} &>& \frac{1-m_i}{(1-r_i)^2} + r_i + 1, \label{eq:near_distance_test_1}
 \end{eqnarray}
then $z$ is not a critical point of the potential $V$.
\end{lemma}

\begin{proof}
First, we note that Equation (\ref{eq:crit-p-full}) can be rewritten as
\begin{equation}\label{eq:cct}
  m_1 \nabla \frac{1}{r_1}= -m_2 \nabla \frac{1}{r_2} - m_3 \nabla \frac{1}{r_3} - (z-c).
\end{equation}
We use this formulation, and consider small $r_1$, so that the norm of the left-hand side is larger than the norms of terms appearing to the right.

Observe that (\ref{eq:near_distance_test_1}) implies that $r_i < 1$, because $ \frac{m_i}{r_i^2} > 1$ and $m_i < 1$.
Hence, throughout the proof, we assume that $r_1 < 1$.

From the triangle inequality $r_1 + r_2 \geq 1$, we have
\begin{eqnarray*}
 r_2 &\geq& 1-r_1 , \\
 r_2^2 &\geq& (1-r_1)^2, \\
 \|z-c\| &\leq& \|z-p_1\| + \|p_1 - c\| \leq r_1 + 1.
\end{eqnarray*}

Using this together with (\ref{eq:near_distance_test_1}), we obtain
\begin{eqnarray*}
  \left\|-m_2 \nabla \frac{1}{r_2} -  m_3 \nabla \frac{1}{r_3} - (z-c) \right\| \leq \frac{m_2}{r_2^2} + \frac{m_3}{r_3^2} + r_1 + 1 \\
    \leq  \frac{m_2}{(1-r_1)^2} + \frac{m_3}{(1-r_1)^2} + r_1 + 1= \frac{1-m_1}{(1-r_1)^2} + r_1 + 1 \\
    <  \frac{m_1}{r_1^2} = \left\|m_1 \nabla \frac{1}{r_1}\right\|,
\end{eqnarray*}
hence (\ref{eq:cct}) is not satisfied.
\end{proof}
Since we are only considering ordered masses, we always have $m_3\ge 1/3$. This fact, combined with Lemma~\ref{lemma:near_distance_test_1}, immediately gives us a uniform bound for the primary $p_3$.

\begin{lemma}\label{lemma:near_distance_test_2}
Assume that $m_3 \geq 1/3$. If $r_3(z) \leq 1/3$, then $z$ is not a critical point of the potential $V$.
\end{lemma}

\begin{proof}
We verify (\ref{eq:near_distance_test_1}) with $i = 3$. Estimating each side of the inequality, we find
\begin{equation*}
  \frac{m_3}{r_3^2} \geq \frac{1/3}{(1/3)^2} = 3 \qquad\mathrm{and}\qquad
  \frac{1-m_3}{(1-r_3)^2} + r_3 + 1 \leq \frac{2/3}{(2/3)^2} + \frac{1}{3} + 1 = \frac{3}{2} + \frac{4}{3} = \frac{17}{6} < 3.
\end{equation*}
\end{proof}

\textbf{Remark:} For $m_3 \ge 1/3$, we can exclude the slightly larger region $r_3 \le 0.3405784$.

Note that Lemma~\ref{lemma:near_distance_test_1} can also be used to exclude a small disc centered at $p_i$ ($i = 1,2$) whenever $m_i$ is not too small. As an example, if $m_i\ge \varepsilon > 0$, then we can exclude the disc $r_i \le \tfrac{1}{2}\sqrt{\varepsilon}$.

Another exclusion principle is given by the following result.
\begin{lemma}\label{lemma:near_distance_test_3}
Let $z \in \mathbb{R}^2$ and set $r_i = r_i(z)$. If for some $i = 1,2,3$ we have
\begin{eqnarray}
  r_i - 1 &>& \frac{1-m_i}{(1-r_i)^2} + \frac{m_i}{r_i^2}, \label{eq:near_distance_test_3}
 \end{eqnarray}
then $z$ is not a critical point of the potential $V$.
\end{lemma}

\begin{proof}
The proof uses the idea that we may rewrite Equation \ref{eq:crit-p-full} as
\begin{equation}\label{eq:crit-p-full-2}
m_1 \nabla \frac{1}{r_1} + m_2 \nabla \frac{1}{r_2} + m_3 \nabla \frac{1}{r_3}  = -(z - c).
\end{equation}
Without any loss of the generality, we may assume that $i=3$ and shift coordinate frame so that $p_3$ is situated at the origin. Then $\|z\| = \|z - p_3\| = r_3$, and from (\ref{eq:near_distance_test_3}) (with $i=3$) it follows that $r_3 > 1$. This, together with the triangle inequality, gives
\begin{eqnarray}
  \|z-c\| \geq \|z\| - \|c\| \geq r_3 - 1.  \label{eq:z-cnorm}
\end{eqnarray}
Here we use the fact that the center of mass $c$ is located within the triangle spanned by the three primaries, and therefore $\|c\|\le 1$.

Applying the triangle inequality again, we have $r_i + 1 \geq r_3$ ($i=1,2$), from which it follows that
\begin{equation}\label{eq:rj-far-test}
  r_i \geq r_3 - 1 \qquad\mathrm{ and }\qquad r_i^2  \geq (r_3-1)^2 \qquad\qquad(i = 1,2).
\end{equation}

Therefore we obtain the following estimate (we use (\ref{eq:nabla-r},\ref{eq:rj-far-test},\ref{eq:z-cnorm}))
\begin{eqnarray*}
  \left\|  m_1 \nabla \frac{1}{r_1}  + m_2 \nabla \frac{1}{r_2} +  m_3 \nabla \frac{1}{r_3}\right\| \leq \frac{m_1}{r_1^2} + \frac{m_2}{r_2^2} + \frac{m_3}{r_3^2} &\leq& \\
  \frac{m_1}{(r_3-1)^2} + \frac{m_2}{(r_3-1)^2} + \frac{m_3}{r_3^2} =
  \frac{1-m_3}{(r_3-1)^2} + \frac{m_3}{r_3^2}  &<&  r_3 - 1 \leq \|z-c\|,
\end{eqnarray*}
hence (\ref{eq:crit-p-full-2}) is not satisfied.
\end{proof}

A direct consequence of Lemma~\ref{lemma:near_distance_test_3} is the following uniform bound.

\begin{lemma}\label{lemma:near_distance_test_4}
If $r_i(z) \ge 2$ for some $i =1,2,3$, then $z$ is not a critical point of the potential $V$.
\end{lemma}
\begin{proof}
Using only $r_i(z) \ge 2$, we verify (\ref{eq:near_distance_test_3}). Indeed, a straight forward computation gives:
\begin{eqnarray*}
  \frac{1-m_i}{(r_i-1)^2} + \frac{m_i}{r_i^2} \leq 1-m_i + \frac{m_i}{4}=1-\frac{3 m_i}{4} < 1 \leq r_i -1.
\end{eqnarray*}
\end{proof}

\section{Reparametrizing the masses}\label{sec:reparametrizing_the_masses}

Due to the normalisation $m_1 + m_2 + m_3 = 1$, the mass space can be viewed as a 2-dimensional set parametrized by $m_1$ and $m_2$. Instead of working directly with the masses $(m_1, m_2)$, we introduce the following non-linear, singular transformation:
\begin{equation}\label{eq:s_and_t}
s = \frac{m_1}{m_1 + m_2}\qquad\mathrm{ and }\qquad t = m_1 + m_2,
\end{equation}
The new parameters $(s,t)$ can be transformed back to mass space via the inverse transformation:
\begin{equation}\label{eq:inv_s_and_t}
m_1 = st, \qquad m_2 = (1-s)t,\qquad (\textrm{and } m_3 = 1 - t).
\end{equation}

The reason for working in the mass space using $(s,t)$-coordinates is as follows: when $m_1$ and $m_2$ tend to zero, then some relative equilibria may move in a non-continuous way and no limit exists. This makes our kind of study virtually impossible. When seen in $(s,t)$-coordinates, however, the movements are regular and  amenable to our computer assisted techniques.

In the $(m_1,m_2)$-space, the ordered mass space $\calM$ is a right-angled triangle, see Figure~\ref{fig:mass_space}(b). Under the transformation (\ref{eq:s_and_t}) it is mapped to a non-linear 2-dimensional region $\tilde\calP$. Taking $\calP$ to be the rectangular hull of $\tilde\calP$, we have our new parameter region, see Figure~\ref{fig:parameter_space_zoom1}. The three vertices of $\calM$ are mapped into $\calP$ according to the following transformations:
\begin{eqnarray*}
(m_1, m_2) = (1/3,1/3) & \mapsto & (s,t) = (1/2, 2/3) \\
(m_1, m_2) = (0, 1/2) & \mapsto & (s,t) = (0, 1/2) \\
(m_1, m_2) = (0, 0) & \mapsto & (s,t) = ([0, 1/2], 0).
\end{eqnarray*}
Note how the single point $(m_1, m_2) = (0, 0)$ is mapped to the line segment $(s,t) = ([0, 1/2], 0)$ when taking all possible limits from within $\calM$. This desingularization is the main reason for moving to the $(s,t)$-space; it gives us a better control when masses are near the multi-restricted cases $m_1 = 0$ and $(m_1, m_2) = (0, 0)$.

\vspace*{3mm}
\begin{figure}[h]
\psfrag{Ptilde}{$\tilde\calP$}
\psfrag{st1}{$\hspace*{-1.2cm}(s,t) = (0, 0)$}
\psfrag{st2}{$\hspace*{-1.1cm}(s,t) = (\tfrac{1}{2}, 0)$}
\psfrag{st3}{$\hspace*{-1.1cm}(s,t) = (\tfrac{1}{2}, \tfrac{2}{3})$}
\psfrag{st4}{$\hspace*{-2.2cm}(s,t) = (0, \tfrac{1}{2})$}
\psfrag{st5}{$\hspace*{-0.9cm}(s^\star, t^\star)$}
\begin{center}
\includegraphics[width=0.25\linewidth]{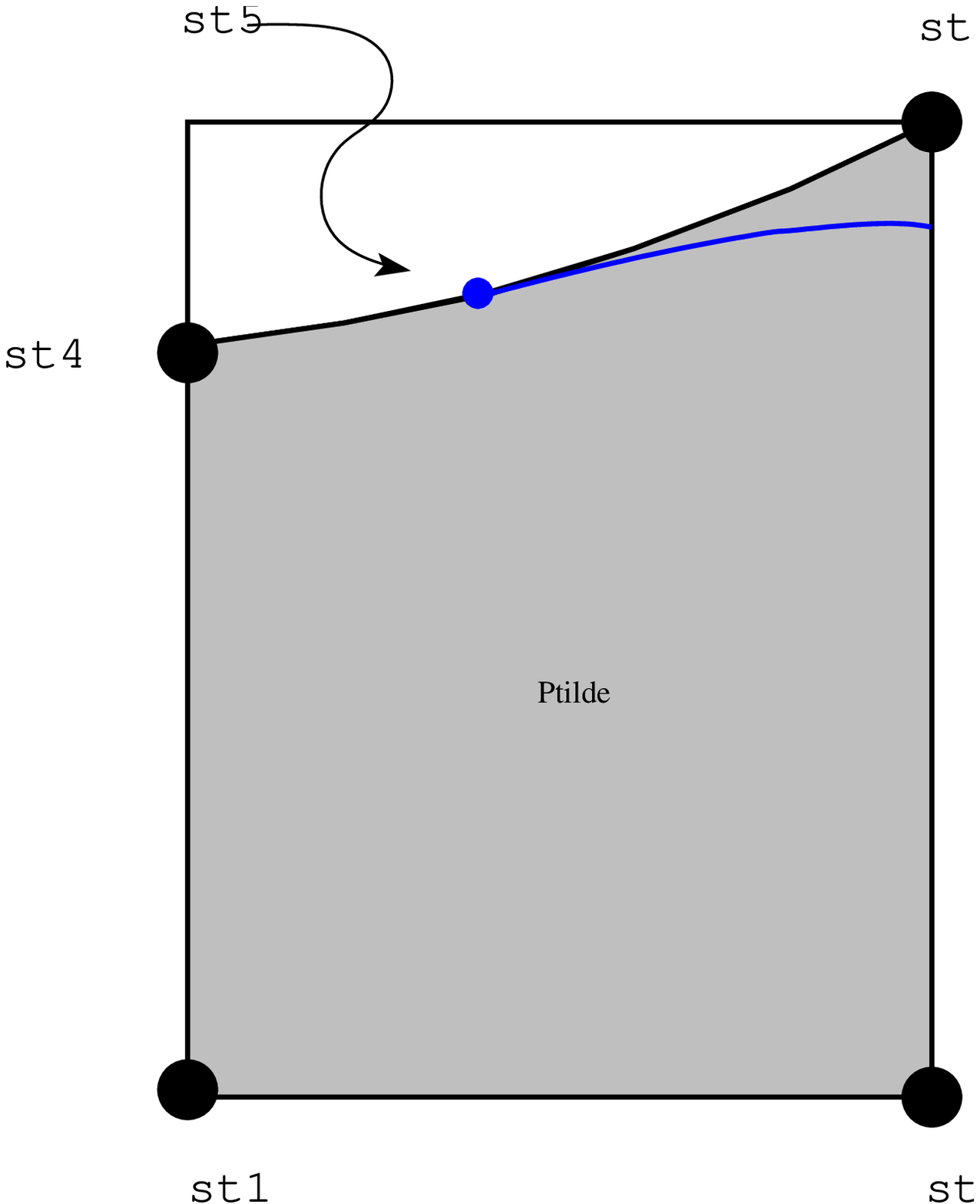}
\hspace*{25mm}
\psfrag{st1b}{}
\psfrag{st2b}{$\hspace*{0cm}\calP_1$}
\psfrag{st3b}{$\hspace*{0.0cm}\calP_3$}
\psfrag{st4b}{}
\psfrag{st5b}{$\hspace*{0.0cm}\calP_2$}
\includegraphics[width=0.25\linewidth]{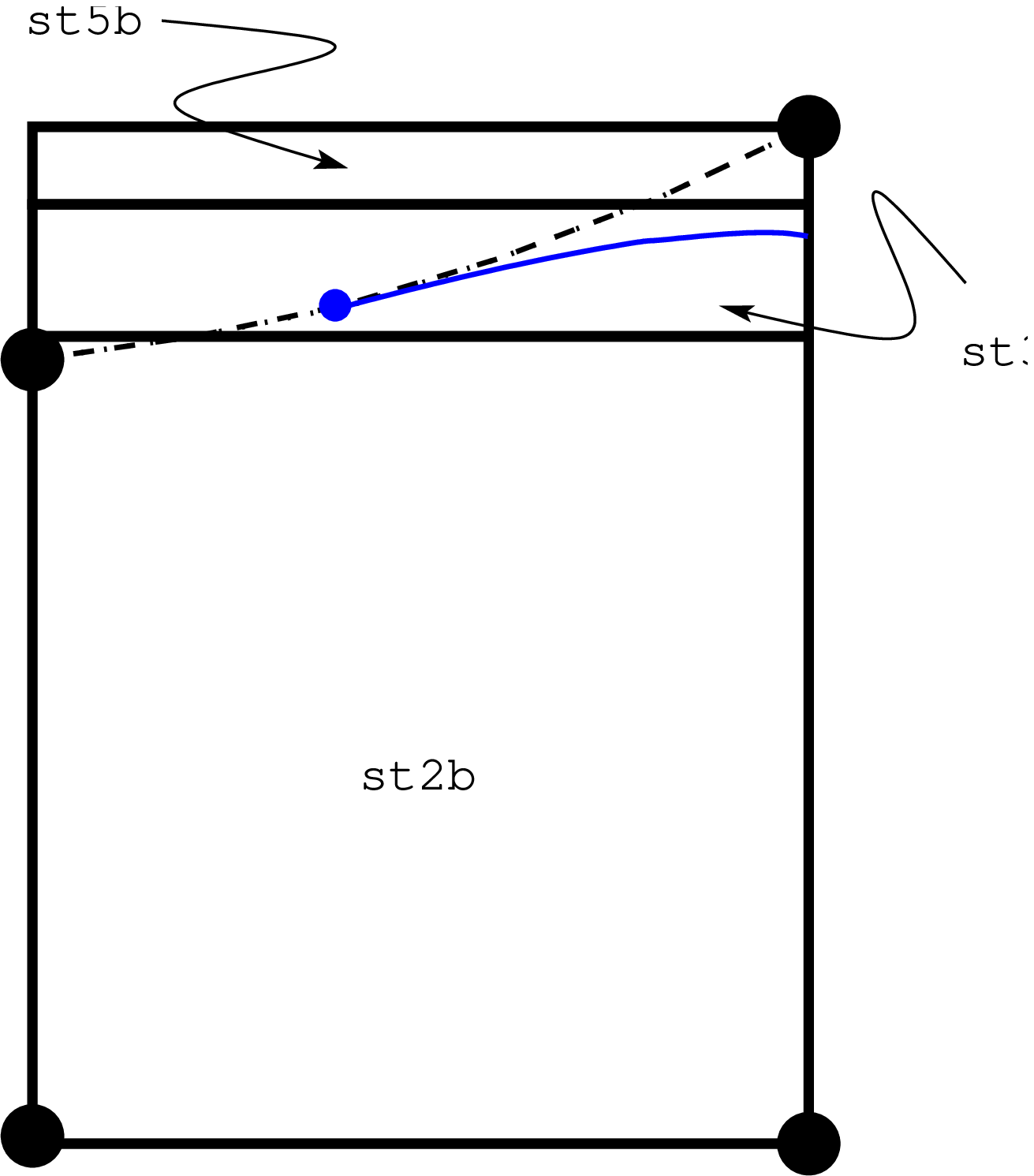}
\end{center}
\captionsetup{width=0.8\linewidth}
\caption{(a) The image $\tilde\calP$ (shaded) of the ordered mass space $\calM$ under the transformation (\ref{eq:s_and_t}). Compare to Figure~\ref{fig:mass_space}(b). The unshaded part of the rectangle $\calP$ corresponds to unordered masses. (b) The partition $\calP = \calP_1\cup\calP_2\cup\calP_3$ we will be using in this parametrization.}
\label{fig:parameter_space_zoom1}
\end{figure}

Based on this, we will define $\calP = \{(s,t)\colon 0\le s\le \frac{1}{2};  0\le t \le \tfrac{2}{3}\}$, and use the partition $\calP = \calP_1\cup\calP_2\cup\calP_3$ as illustrated in Figure~\ref{fig:parameter_space_zoom1}. More precisely we use
\begin{equation*}
\calP_1 = [0,0.5]\times [0,0.55],\quad \calP_2 = [0,0.5]\times [0.58, 0.67],\quad \calP_3 = [0,0.5]\times [0.55,0.58].
\end{equation*}
Note that each of the three partition elements contain points from $\calP\setminus\tilde\calP$. Such points correspond to unordered masses, and we will automatically remove most them from our computations. In the following, when we say that $\calP_i$ has some property, we mean that the ordered parameters $\tilde{\calP_i} = \calP_i\cap\tilde{\calP}$ have that property.

The three partition elements $\calP_1, \calP_2, \calP_3$ have the following properties:
\begin{itemize}
\item For each $(s,t)\in\calP_1$, there are exactly 8 solutions in $\calC$.
\item For each $(s,t)\in\calP_2$, there are exactly 10 solutions in $\calC$.
\item For each $(s,t)\in\calP_3$, there are between 8 and 10 solutions in $\calC$.
\end{itemize}

Let us describe these three regions in more detail. In subsequent sections, we will prove that these descriptions are accurate.

For $(s,t)\in\calP_2$ no bifurcations take place in $\calC$; exactly ten solutions exist but never come close to each other or a primary. This is the easiest region to account for. Also for $(s,t)\in\calP_1$ there are no bifurcations in $\calC$; exactly eight solutions exist. This region, however, includes parameters corresponding to arbitrarily small masses which presents other complications that must be resolved; more about this in section~\ref{subsubsec:small_masses}. The remaining set $\calP_3$ contains all parameters for which a bifucation occurs. As discussed in section~\ref{subsec:the_mass_space}, there are two bifurcation types that we must account for; quadratic and cubic. These bifurcations only take place in a small subset $\calC_2$ of the full configuration space. For $(s,t)\in\calP_3$, there can be 1, 2 or 3 solutions in $\calC_2$. In the remaining space $\calC_1 = \calC\setminus\calC_2$ there are exactly seven solutions, all isolated from each other and the primaries. Summing up, when $(s,t)\in\calP_3$ we have 8, 9 or 10 solutions in $\calC$ .

In the ideal setting, $\calP_3$ would correspond to the transformed (blue) bifurcation curve illustrated in Figure~\ref{fig:parameter_space_zoom1}. It would bisect $\tilde\calP$, acting as a common boundary line separating $\tilde\calP_1$ from $\tilde\calP_2$. Our approach, however, will build upon finite resolution computations, and therefore $\calP_2$ will be constructed as a rectangular subset of $\calP$, covering the entire bifurcation curve, see Figure~\ref{fig:parameter_space_zoom1}.

\section{Polar coordinates}\label{sec:polar_coordinates}

Given the shape of the configuration space (\ref{eq:configuration_space}), and how the solutions behave when masses become small, it makes sense to work in polar coordinates centered at the heaviest primary $p_3 = (0,0)$. In these coordinates the lighter primaries become $p_1=(1,\pi/6)$ and $p_2=(1,-\pi/6)$.

For convenience, let us define
\begin{equation*}
\alpha_1 = \pi/6, \qquad \alpha_2=-\pi/6.
\end{equation*}

Let $(z_1|z_2)$ denote the scalar product on $\mathbb{R}^2$. Given $z=(r\cos \phi,r \sin \phi)$, we have
\begin{eqnarray*}
  \|z- c\|^2 &=& (z-m_1 p_1 - m_2 p_2 | z-m_1 p_1 - m_2 p_2)=\\
  & &z^2 - 2m_1(p_1|z) - 2m_2(p_2|z) + 2m_1m_2(p_1|p_2) + p_1^2 + p_2^2, \\
  (p_i|z)&=& r \cos(\varphi-\alpha_i), \quad  i=1,2.
\end{eqnarray*}
It follows that
\begin{equation*}
   \|z- c\|^2= r^2 - 2 m_1 r \cos(\varphi-\pi/6) - 2m_2 r \cos(\varphi+ \pi/6) + g(m_1,m_2)
\end{equation*}
where $g(m_1,m_2)$ depends only on the masses and not on $z$. Therefore we can ignore $g$ when studying spatial derivatives of the potential $V$.

Note that, for $i=1,2$, we have
\begin{eqnarray*}
  r_i^2= (z-p_i)^2=z^2 - 2(p_i|z) + p_i^2=  r^2 - 2r \cos(\varphi-\alpha_i) + 1.
\end{eqnarray*}
Therefore, if we define
\begin{equation}
d(r,\alpha)= \left( r^2 - 2r \cos \alpha + 1 \right)^{1/2},
\end{equation}
then for $i=1,2$ we have
\begin{equation*}%\label{eq:r-rs}
 r_i(r,\varphi)=d(r,\varphi-\alpha_i).
\end{equation*}

Concluding, in polar coordinates, we obtain the new expression for the amended potential (compare to (\ref{eq:potential_function})):
\begin{equation}\label{eq:V-dec-W}
  V(r,\varphi;m)= V_0(r) +   m_1 W(r, \varphi-\alpha_1) + m_2 W(r, \varphi-\alpha_2) + g(m_1, m_2)
\end{equation}
where
\begin{equation*}
V_0(r) = \frac{r^2}{2} + \frac{m_3}{r}
\quad \mathrm{ and }\quad
W(r,\alpha) = \frac{1}{d(r,\alpha)} - r \cos(\alpha).
\end{equation*}

Taking partial derivatives, the gradient of the potential (\ref{eq:V-dec-W}) is given by
\begin{eqnarray}
\frac{\partial V}{\partial r}(r,\varphi;m)&=& r - \frac{m_3}{r^2} + m_1 \frac{\partial W}{\partial r}(r,\varphi-\alpha_1) + m_2 \frac{\partial W}{\partial r}(r,\varphi-\alpha_2),  \label{eq:dVdr} \\
\frac{\partial V}{\partial \varphi}(r,\varphi;m)&=&  m_1 \frac{\partial W}{\partial \alpha}(r,\varphi-\alpha_1) + m_2 \frac{\partial W}{\partial \alpha }(r,\varphi-\alpha_2),  \label{eq:dVdphi}
\end{eqnarray}
where
\begin{eqnarray*}
\frac{\partial W}{\partial \alpha}(r,\alpha)&=& r \sin \alpha \left(1 - \frac{1}{d(r,\alpha)^3} \right), \label{eq:dWda} \\
\frac{\partial W}{\partial r}(r,\alpha)&=& -\frac{r - \cos \alpha}{d(r,\alpha)^3} - \cos \alpha. \label{eq:dWdr}
\end{eqnarray*}
Note that (\ref{eq:dVdphi}) has a factor $r$ in both terms. We rescale this equation by a factor $1/r$ and $1/(m_1 + m_2)$. In the $(s,t)$ parameters this gives us
\begin{eqnarray}
F_1(r,\varphi;s,t)&=& r - \frac{1 - t}{r^2} + st \frac{\partial W}{\partial r}(r,\varphi-\alpha_1) + (1 - s)t \frac{\partial W}{\partial r}(r,\varphi-\alpha_2),  \label{eq:F1_of_st} \\
F_2(r,\varphi;s,t)&=&  s \frac{\partial W^\star}{\partial \alpha}(r,\varphi-\alpha_1) + (1-s) \frac{\partial W^\star}{\partial \alpha }(r,\varphi-\alpha_2),  \label{eq:F2_of_st}
\end{eqnarray}
where
\begin{equation*}
\frac{\partial W^\star}{\partial \alpha}(r,\alpha) = \frac{1}{r}\frac{\partial W}{\partial \alpha}(r,\alpha) = \sin \alpha \left(1 - \frac{1}{d(r,\alpha)^3} \right), \label{eq:dWstarda}
\end{equation*}
This $F(r,\varphi; s,t)$ will be used in place of $\nabla_z V(z;m)$ appearing in the previous sections. Zeros of $F$ correspond to critical points of $V$, and these correspond to relative equilibria.

\section{General strategy and key results}\label{sec:general_strategy_and_key_results}

Recall that we want to determine the number of solutions to $\nabla_z V(z;m) = 0$, and to understand how these behave within $\calC$ when the masses vary within $\calM$ (see Main Problem~\ref{pro:original}). For this to succeed, we must have a means of locating all solutions to the critical equation, and a way to analyze the various bifurcations that are possible as the masses vary.

As mentioned in section~\ref{sec:introduction} our overall goal is to construct a completely analytic proof of following theorem:
\begin{theorem}\label{thm:main_result}
For each $m\in\calM$ there are exactly 8,9 or 10 relative equilibria (i.e., solutions to the critical equation $\nabla_z V(z;m) = 0$) in $\calC$.
\end{theorem}
This result was originally established by Barros and Leandro \cite{BarrosLeandro14} using algebraic techniques. Developing analytic tools for the proof, we hope to apply these to harder instances of the $n$-body problem that are not within reach using algebraic methods.

In our new setting, using the $(s,t)$-coordinates together with polar coordinates (described in Sections~\ref{sec:reparametrizing_the_masses} and \ref{sec:polar_coordinates}), the critical equation is transformed into its equivalent form $F(r,\varphi; s,t) = 0$. As explained earlier, this form is better suited for our approach, where set-valued numerical computations will play a major role.

Before going into the details of the computations used as part of our computer-assisted framework, we present the following key results. We will use three different techniques: finding the exact number of solutions for given parameters $(s,t)$; proving that no bifurcations take place for a range of parameters; and controlling the number of solutions when when a bifurcation takes place.

We begin by determining the number of solutions for two different parameters.
\begin{theorem}
\label{thm:point_masses}
Consider the critical equation $F(r,\varphi;s,t) = 0$ for the $3+1$ body problem.
\begin{itemize}
\item[(a)] For parameters $(s,t) = (1/4, 1/4)$, there are exactly 8 solutions in $\calC$.
\item[(b)] For parameters $(s,t) = (9/20, 3/5)$, there are exactly 10 solutions in $\calC$.
\end{itemize}
\end{theorem}
Note that $(s,t) = (1/4, 1/4)\in\calP_1$ and $(s,t) = (9/20, 3/5)\in\calP_2$, as discussed in Section~\ref{sec:reparametrizing_the_masses}.

Combining the results of Theorem~\ref{thm:point_masses} with a criterion that ensures that no bifurcations are taking place (see Section~\ref{subsec:local_uniqueness}), we can extend the two solution counts to the two connected regions $\calP_1$ and $\calP_2$, respectively:

\begin{theorem}
\label{thm:mass_regions_one_and_two}
Consider the critical equation $F(r,\varphi;s,t) = 0$ for the $3+1$ body problem.
\begin{itemize}
\item[(a)] For all parameters $(s,t)\in\calP_1$, no solution in $\calC$ bifurcates.
\item[(b)] For all parameters $(s,t)\in\calP_2$, no solution in $\calC$ bifurcates.
\end{itemize}
\end{theorem}
It follows that the number of solutions in $\calP_1$ and $\calP_2$ is constant (8 and 10, respectively).

For the remaining part $\calP_3$ of the parameter space, we must be a bit more detailed. As discussed in Section~\ref{sec:computational_results}, we will split the configuration space into two connected components: $\calC = \calC_1 \cup \calC_2$, thus isolating the region where all bifurcations occur.

\begin{theorem}
\label{thm:configuration_region_1_and_2}
Consider the critical equation $F(r,\varphi;s,t) = 0$ for the $3+1$ body problem.
\begin{itemize}
\item[(a)] For all parameters $(s,t)\in\calP_3$, there are exactly 7 solutions in $\calC_1$.
\item[(b)] For all parameters $(s,t)\in\calP_3$, there are 1, 2, or 3 solutions in $\calC_2$.
\end{itemize}
\end{theorem}

Theorem~\ref{thm:main_result} now follows from the combination of Theorem~\ref{thm:point_masses}, \ref{thm:mass_regions_one_and_two}, and \ref{thm:configuration_region_1_and_2}.
In what follows, we will justify each of the three steps described above.

\section{Computational techniques}\label{sec:computational_techniques}

Here we present the computational techniques that we need to employ in order to establish our main theorem. We also discuss the underlying set-valued methods used later on in the computer-assisted proofs.

\subsection{Set-valued mathematics}\label{subsec:set_valued_mathematics}

We begin by giving a very brief introduction to set-valued mathematics and rigorous numerics. For more in-depth accounts we refer to e.g. \cite{Neumaier1990book, Tucker2011book}.

We will exclusively work with compact boxes $\ivX$ in $\R^n$, represented as vectors whose components are compact intervals: $\ivX = (\ivX_1,\dots,\ivX_n)$, where $\ivX_i = \{x\in\R\colon \lo x_i \le x \le \hi x_i\}$ for $i = 1,\dots, n$.

Given en explicit formula for a function $f\colon\R^n\to \R^m$, we can form its \textit{interval extension} (which we also denote by $f$), by extending each real operation by its interval counterpart. As long as the resulting \textit{interval image} $f(\ivX)$ is well-defined, we always have the following \textit{inclusion property}:
\begin{equation}\label{eq:inclusion_property}
\mathrm{range}(f;\ivX) = \{f(x)\colon x\in \ivX\} \subseteq f(\ivX).
\end{equation}
The main benefit of moving from real-valued to interval-valued analysis is the ability to discretize continuous problems while retaining full control of the discretisation errors. Indeed, whilst the exact range of a function $\mathrm{range}(f;\ivX)$ is hard to compute, its interval image $f(\ivX)$ can be found by a finite computation. In practice, the interval image is computed via a finite sequence of numerical operations. Carefully crafted libraries using directed rounding ensures that the numerical output respects the mathematical incusion propery of (\ref{eq:inclusion_property}).

\subsection{Equation solving and safe exclusions}\label{subsec:equation_solving_and_safe_exclusions}

We begin by stating what is known as the \textit{exclusion principle}:
\begin{theorem}\label{thm:exclusion_principle}
If f(\ivX) is well-defined and if $0\notin f(\ivX)$, then $f$ has no zero in $\ivX$.
\end{theorem}
The proof is an immediate consequence of (\ref{eq:inclusion_property}). The exclusion principle can be used in an adaptive bisection scheme, gradually discarding subsets of a global search space $\ivXX$. At each stage of the bisection process, a (possibly empty) collection of subsets $\ivX_1,\dots, \ivX_n$ remain, whose union must contain all zeros of $f$. Note, however, that this does not imply that $f(x) = 0$ has any solutions; we have only discarded subsets of $\ivXX$ where we are certain that no zeros of $f$ reside. In order to prove the existence of zeros, we need an addition result.

Let $f\in C^1(\ivXX, \R^n)$, where $\mathrm{Dom}(f) = \ivXX \subseteq \R^n$. Given an interval vector $\ivX\subset\ivXX$, a point $\check x\in\ivX$, and an invertible $n\times n$ matrix $C$, we define the \textit{Krawczyk operator}  \cite{Krawczyk,Neumaier1990book}   as
\begin{equation}
K_f(\ivX; \check x; C) = f(\check x) - C\cdot f(\check x) + \big(I - C\cdot Df(\ivX)\big)\cdot (\ivX - \check x).
\end{equation}
Popular choices are $\check x = \mathrm{mid}(\ivX)$ and $C = \mathrm{mid}([Df(\check x)])^{-1}$, resulting in Newton-like convergence rates near simple zeros.

\begin{theorem}\label{thm:krawczyk}
Assume that $K_f(\ivX; \check x; C)$ is well-defined. Then the following statements hold:
\begin{enumerate}
\item If $K_f(\ivX; \check x; C) \cap \ivX = \emptyset$, then $f$ has no zeros in $\ivX$.
\item If $K_f(\ivX; \check x; C) \subset \mathrm{int}\,\ivX$, then $f$ has a unique zero in $\ivX$.
\end{enumerate}
\end{theorem}
We will use this theorem together with the interval bisection scheme, were we adaptively bisect the initial search region $\ivXX$ into subsets $\ivX$ that are either discarded due to the fact that $0\notin f(\ivX)$, kept intact because of $K_f(\ivX; \check x; C) \subset \mathrm{int}\,\ivX$, or bisected for further study. On termination, this will give us an exact count on the number of zeros of $f$ within $\ivXX$.

Theorem~\ref{thm:krawczyk} can be extended to the setting where $f$ also depends on some $m$-dimensional parameter: $f\colon \R^n\times\R^m \to \R^n$ with $(x;p)\mapsto f(x;p)$. This is what we use to establish Theorem~\ref{thm:point_masses}.

\subsection{A set-valued criterion for local uniqueness}
\label{subsec:local_uniqueness}

Continuing in the set-valued, parameter dependent setting, we will explain in detail the criteria used for detecting when (and when not) a bifurcation occurs for a general system of non-linear equations, depending on some parameters. We will also derive results aimed at extracting more detailed information about some specific bifurcations that may occur.

Let us begin by considering the general problem of solving a system of (non-linear) equations
\begin{equation}\label{eq:equations-det}
f(x;p) = 0\qquad x\in\ivX,\quad p\in\ivP,
\end{equation}
where $\ivX\subset\R^n$ and $\ivP\subset\R^m$ are high-dimensional boxes. For a sufficiently smooth $f\colon\R^n\times\R^m\to\R^n$, we want to know how many solutions (\ref{eq:equations-det}) can have. We will focus on the bifurcations and develop a criterion which will tell us when the number of solutions of (\ref{eq:equations-det}) changes.

For now, we will suppress the parameter dependence of $f$ for clarity. All results that follow are extendable to the parameter-dependent setting.

An obvious condition for the local uniqueness of solutions to (\ref{eq:equations-det}) is given by the following theorem.
\begin{theorem}
\label{thm:interval-crit}
Let $f:\R^n\to \R^n$ be $C^1$. Assume that we are given a box $\ivX \subset \R^n$ such that $0 \notin \det Df(\ivX)$. Then $f$ is a bijection from $\ivX$ onto its image.
\end{theorem}

Note that $Df(\ivX)$ is a matrix with interval entries; it contains all possible Jacobian matrices $Df(x)$, where $x\in\ivX$.

\begin{proof}
We have for $x,y \in \ivX$, $x\neq y$
\begin{equation}\label{eq:bijection}
  f(x) - f(y) = \int_0^1 Df(y + t(x-y))dt \cdot (x-y) \in Df(\ivX)\cdot (x-y).
\end{equation}
Now, since $0 \notin \det Df(\ivX)$, all (point-valued) matrices in $Df(\ivX)$ are non-singular. Therefore the right-hand side of (\ref{eq:bijection}) cannot contain the zero vector. Hence the left-hand side $f(x) - f(y)$ cannot vanish.
\end{proof}

This result constitutes the core of most of our computations. Given two boxes $\ivX$ and $\ivP$, we can easily compute $\ivY = \det Df(\ivX;\ivP)$ using interval arithmetic. If $0\notin\ivY$, we know that $f(\cdot;p)$ is a bijection (between $\ivX$ and its image) for each $p\in\ivP$. Therefore, the equation (\ref{eq:equations-det}) can have \textit{at most} one solution in $\ivX$, and such a solution cannot undergo any bifurcation when $p\in\ivP$.

In dimension two, it is easy to illustrate the condition we are checking: the level sets of $f_1$ and $f_2$ must never become parallel, see Figure~\ref{fig:LSX}.

\begin{figure}[ht]
\begin{center}
\includegraphics[width=0.25\linewidth]{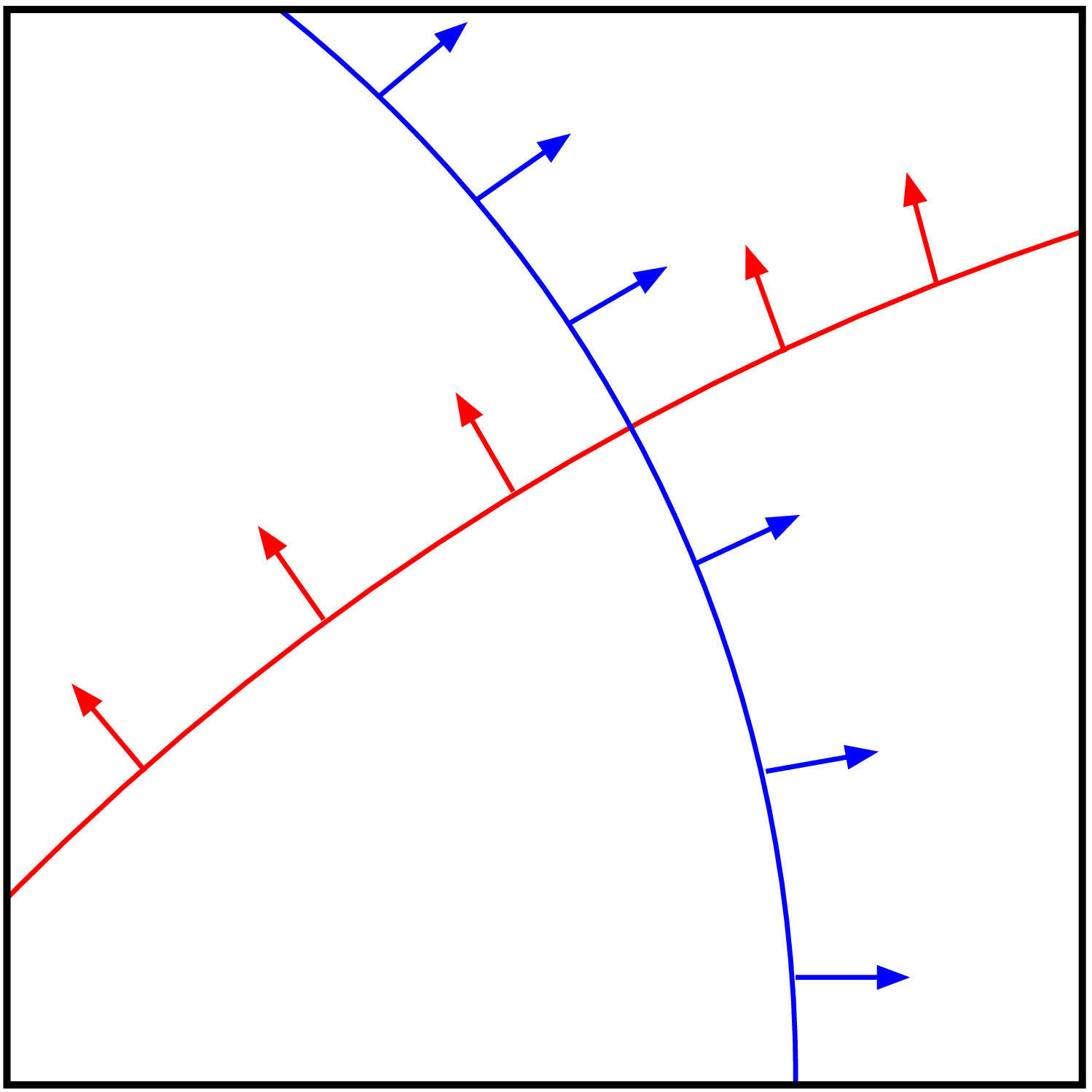}
\hspace*{3mm}
\includegraphics[width=0.25\linewidth]{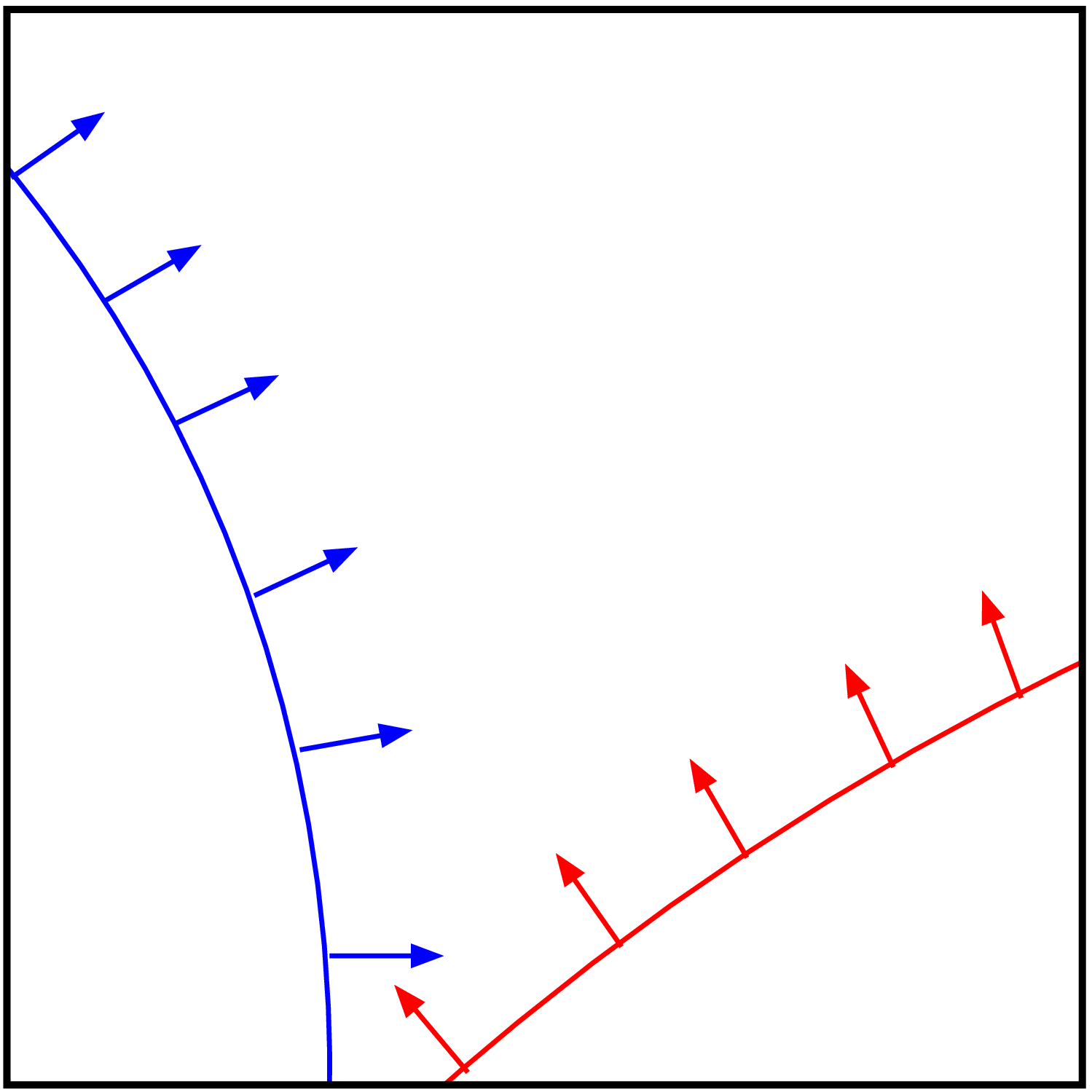}
\hspace*{3mm}
\includegraphics[width=0.25\linewidth]{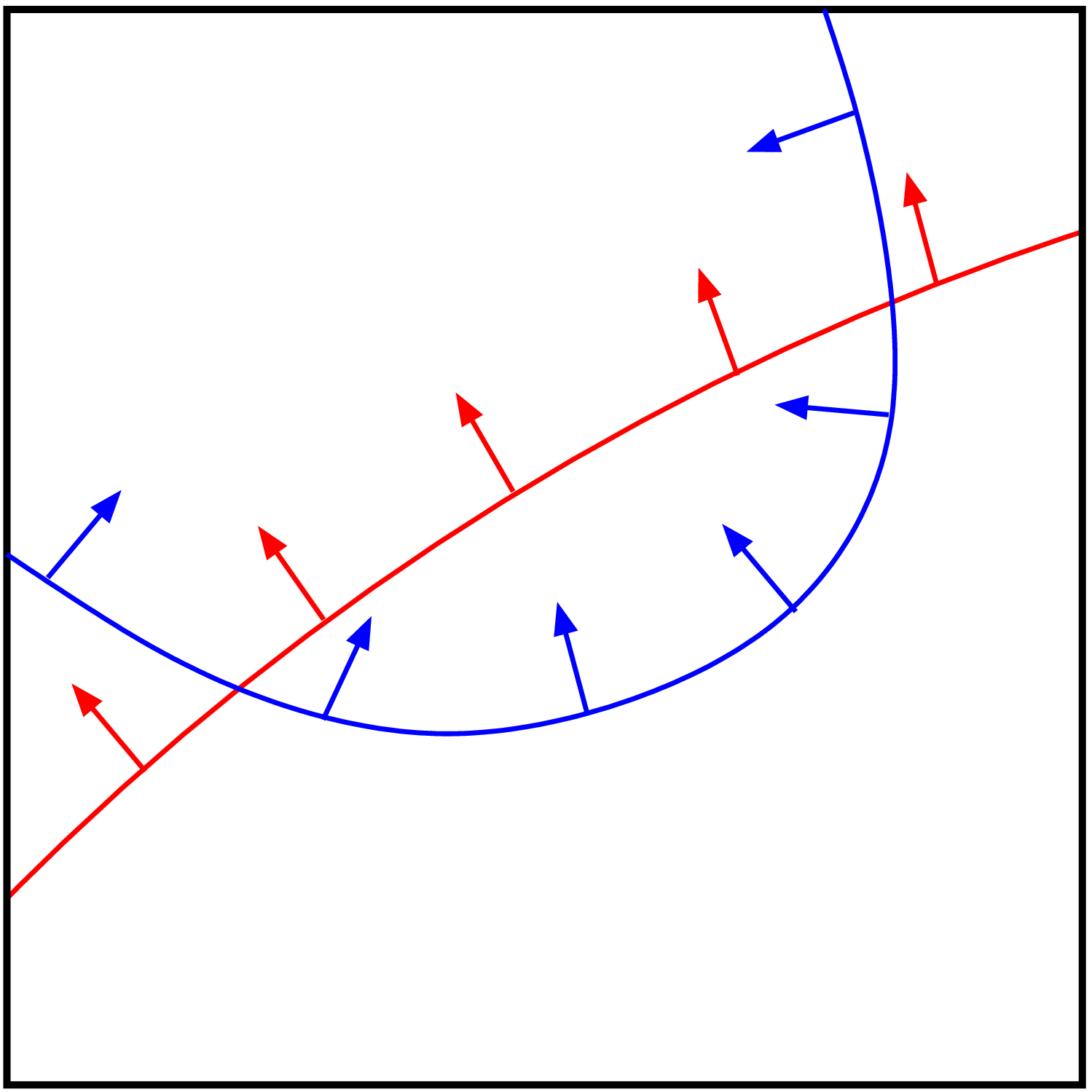}
\end{center}
\captionsetup{width=0.8\linewidth}
\caption{Some of the possible configurations of level sets of $f_1$ and $f_2$. The right-most configuration cannot happen under the assumptions of Theorem~\ref{thm:interval-crit}.}
\label{fig:LSX}
\end{figure}

\subsubsection{Small masses}
\label{subsubsec:small_masses}
In the case when one or two masses are very small, Theorem~\ref{thm:interval-crit} is not practically applicable. In this situation some solutions to the critical equation are close to the primaries, and a careful analysis combined with rigorous bounds produced with computer-assistance are needed.

The strategy to study this problem is the following: we focus on each primary with small mass, and use (local) polar coordinates $(\varrho,\varphi)$ around it to study the problem. In this setting we study the following system of equations (suppressing the dependence on masses) based on the original amended potential (\ref{eq:potential_function})
\begin{equation}\label{eq:system}
\left\{
\begin{matrix}
 \frac{\partial V}{\partial \varrho}(\varrho,\varphi)=0,\\
 \frac{\partial V}{\partial \varphi}(\varrho,\varphi)=0.
\end{matrix}
\right.
\end{equation}

It turns out that it is relatively easy to control the solutions of the equation $\frac{\partial V}{\partial \varphi}(\varrho,\varphi)=0$ for $\varphi$, obtaining four curves $(\varrho,\varphi(\varrho))$ on which we have uniform estimates over a whole range of $(m_1,m_2)$ including $(0,0)$. Turning to the remaining equation
\begin{equation*}
\frac{\partial V}{\partial \varrho}(\varrho,\varphi(\varrho))=0,
\end{equation*}
we study two cases separately: when $m_1$ alone tends to zero, and when $m_1$ and $m_2$ both tend to zero. In both cases we can prove that any solution of \eqref{eq:system} is of the form $(\varrho, \varphi(\varrho))$, and satisfies
\begin{equation*}
\frac{d}{d\varrho}\left(\frac{\partial V}{\partial \varrho}(\varrho,\varphi(\varrho))\right)\neq 0.
\end{equation*}
This implies that all solutions are regular: there are no bifurcations. We summarise our findings in a quantitative and practically applicable statement:

\begin{theorem}
\label{thm:small-masses-no-bif}
For $0 < m_i \leq 10^{-2}$ and $R=10^{-3}$, any relative equilibria $z$ of \eqref{eq:potential_function} with $\|z-p_i\|\leq R$, $i=1,2$, is non-degenerate: it does not bifurcate.
\end{theorem}

The details of all this can be found in the appendix \ref{app:small_masses}, with a thorough analysis of the solutions close to the light primaries in question, and with quantitative bounds that are both useful for proving the theorem, and of general interest for further studies.

This concludes what we use to establish Theorem~\ref{thm:mass_regions_one_and_two}.

\subsection{Lyapunov-Schmidt reduction in $\R^2$}
\label{subsec:LapSchmidt}

If we cannot invoke Theorem~\ref{thm:interval-crit} or \ref{thm:small-masses-no-bif} on a given region $\ivX\times\ivP$, we must explore the system of equations (\ref{eq:equations-det}) further.

Assume that the elimination of one variable (the Lyapunov-Schmidt
reduction) is possible in the box $\ivX$ for all $p \in \ivP$. The condition $0\notin \frac{\partial f_i}{\partial x_j}(\ivX;\ivP)$ for some $i,j\in\{1,\dots,n\}$ is sufficient for this to be possible. It ensures that $f_i$ is strictly monotone in the variable $x_j$ for all parameters $p\in\ivP$. This implies that a relation $f_i(x) = C$ implicitly defines $x_j$ in terms of the other independent variables: $x_j = x_j(x_1,\dots, x_{j-1}, x_{j+1},\dots, x_n)$. The domain of this parameterisation will depend on the constant $C$.

From now on we will work exclusively in the two-dimensional setting, and we will once again suppress the parameter dependency for clarity. As we are interested in solutions to (\ref{eq:equations-det}), we want to understand how the zero-level sets of $f_1$ and $f_2$ behave for various parameters. Without any loss of generality consider the case $(i,j)=(1,1)$, when we have
\begin{equation}\label{eq:df1dx1_is_nonzero}
  \frac{\partial f_1}{\partial x_1}(x) \neq 0, \quad x \in\ivX.
\end{equation}

Now assume also that the level set $f_1(x) = 0$ forms \textit{exactly} one connected component in $\ivX$. Then there exists a parametrisation $x_1 = x_1(x_2)$ defined on a connected domain $[x_2^-,x_2^+]\subset\ivX_2$, such that
\begin{equation}\label{eq:f1_is_zero}
 f_1(x_1,x_2)=0 \quad \mbox{if and only if} \quad x_1=x_1(x_2).
\end{equation}
We can now define the reduced form of (\ref{eq:equations-det}) as follows
\begin{equation}
g(x_2) = f_2(x_1(x_2),x_2) = 0, \qquad x_2 \in [x_2^-,x_2^+].  \label{eq:g-bif}
\end{equation}
Thus, the number of zeros of $g$ will determine the number of solutions to (\ref{eq:equations-det}).

For the remaining three cases, the analogous construction would produce
\begin{eqnarray*}
g(x_1) & = & f_2(x_1,x_2(x_1)) \qquad (i,j) = (1,2), \\
g(x_2) & = & f_1(x_1(x_2),x_2) \qquad (i,j) = (2,1), \\
g(x_1) & = & f_1(x_1,x_2(x_1)) \qquad (i,j) = (2,2).
\end{eqnarray*}

To simplify notation, we will use $y$ as the independent variable of $g$, and its domain will be denoted $\ivY = [y^-,y^+]$. Note that a sufficient condition for the uniqueness of a solution of $g(y)=0$, and hence also of (\ref{eq:equations-det}) is
\begin{equation}
 g'(y) \neq 0, \quad  y \in \ivY.  \label{eq:det-der-neq0}
\end{equation}

It turns out that (\ref{eq:det-der-neq0}) can be formulated in an invariant way. We will return to the case $(i,j) = (1,1)$ for the sake of clarity.

Implicit differentiation of (\ref{eq:f1_is_zero}) gives
\begin{equation*}
x_1'(x_2)= -\left(\frac{\partial f_1}{\partial x_1}(x_1(x_2),x_2)\right)^{-1} \frac{\partial f_1}{\partial x_2}(x_1(x_2),x_2).
\end{equation*}
Applying the chain rule to (\ref{eq:g-bif}) gives
\begin{equation*}
 g'(x_2)=  \frac{\partial f_2}{\partial x_1}(x_1(x_2),x_2) x_1'(x_2) +  \frac{\partial f_2}{\partial x_2}(x_1(x_2),x_2)
\end{equation*}
and substituting the expression for $x_1'(x_2)$ into this produces
\begin{eqnarray*}
g'(x_2) & = &
-\left(\frac{\partial f_1}{\partial x_1}\right)^{-1} \frac{\partial f_1}{\partial x_2}\frac{\partial f_2}{\partial x_1}+\frac{\partial f_2}{\partial x_2}  \\
& = &
\left(\frac{\partial f_1}{\partial x_1}\right)^{-1} \left( \frac{\partial f_1}{\partial x_1} \frac{\partial f_2}{\partial x_2} -\frac{\partial f_1}{\partial x_2}\frac{\partial f_2}{\partial x_1}\right)
=
\left(\frac{\partial f_1}{\partial x_1}\right)^{-1} \det(Df).
\end{eqnarray*}
The importance of the non-vanishing condition (\ref{eq:df1dx1_is_nonzero}) is now clear, as is the non-vanishing determinant condition of Theorem~\ref{thm:interval-crit}.

In what follows, we will refer to the constructed function $g$ as the \textit{bifurcation function}.

\subsection{Bifurcation analysis}\label{subsec:bifurcation_analysis}

Having seen how to apply the Lyapunov-Schmidt reduction, we now re-introduce the parameters to $f$ (and thus $g$). Given a bifurcation function $g:\ivY\times\ivP\to \mathbb{R}$, we would like to study the maximal number of solutions $g(y;p) = 0$ can have in $\ivY$ for $p \in \ivP$.

Instead of fully resolving the details of all possible bifurcations, we will use the following simple observation:
\begin{lemma}
\label{lem:gr}
Let $g$ be the bifurcation function as defined in Section~\ref{subsec:LapSchmidt}. Assume that for some $k \in \mathbb{Z}^+$ we have
\begin{equation}
  \frac{\partial^k g}{\partial y^k}(y;p) \neq 0 , \quad \forall (y,p) \in \ivY\times\ivP.
\end{equation}
Then for each $p \in \ivP$, the equation $g(y;p) = 0$ has at most $k$ solutions in $\ivY$.
\end{lemma}

The rightmost part of Figure~\ref{fig:LSX} illustrates the case $k=2$: the bifurcation function $g$ then has a quadratic behaviour.

Given a search region $\ivX\times\ivP$ for the original problem (\ref{eq:equations-det}), we can try to find a positive integer $k$ such that $0\notin g^{(k)}(\ivX_i;\ivP)$. Here $i$ can be any index for which the Lyapunov-Schmidt reduction works (note that then we have $\ivY\subset\ivX_i$). If we succeed, we have an upper bound on the number of solutions to $g(y;p) = 0$ in the region $\ivX_i\times\ivP\supset\ivY\times\ivP$. Note that this number translates to the original system of equations (\ref{eq:equations-det}). By construction, the solutions to $g(y;p) = 0$ are in one-to-one correspondence to those of $f(x;p) = 0$.

For the planar circular restricted 4-body problem it turns out that we only have to consider the cases $k=1,2,3$. The bifurcation function $g$ never behaves worse than a cubic function. The actual evaluation of the increasingly complicated expressions $g^{(k)}$ is achieved by \textit{automatic differentiation}. Furthermore, the Lyapunov-Schmidt reduction always succeeds in the case $(i,j) = (2,2)$, so we always have $g = g(x_1) = f_1(x_1, x_2(x_1))$.

We end this section by describing how we ensure that the level set $f_2(x) = 0$ forms exactly one connected component in $\ivX$. This is an important part of the Lyapunov-Schmidt reduction, and allows us to obtain an upper bound on the number of solutions via Lemma~\ref{lem:gr}. First we verify that $(i,j) = (2,2)$ are suitable indices by checking that $0 \notin \frac{\partial f_2}{\partial x_2}(\ivX)$. Note that this condition prevents a component of $f_2(x) = 0$ forming a closed loop inside $\ivX$; each component must enter and exit $\ivX$. Writing $\ivX = [x_1^-, x_1^+]\times [x_2^-, x_2^+]$, we next verify that $f_2(x_1^-, x_2^-) < 0 < f_2(x_1^+, x_2^+)$. This implies that $f_2$ must vanish at least twice on the boundary of $\ivX$. For each of the four sides $\ivS_i$ $(i = 1,\dots,4)$ of the rectangle $\ivX$, we compute an enclosure of the zero set $\ivZ_i = \{x\in\ivS_i\colon f_2(x) = 0\}$. On each such non-empty $\ivZ_i$, we check that $f_2$ is strictly increasing in its non-constant variable. This implies that $\cup_{i=1}^4 \ivZ_i$ must form two connected components $w_1$ and $w_2$. Each $w_i$ is made up of either one (non-empty) zero set $\ivZ_j$ or two such sets joined at a corner of $\ivX$. The level set $f_2(x) = 0$ crosses each of $w_1$ and $w_2$ transversally, and exactly once. Thus we have proved that $f_2(x) = 0$ forms exactly one connected component in $\ivX$.

\section{Computational results}\label{sec:computational_results}

We will now describe the program used for proving the results of Section~\ref{sec:general_strategy_and_key_results}. Throughout all computations, we use the parameters $(s,t)\in\calP = [0, \tfrac{1}{2}]\times [0, \tfrac{2}{3}]$ in place of the masses $(m_1, m_2, m_3)\in\calM$, as described in Section~\ref{sec:reparametrizing_the_masses}. We also represent the phase space variable $z\in\calC$ in polar coordinates $(r, \theta) \in [\tfrac{1}{3},2]\times [-\pi, \pi]$, and desingularize the equations as described in Section~\ref{sec:polar_coordinates}. This transforms the original critical equation $\nabla_z V(z;m) = 0$ into the equivalent problem $F(r,\varphi; s, t) = 0$, which is more suitable for computations.

The syntax of the program is rather straight-forward:
\begin{verbatim}
           ./pcr4bp tol minS maxS minT maxT strategy
\end{verbatim}

\begin{itemize}
\item {\cc tol} is the stopping tolerance used in the adaptive bisection process (used to discard subsets of the search region proved to have no zeros);
\item We consider parameters $s\in \ivS$ satisfying {\cc minS} $\le s\le$ {\cc maxS}.
\item We consider parameters $t\in \ivT$ satisfying {\cc minT} $\le t\le$ {\cc maxT}.
\item {\cc strategy} determines which method we use:
\begin{enumerate}
\item Explicitly count all solutions in $\calC$ -- used for $\ivS\times\ivT\subset\calP$ of very small width only.
\item Verify that there are no bifurcations taking place in $\calC$ for $(s,t)\in\ivS\times\ivT\subset\calP$.
\item Resolve all bifurcations taking place in $\calC$ for $(s,t)\in\ivS\times\ivT\subset\calP$.
\end{enumerate}
\end{itemize}
The three strategies are based on Theorem~\ref{thm:krawczyk}, Theorem~\ref{thm:interval-crit} together with Theorem~\ref{thm:small-masses-no-bif}, and Lemma~\ref{lem:gr}, respectively.

\subsection{Proof of Theorem~\ref{thm:point_masses}}

We demonstrate the program by proving Theorem~\ref{thm:point_masses}, running it on two different point-valued parameters (using strategy 1). The first one is for $(s,t) = (1/4, 1/4)$. This corresponds to the masses $m = (1/16, 3/16, 12/16)$, and produces 8 solutions.

\begin{verbatim}
./pcr4bp 1e-6 0.25 0.25 0.25 0.25 1

============================================================
========================== pcr4bp ==========================
Started computations: Fri Mar 18 17:57:37 2022

Stopping tolerance in the search: TOL = +1.000000e-06
Parameter range for s = [+2.500000e-01, +2.500000e-01]
Parameter range for t = [+2.500000e-01, +2.500000e-01]
  This corresponds to
  m1 = [+6.250000e-02, +6.250000e-02]
  m2 = [+1.875000e-01, +1.875000e-01]
  m3 = [+7.500000e-01, +7.500000e-01]
=========================================================
Using strategy #1: Trying to get an exact count valid for
the entire parameter region.
=========================================================
Searching the full configuration space:
  C = {[+3.333333e-01, +2.000000e+00],[-3.141593e+00, +3.141593e+00]}
After bisection/Krawczyk stage:
  |smallList|  : +0
  |noList|     : +185
  |yesList|    : +8
  |tightList|  : +8
  |ndtList|    : +4
=========================================================
SUCCESS: the explicit search was conclusive.
We found exactly +8 solutions.
=========================================================
Ended computations: Fri Mar 18 17:57:37 2022
=========================================================
\end{verbatim}
Here ${\cc smallList}$ is empty, signaling that the bisection stage was successful and had no unresolved subdomains. ${\cc noList}$ contains all boxes that were excluded using Theorem~\ref{thm:exclusion_principle}: there were 185 such boxes in this run. The eight elements of ${\cc yesList}$ are certified to contain unique zeros of $f$, according to Theorem~\ref{thm:krawczyk}. The same eight zeros are enclosed in the much smaller boxes stored in ${\cc tightList}$. Finally, the four boxes of $\cc ndtList$ are close to the lighter primaries, and are proven not to contain any zeros of $f$ according to the exclusion results of Section~\ref{subsec:the_configuration_space}. The output of this run is illustrated in Figure~\ref{fig:solutions}(a).

\begin{figure}[h]
\begin{center}
\includegraphics[width=0.45\linewidth]{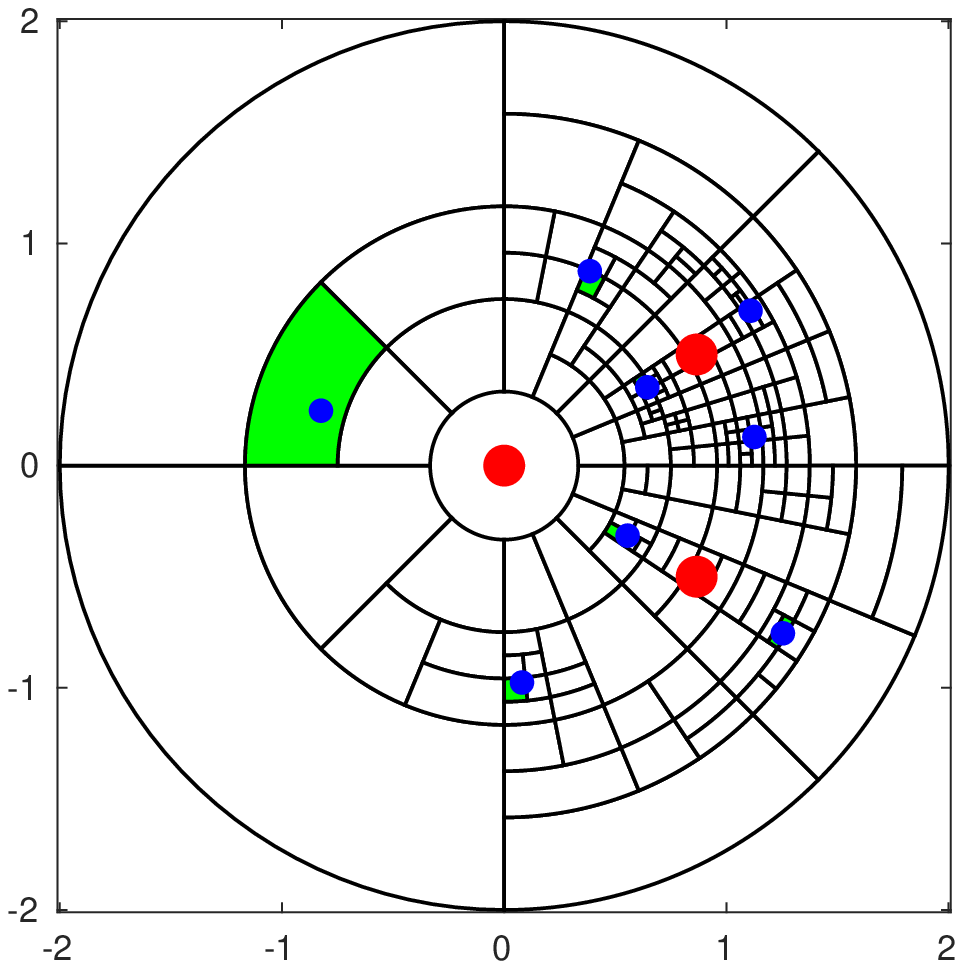}
\hspace{-0.05\linewidth}
\includegraphics[width=0.45\linewidth]{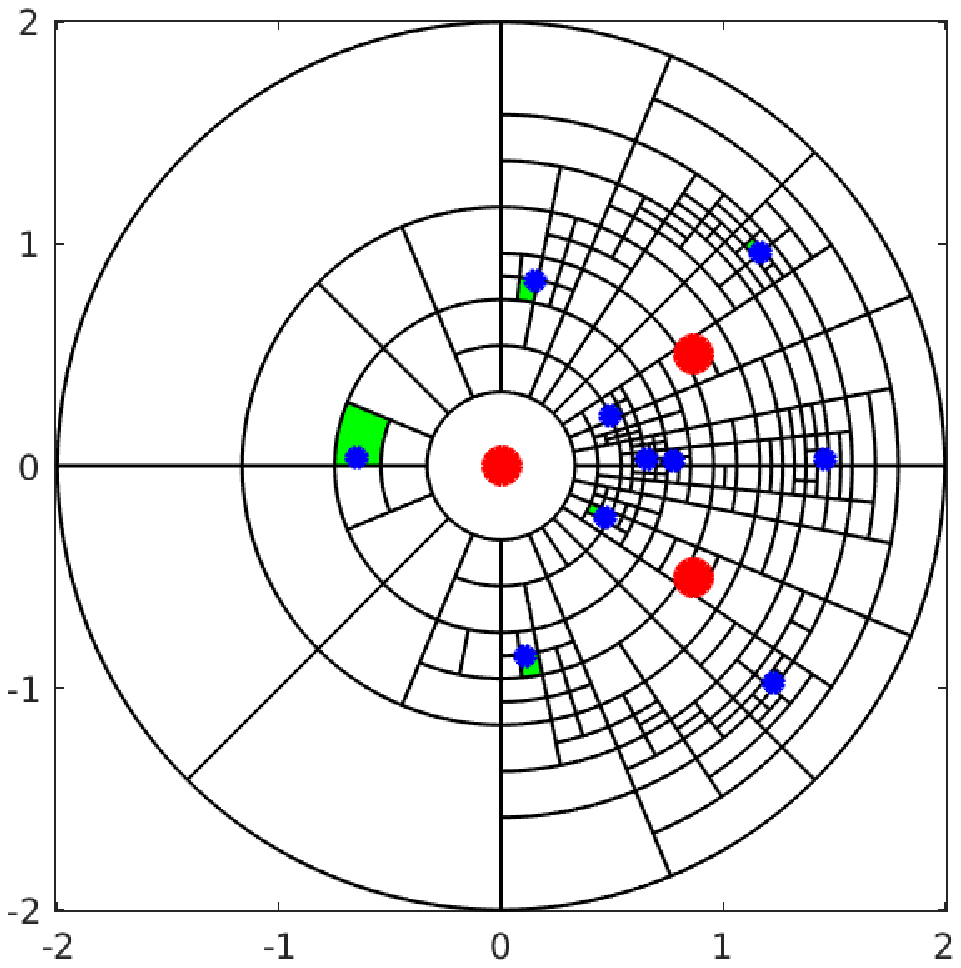}
\end{center}
\captionsetup{width=0.8\linewidth}
\caption{Relative equilibria (blue small disks) for the restricted $4$--body problem. The three primaries are red large disks. Green sectors contain unique solutions. White sectors are solution-free. (a) When masses differ substantially, there are eight relative equilibria. (b) Near equal masses yield ten relative equilibria.}
\label{fig:solutions}
\end{figure}

The second run is for $(s,t) = (9/20, 3/5)$. This corresponds to $m = (27/100, 33/100, 40/100)$, and produces 10 solutions. The output of this run is listed below, and is illustrated in Figure~\ref{fig:solutions}(b).

\begin{verbatim}
./pcr4bp 1e-6 0.45 0.45 0.60 0.60 1

============================================================
========================== pcr4bp ==========================
Started computations: Fri Mar 18 18:42:02 2022

Stopping tolerance in the search: TOL = +1.000000e-06
Parameter range for s = [+4.500000e-01, +4.500000e-01]
Parameter range for t = [+6.000000e-01, +6.000000e-01]
  This corresponds to
  m1 = [+2.700000e-01, +2.700000e-01]
  m2 = [+3.300000e-01, +3.300000e-01]
  m3 = [+4.000000e-01, +4.000000e-01]
=========================================================
Using strategy #1: Trying to get an exact count valid for
the entire parameter region.
=========================================================
Searching the full configuration space:
  C = {[+3.333333e-01, +2.000000e+00],[-3.141593e+00, +3.141593e+00]}
After bisection/Krawczyk stage:
  |smallList|  : +0
  |noList|     : +307
  |yesList|    : +10
  |tightList|  : +10
  |ndtList|    : +3
=========================================================
SUCCESS: the explicit search was conclusive.
We found exactly +10 solutions.
=========================================================
Ended computations: Fri Mar 18 18:42:02 2022
=========================================================
\end{verbatim}

These two runs complete the proof of Theorem~\ref{thm:point_masses}.

\subsection{Proof of Theorem~\ref{thm:mass_regions_one_and_two}}

Extending these results to larger domains in the $(s,t)$-parameter space, we turn to Theorem~\ref{thm:mass_regions_one_and_two}. In order to improve execution times, we pre-split the parameter domain into smaller subsets as illustrated in Figure~\ref{fig:parameter_space_splitting}. This particular splitting is rather ad-hoc, and is based on some heuristic trial runs; other splittings would work fine too.

\begin{figure}
\begin{center}
\vspace*{-5mm}
\psfrag{st2b}{$\hspace*{0cm}\calP_1$}
\psfrag{st3b}{$\hspace*{0.0cm}\calP_3$}
\psfrag{st5b}{$\hspace*{0.0cm}\calP_2$}
\includegraphics[width=0.29\linewidth]{./pictures/parameter_space_zoom1b.eps}
\hspace*{25mm}
\includegraphics[width=0.245\linewidth]{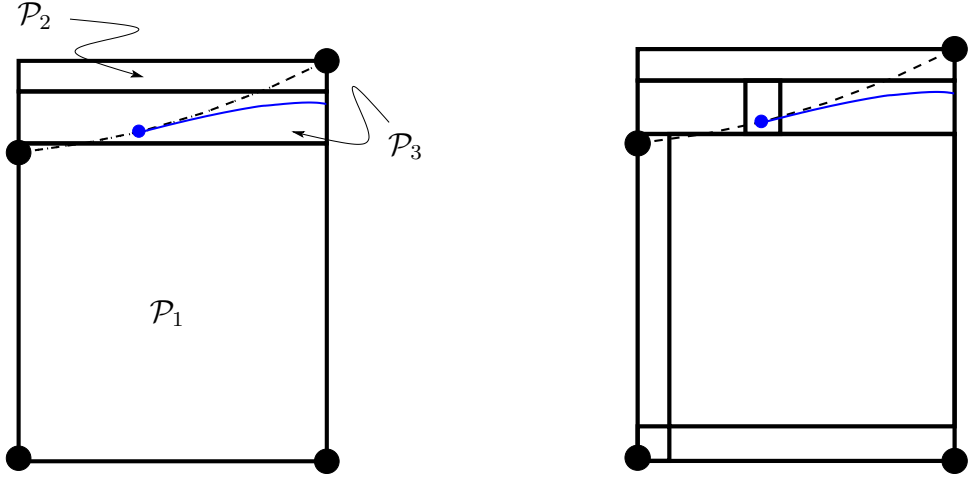}
\end{center}
\captionsetup{width=0.8\linewidth}
\caption{(a) The original partition $\calP = \calP_1\cup\calP_2\cup\calP_3$ used in the $(s,t)$-space. (b) The finer partition of $\calP$ used in our computations (not to scale). $\calP_3$ is pre-split into three subsets, and $\calP_1$ is pre-split into four. $\calP_2$ remains intact.}
\label{fig:parameter_space_splitting}
\end{figure}

We begin with the parameter set $\calP_1 = [0,0.5]\times [0,0.55]$ which is pre-split into four rectangles:
\begin{equation*}
\calP_1 = [0, 10^{-6}]\times [0, 10^{-6}] \cup [10^{-6},0.5]\times [0, 10^{-6}] \cup [0, 10^{-6}]\times [10^{-6}, 0.55] \cup [10^{-6},0.5]\times [10^{-6}, 0.55],
\end{equation*}
see Figure~\ref{fig:parameter_space_splitting}(b). Note that $\tilde{\calP_1} = \calP_1\cap\tilde\calP$ corresponds to the mass region above the bifurcation line in Figure~\ref{fig:mass_space}(b): this region contains all small masses $m_1$ and $m_2$. Examining each of the four rectangles separately, we now use {\cc strategy} number 2, which means that we are attempting to verify that no bifurcations take place for these parameter ranges. We begin with $(s,t)\in [0, 10^{-6}]\times [0, 10^{-6}]$:

\begin{verbatim}
$ ./pcr4bp 1e-6 0.00 1e-6 0.00 1e-6 2

============================================================
========================== pcr4bp ==========================
Started computations: Tue Mar 22 09:32:37 2022

Stopping tolerance in the search: TOL = +1.000000e-06
Parameter range for s = [+0.000000e+00, +1.000000e-06]
Parameter range for t = [+0.000000e+00, +1.000000e-06]
  This corresponds to
  m1 = [+0.000000e+00, +1.000000e-12]
  m2 = [+0.000000e+00, +1.000000e-06]
  m3 = [+9.999990e-01, +1.000000e+00]
=========================================================
Using strategy #2: Trying to prove that no bifurcations
take place for the entire parameter region.
=========================================================
Searching the full configuration space:
  C = {[+3.333333e-01, +2.000000e+00],[-3.141593e+00, +3.141593e+00]}
After bisection/determinant stage:
  unordered masses: +0
  |smallList|     : +0
  |noList|        : +1972
  |ndgList|       : +1637
  |ndtList|       : +99
=========================================================
SUCCESS: the implicit search was conclusive.
There are no bifurcations occurring for these parameters.
=========================================================
Ended computations: Tue Mar 22 09:32:37 2022
=========================================================
\end{verbatim}

We repeat the computation on the second rectangle $(s,t)\in [10^{-6},0.5]\times [0, 10^{-6}]$:

\begin{verbatim}
$ ./pcr4bp 1e-6 1e-6 0.50 0.00 1e-6 2

============================================================
========================== pcr4bp ==========================
Started computations: Tue Mar 22 09:32:47 2022

Stopping tolerance in the search: TOL = +1.000000e-06
Parameter range for s = [+1.000000e-06, +5.000000e-01]
Parameter range for t = [+0.000000e+00, +1.000000e-06]
  This corresponds to
  m1 = [+0.000000e+00, +5.000000e-07]
  m2 = [+0.000000e+00, +9.999990e-07]
  m3 = [+9.999990e-01, +1.000000e+00]
=========================================================
Using strategy #2: Trying to prove that no bifurcations
take place for the entire parameter region.
=========================================================
Searching the full configuration space:
  C = {[+3.333333e-01, +2.000000e+00],[-3.141593e+00, +3.141593e+00]}
After bisection/determinant stage:
  unordered masses: +0
  |smallList|     : +0
  |noList|        : +1453289
  |ndgList|       : +1268457
  |ndtList|       : +99963
=========================================================
SUCCESS: the implicit search was conclusive.
There are no bifurcations occurring for these parameters.
=========================================================
Ended computations: Tue Mar 22 09:34:13 2022
=========================================================
\end{verbatim}

We repeat the computation on the third rectangle $(s,t)\in [0, 10^{-6}]\times [10^{-6}, 0.55]$:

\begin{verbatim}
$ ./pcr4bp 1e-6 0.00 1e-6 1e-6 0.55 2

============================================================
========================== pcr4bp ==========================
Started computations: Tue Mar 22 09:34:42 2022

Stopping tolerance in the search: TOL = +1.000000e-06
Parameter range for s = [+0.000000e+00, +1.000000e-06]
Parameter range for t = [+1.000000e-06, +5.500000e-01]
  This corresponds to
  m1 = [+0.000000e+00, +5.500000e-07]
  m2 = [+9.999990e-07, +5.500000e-01]
  m3 = [+4.500000e-01, +9.999990e-01]
=========================================================
Using strategy #2: Trying to prove that no bifurcations
take place for the entire parameter region.
=========================================================
Searching the full configuration space:
  C = {[+3.333333e-01, +2.000000e+00],[-3.141593e+00, +3.141593e+00]}
After bisection/determinant stage:
  unordered masses: +954
  |smallList|     : +0
  |noList|        : +13330641
  |ndgList|       : +4592794
  |ndtList|       : +336900
NOTE: we encountered regions in parameter space that
correspond to unordered masses. These were not analyzed.
=========================================================
SUCCESS: the implicit search was conclusive.
There are no bifurcations occurring for these parameters.
=========================================================
Ended computations: Tue Mar 22 09:42:22 2022
=========================================================
\end{verbatim}

Finally, we repeat the computation on the fourth rectangle $(s,t)\in [10^{-6},0.5]\times [10^{-6}, 0.55]$:

\begin{verbatim}
$ ./pcr4bp 1e-6 1e-6 0.50 1e-6 0.55 2

============================================================
========================== pcr4bp ==========================
Started computations: Tue Mar 22 09:30:32 2022

Stopping tolerance in the search: TOL = +1.000000e-06
Parameter range for s = [+1.000000e-06, +5.000000e-01]
Parameter range for t = [+1.000000e-06, +5.500000e-01]
  This corresponds to
  m1 = [+1.000000e-12, +2.750000e-01]
  m2 = [+5.000000e-07, +5.499995e-01]
  m3 = [+4.500000e-01, +9.999990e-01]
=========================================================
Using strategy #2: Trying to prove that no bifurcations
take place for the entire parameter region.
=========================================================
Searching the full configuration space:
  C = {[+3.333333e-01, +2.000000e+00],[-3.141593e+00, +3.141593e+00]}
After bisection/determinant stage:
  unordered masses: +9125
  |smallList|     : +0
  |noList|        : +2294008
  |ndgList|       : +1018710
  |ndtList|       : +160852
NOTE: we encountered regions in parameter space that
correspond to unordered masses. These were not analyzed.
=========================================================
SUCCESS: the implicit search was conclusive.
There are no bifurcations occurring for these parameters.
=========================================================
Ended computations: Tue Mar 22 09:32:10 2022
=========================================================
\end{verbatim}

These four runs prove that $\calP_1$ is bifurcation free, which establishes part (a) of Theorem~\ref{thm:mass_regions_one_and_two}. Here we also get a report of several encountered parameter regions that are unordered; they are a consequence of the reparametrization of the mass space, and belong to the set $\calP \setminus\tilde\calP$, see Figure~\ref{fig:parameter_space_zoom1}. Such parameters are automatically discarded. Furthermore, {\cc ndgList} contains all boxes where we have established that no bifurcations are taking place. As we are encountering parameters corresponding to both $m_1$ and $m_2$ becoming (arbitrarily) small, there is a significant partitioning near the lighter primaries, resulting in many boxes in {\cc ndtList}.

Turning to part (b) of Theorem~\ref{thm:mass_regions_one_and_two}, we continue with the parameter set $\calP_2 = [0,0.5]\times [0.58, 0.67]$. Note that $\calP_2\cap\tilde\calP$ corresponds to the mass region below the bifurcation line in Figure~\ref{fig:mass_space}(b). This region contains the point of equal masses. Again we use {\cc strategy} number 2.

\begin{verbatim}
$ ./pcr4bp 1e-6 0.00 0.50 0.58 0.67 2

============================================================
========================== pcr4bp ==========================
Started computations: Tue Mar 22 09:30:08 2022

Stopping tolerance in the search: TOL = +1.000000e-06
Parameter range for s = [+0.000000e+00, +5.000000e-01]
Parameter range for t = [+5.800000e-01, +6.700000e-01]
  This corresponds to
  m1 = [+0.000000e+00, +3.350000e-01]
  m2 = [+2.900000e-01, +6.700000e-01]
  m3 = [+3.300000e-01, +4.200000e-01]
=========================================================
Using strategy #2: Trying to prove that no bifurcations
take place for the entire parameter region.
=========================================================
Searching the full configuration space:
  C = {[+3.333333e-01, +2.000000e+00],[-3.141593e+00, +3.141593e+00]}
After bisection/determinant stage:
  unordered masses: +3324
  |smallList|     : +0
  |noList|        : +50435
  |ndgList|       : +13145
  |ndtList|       : +13
NOTE: we encountered regions in parameter space that
correspond to unordered masses. These were not analyzed.
=========================================================
SUCCESS: the implicit search was conclusive.
There are no bifurcations occurring for these parameters.
=========================================================
Ended computations: Tue Mar 22 09:30:10 2022
=========================================================
\end{verbatim}

This run proves that $\calP_2$ is bifurcation free, as claimed. Note that the computational effort was smaller than that for $\calP_1$. This is due to the fact that we are not considering any parameters corresponding to small masses in this run. We have now completed the proof of Theorem~\ref{thm:mass_regions_one_and_two}.

\subsection{Proof of Theorem~\ref{thm:configuration_region_1_and_2}}

Turning to Theorem~\ref{thm:configuration_region_1_and_2}, we focus on the only remaining parameter set $\calP_3 = [0,0.5]\times [0.55, 0.58]$, which has been constructed to contain the entire line of bifurcation, see Figure~\ref{fig:parameter_space_zoom1}(b). We now use {\cc strategy} number 3, which means that we will locate and resolve all occurring bifurcations.

It turns out that all bifurcations take place within a small subset $\calC_2$ of the configuration space. In the complement $\calC_1 = \calC\setminus\calC_2$, the solutions are non-degenerate and persist for all parameter values in $\calP_3$. In light of this, the program splits the configuration space into two parts: $\calC = \calC_1\cup \calC_2$ which are examined separately. For purely technical reasons $\calC_1$ is further divided into three subregions (named {\cc C11, C12, C13} in the output).

As explained above, $\calP_3$ is pre-split into three rectangles:
\begin{equation*}
\calP_3 = [0, 0.2]\times [0.55, 0.58] \cup [0.2,0.25]\times [0.55, 0.58] \cup [0.25, 0.5]\times [0.55, 0.58],
\end{equation*}
see Figure~\ref{fig:parameter_space_splitting}(b).

We begin with $(s,t)\in [0, 0.2]\times [0.55, 0.58]$:
\begin{verbatim}
$ ./pcr4bp 1e-6 0.00 0.20 0.55 0.58 3

============================================================
========================== pcr4bp ==========================
Started computations: Tue Mar 22 09:45:19 2022

Stopping tolerance in the search: TOL = +1.000000e-06
Parameter range for s = [+0.000000e+00, +2.000000e-01]
Parameter range for t = [+5.500000e-01, +5.800000e-01]
  This corresponds to
  m1 = [+0.000000e+00, +1.160000e-01]
  m2 = [+4.400000e-01, +5.800000e-01]
  m3 = [+4.200000e-01, +4.500000e-01]
=========================================================
Using strategy #3: Doing a full bifurcation analysis for
the entire parameter region.
=========================================================
Searching the following three outer regions of configuration space:
  C1 = {[+3.333333e-01, +2.000000e+00],[+7.000000e-01, +3.141593e+00]}
  C2 = {[+3.333333e-01, +2.000000e+00],[-3.141593e+00, -2.000000e-01]}
  C3 = {[+1.000000e+00, +2.000000e+00],[-2.000000e-01, +7.000000e-01]}
After bisection/determinant stage:
  unordered masses: +160
  |smallList|     : +0
  |noList|        : +368
  |ndgList|       : +95
  |ndtList|       : +3
NOTE: we encountered regions in parameter space that
correspond to unordered masses. These were not analyzed.
=========================================================
SUCCESS: the outer bifurcation analysis was conclusive.
There are no bifurcations occurring for these parameters.
=========================================================
Searching the single, inner region of configuration space:
  C0 = {[+3.333333e-01, +1.000000e+00],[-2.000000e-01, +7.000000e-01]}
Splitting tolerances for parameters: [+1.000000e-08, +5.000000e-02].
After the bisection/bifurcation stage:
  |smallList|: +0
  |s0List|   : +11
  |s1List|   : +12
  |s2List|   : +0
  |s3List|   : +0
  |s111List| : +0
  |s210List| : +0
  |s300List| : +0
There are subregions of C0 with at most one solution.
NOTE: we encountered regions in parameter space that
correspond to unordered masses. These were not analyzed.
=========================================================
SUCCESS: the inner bifurcation analysis was conclusive.
=========================================================
Ended computations: Tue Mar 22 09:45:20 2022
=========================================================
\end{verbatim}

This run is not so interesting: there are no bifurcations taking place in the parameter region. Indeed, most of these parameters are unordered, and are discarded immediately. Apart from these, there are some parameters that are provably bifurcation-free (we are actually using the techniques of {\cc strategy} 2 as part of the computation). No higher order bifurcations take place at all.

The next two parameter regions are much more challenging. We continue with $(s,t)\in [0.2, 0.25]\times [0.55, 0.58]$:
\begin{verbatim}
./pcr4bp 1e-6 0.25 0.50 0.55 0.58 3

============================================================
========================== pcr4bp ==========================
Started computations: Tue Mar 22 09:46:11 2022

Stopping tolerance in the search: TOL = +1.000000e-06
Parameter range for s = [+2.500000e-01, +5.000000e-01]
Parameter range for t = [+5.500000e-01, +5.800000e-01]
  This corresponds to
  m1 = [+1.375000e-01, +2.900000e-01]
  m2 = [+2.750000e-01, +4.350000e-01]
  m3 = [+4.200000e-01, +4.500000e-01]
=========================================================
Using strategy #3: Doing a full bifurcation analysis for
the entire parameter region.
=========================================================
Searching the following three outer regions of configuration space:
  C1 = {[+3.333333e-01, +2.000000e+00],[+7.000000e-01, +3.141593e+00]}
  C2 = {[+3.333333e-01, +2.000000e+00],[-3.141593e+00, -2.000000e-01]}
  C3 = {[+1.000000e+00, +2.000000e+00],[-2.000000e-01, +7.000000e-01]}
After bisection/determinant stage:
  unordered masses: +0
  |smallList|     : +0
  |noList|        : +1922
  |ndgList|       : +526
  |ndtList|       : +10
=========================================================
SUCCESS: the outer bifurcation analysis was conclusive.
There are no bifurcations occurring for these parameters.
=========================================================
Searching the single, inner region of configuration space:
  C0 = {[+3.333333e-01, +1.000000e+00],[-2.000000e-01, +7.000000e-01]}
Splitting tolerances for parameters: [+1.000000e-08, +5.000000e-02].
After the bisection/bifurcation stage:
  |smallList|: +0
  |s0List|   : +6
  |s1List|   : +149
  |s2List|   : +0
  |s3List|   : +253
  |s111List| : +103
  |s210List| : +150
  |s300List| : +0
There are subregions of C0 with at most three solutions.
There are subregions of C0 with at most one solution.
NOTE: we encountered regions in parameter space that
correspond to unordered masses. These were not analyzed.
=========================================================
SUCCESS: the inner bifurcation analysis was conclusive.
=========================================================
Ended computations: Tue Mar 22 09:46:53 2022
=========================================================
\end{verbatim}

Let us now explain what we can learn from this run.

The first part of the run verifies that no bifurcations take place within $\calC_1$ for parameters in $\calP_3$: we use {\cc strategy} 2 to establish this fact. Figure~\ref{fig:run4_outer_bw} (a) illustrates the outcome of this part of the computations.

\begin{figure}[h]
\begin{center}
\includegraphics[width=0.48\linewidth]{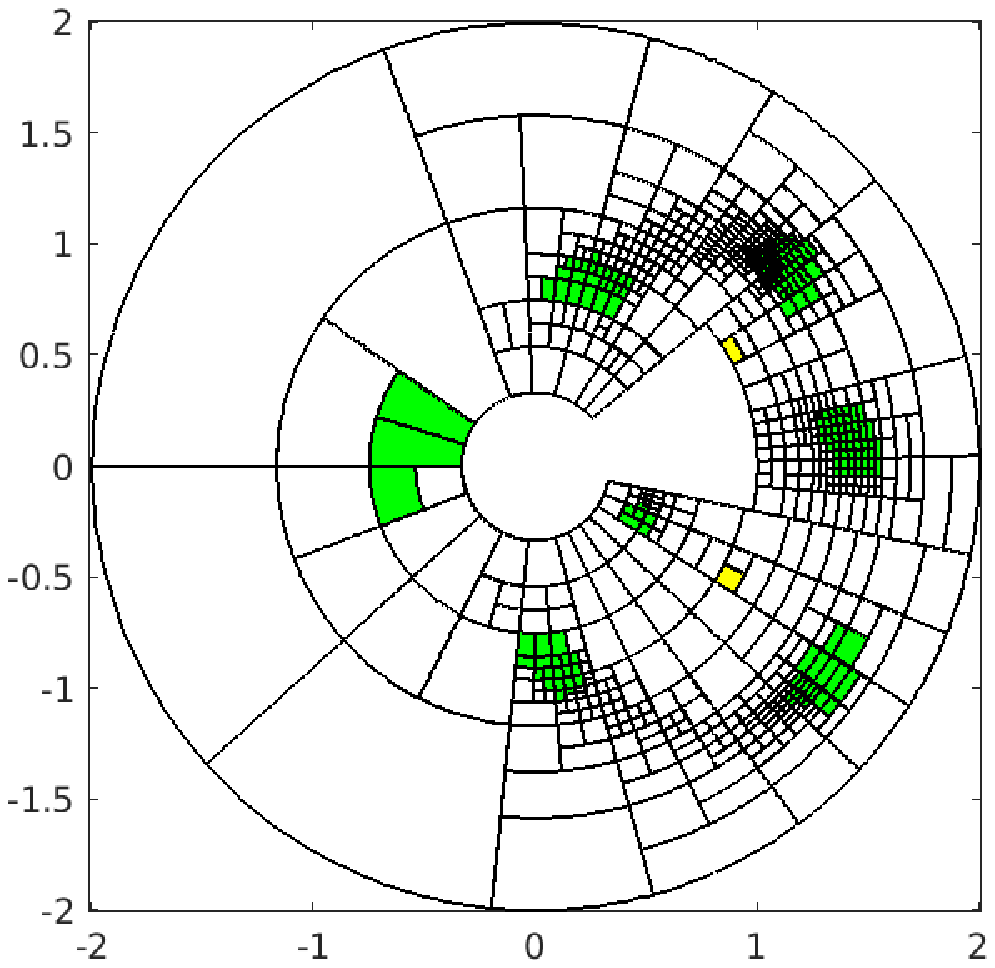}
\hspace*{0mm}
\includegraphics[width=0.48\linewidth]{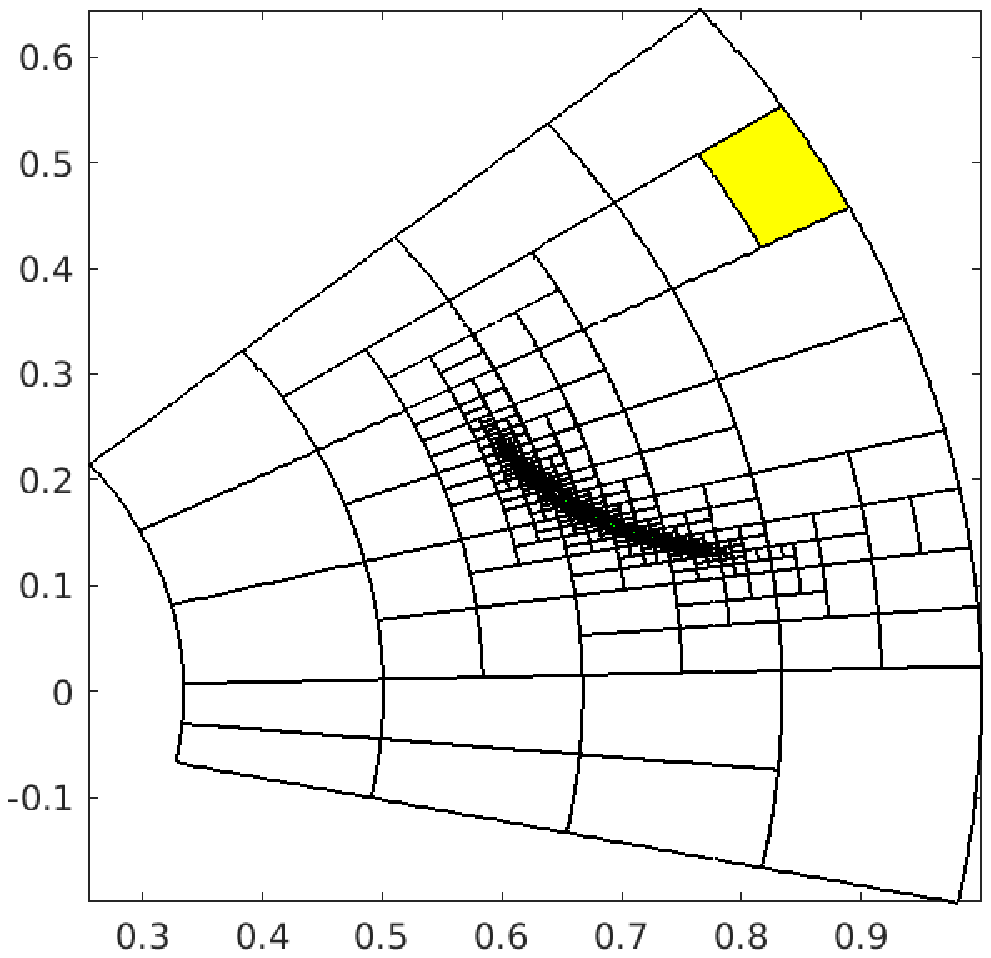}
\end{center}
\captionsetup{width=0.8\linewidth}
\caption{(a) There are no bifurcations taking place inside $\calC_1$. For all parameters in $\calP_3$, there exist seven solutions here, each located in one of the green connected regions. (b) Inside $\calC_2$ two kinds of bifurcations occur: quadratic and cubic. As before, some exclusion theorems are valid in the yellow sections.}
\label{fig:run4_outer_bw}
\end{figure}

The second part of the run focuses on the remaining portion of the configuration space $\calC_2$ (named ${\cc C0}$ in the output); in polar coordinates $\calC_2 = [\tfrac{1}{3},1]\times [\tfrac{2}{10}, \tfrac{7}{10}]$. Throughout these computations, the parameter region $\calP_3$ is adaptively bisected into many smaller sets; each one being examined via various test. When successfully classified, a parameter set is stored in one of several lists. Together with each such parameter set, we also store a covering of all possible solutions in $\calC_2$. The lists are organized as follows:
\begin{itemize}
\item[1.] {\cc s0List}: the parameters are unordered (we do not consider these).
\item[2.] {\cc s1List}: the number of solutions is at most one and fixed (no bifurcations) in $\calC_2$.
\item[3.] {\cc s111List}: there are three connected components in $\calC_2$; in each one of them the number of solutions is at most one and fixed (no bifurcations).
\item[4.] {\cc s210List}: there are two connected components in $\calC_2$ where solutions may reside. For one of these the number of solutions is at most one and fixed (no bifurcations). For the second component there can be 0, 1, or 2 solutions (a quadratic bifurcation).
\item[5.] {\cc s300List}: there is one connected component in $\calC_2$ where the number of solutions can be 1, 2 or 3 (a cubic bifurcation).
\end{itemize}
We also mention that {\cc s3List} is the union of {\cc s111List}, {\cc s210List}, and {\cc s300List}.

Our final computation deals with the region $(s,t)\in [0.25, 0.5]\times [0.55, 0.58]$:
\begin{verbatim}
$ ./pcr4bp 1e-6 0.25 0.50 0.55 0.58 3

============================================================
========================== pcr4bp ==========================
Started computations: Tue Mar 22 09:46:11 2022

Stopping tolerance in the search: TOL = +1.000000e-06
Parameter range for s = [+2.500000e-01, +5.000000e-01]
Parameter range for t = [+5.500000e-01, +5.800000e-01]
  This corresponds to
  m1 = [+1.375000e-01, +2.900000e-01]
  m2 = [+2.750000e-01, +4.350000e-01]
  m3 = [+4.200000e-01, +4.500000e-01]
=========================================================
Using strategy #3: Doing a full bifurcation analysis for
the entire parameter region.
=========================================================
Searching the following three outer regions of configuration space:
  C1 = {[+3.333333e-01, +2.000000e+00],[+7.000000e-01, +3.141593e+00]}
  C2 = {[+3.333333e-01, +2.000000e+00],[-3.141593e+00, -2.000000e-01]}
  C3 = {[+1.000000e+00, +2.000000e+00],[-2.000000e-01, +7.000000e-01]}
After bisection/determinant stage:
  unordered masses: +0
  |smallList|     : +0
  |noList|        : +1922
  |ndgList|       : +526
  |ndtList|       : +10
=========================================================
SUCCESS: the outer bifurcation analysis was conclusive.
There are no bifurcations occurring for these parameters.
=========================================================
Searching the single, inner region of configuration space:
  C0 = {[+3.333333e-01, +1.000000e+00],[-2.000000e-01, +7.000000e-01]}
Splitting tolerances for parameters: [+1.000000e-08, +5.000000e-02].
After the bisection/bifurcation stage:
  |smallList|: +0
  |s0List|   : +6
  |s1List|   : +149
  |s2List|   : +0
  |s3List|   : +253
  |s111List| : +103
  |s210List| : +150
  |s300List| : +0
There are subregions of C0 with at most three solutions.
There are subregions of C0 with at most one solution.
NOTE: we encountered regions in parameter space that
correspond to unordered masses. These were not analyzed.
=========================================================
SUCCESS: the inner bifurcation analysis was conclusive.
=========================================================
Ended computations: Tue Mar 22 09:46:53 2022
=========================================================
\end{verbatim}
The main difference to the previous run, is that there are no cubic bifurcations in this parameter region; only quadratic. Indeed, {\cc s300List} is empty, but {\cc s210List} is not.

In summary, the classifications of the parameter sets corresponding to the lists produced by these three runs are illustrated in Figure~\ref{fig:P3_final4}.
\begin{figure}[h]
\begin{center}
\includegraphics[width=0.99\linewidth]{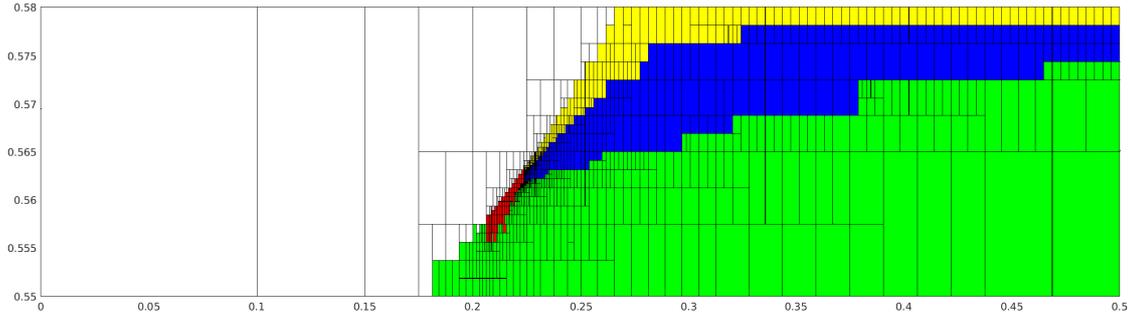}
\end{center}
\captionsetup{width=0.8\linewidth}
\caption{The bifurcations (and non-bifurcations) taking place within the parameter region $\calP_3 = [0, 0.5]\times[0.55, 0.58]$. There are five groups of parameters illustrated here: 1. the white region corresponds to {\cc s0List}; 2. the green region corresponds to {\cc s1List}; 3. the yellow region corresponds to {\cc s111List}; 4. the blue region corresponds to {\cc s210List}; 5. the red region corresponds to {\cc s300List}. Each group of parameters forms a simply connected set.}
\label{fig:P3_final4}
\end{figure}
The data presented in Figure~\ref{fig:P3_final4} (and Figure~\ref{fig:P3_final4_zoom}) combined with our previous results provide all information needed to get an accurate count on the number of solutions.

First, note that the parameters in {\cc s1List} form a connected set having a non-empty intersection with $\calP_1$. We can therefore use part (a) of Theorem~\ref{thm:mass_regions_one_and_two} to conclude that each parameter in {\cc s1List} will yield exactly eight solutions in $\calC$: seven in $\calC_1$ and one in $\calC_2$.

Similarly, we note that the parameters in {\cc s111List} form a connected set having a non-empty intersection with $\calP_2$. Using part (b) of Theorem~\ref{thm:mass_regions_one_and_two} we conclude that out of the ten solutions in $\calC$, seven reside in $\calC_1$ and the remaining three belong to $\calC_2$.

Turning to {\cc s210List}, these parameters form a connected set having a non-empty intersection with both {\cc s1List} and {\cc s111List}. The bifurcation-free connected component of $\calC_2$ detected for all parameters in {\cc s210List} must, by continuity, carry exactly one solution in $\calC_2$. Thus each parameter in {\cc s210List} yields 1, 2 or 3 solutions in $\calC_2$.

\begin{figure}[ht]
\begin{center}
\includegraphics[width=0.40\linewidth]{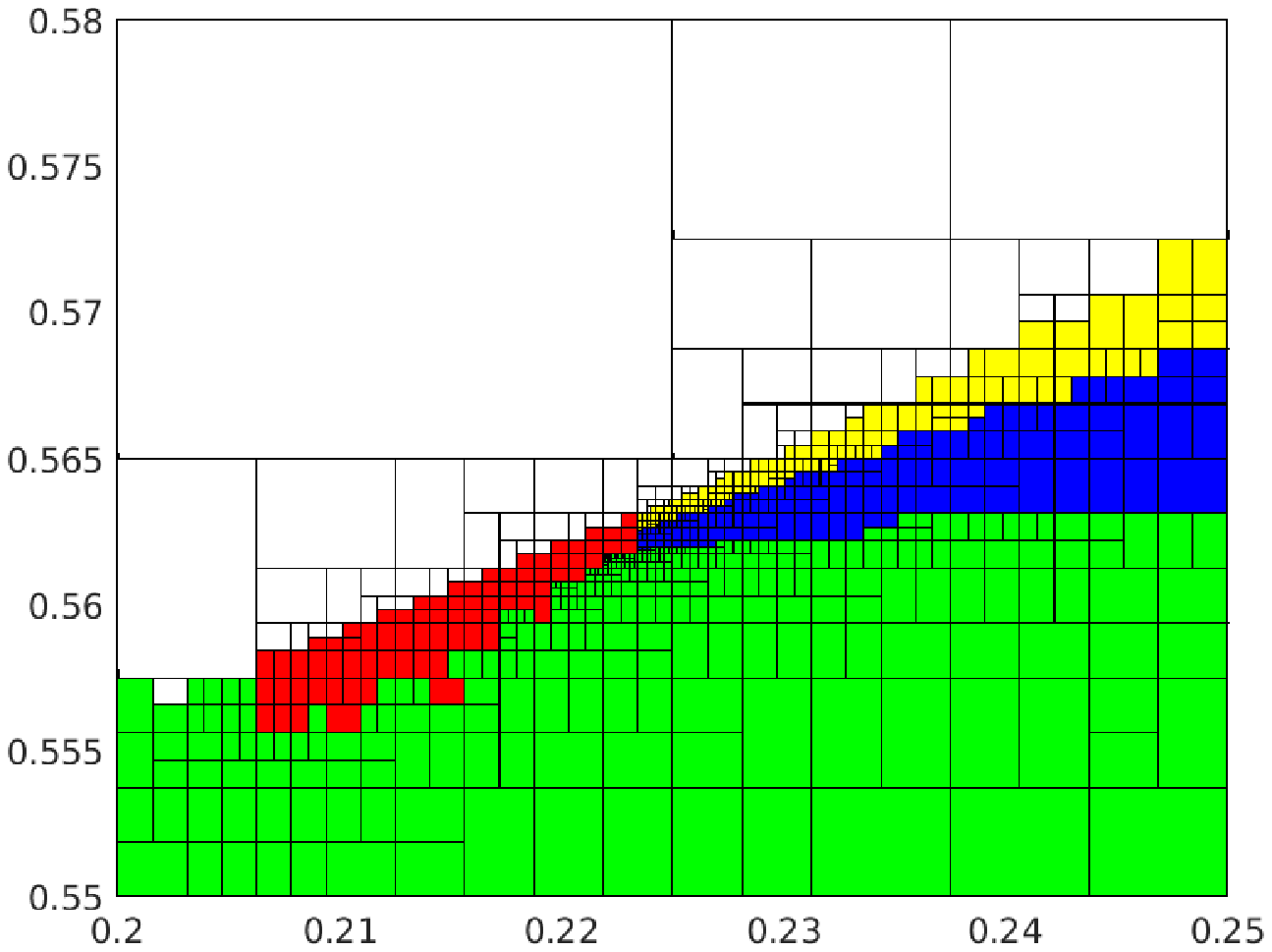}
\hspace*{10mm}
\includegraphics[width=0.40\linewidth]{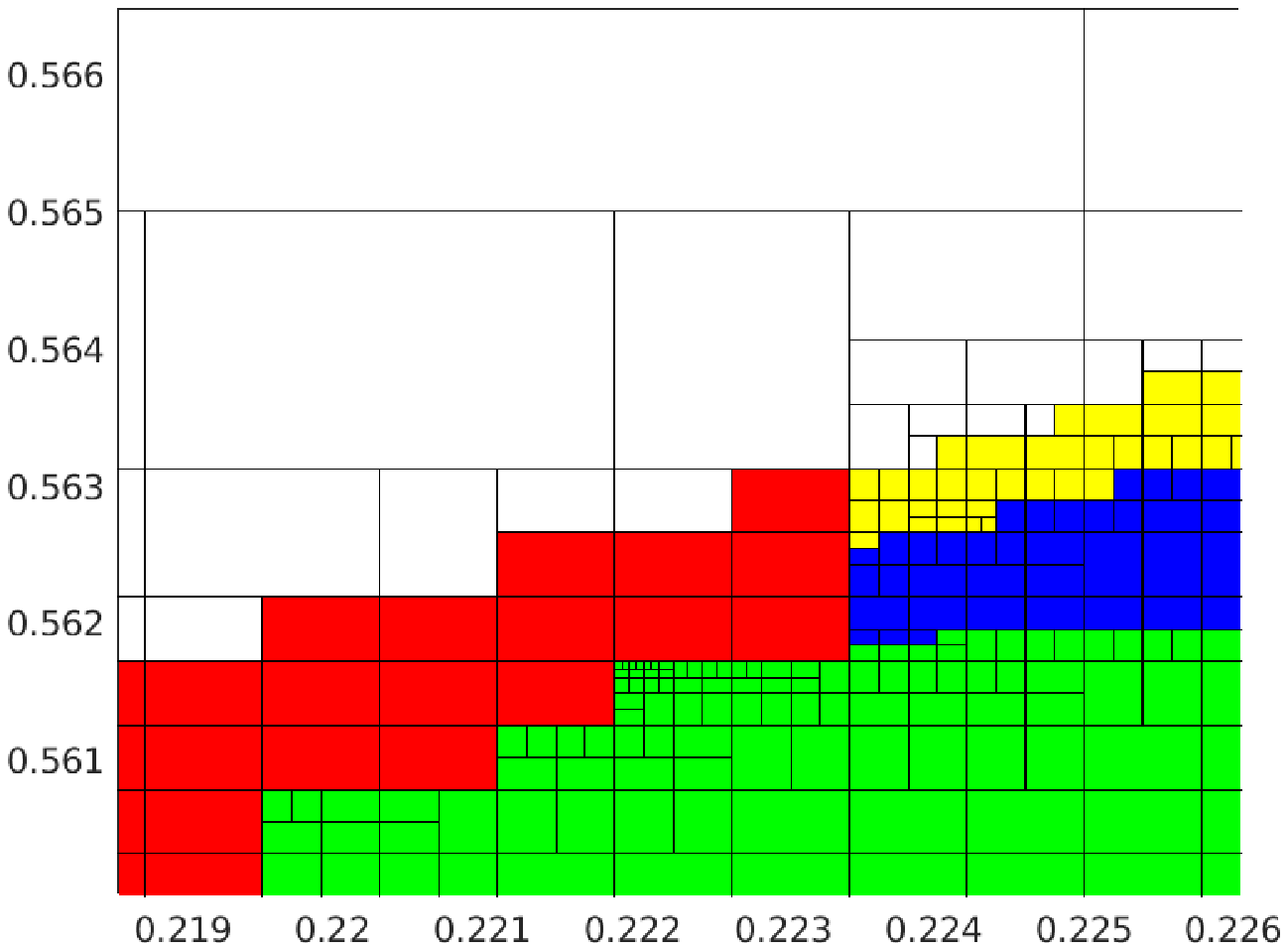}
\end{center}
\captionsetup{width=0.8\linewidth}
\caption{Successive zoom-ins toward the region of finest subdivision in Figure~\ref{fig:P3_final4}.}
\label{fig:P3_final4_zoom}
\end{figure}

Finally, we discuss the parameters in {\cc s300List}. Without additional information, the condition used to indicate the presence of a cubic bifurcation does not exclude there being no solutions; it only bounds the number of solutions from above by three. We do, however, guarantee the existence of at least one solution for elements of {\cc s300List} via an extra topological check performed during the computations.

As mentioned above, each parameter set stored in {\cc s300List} comes equipped with a cover of the associated solution set $\{\ivC_i\}_{i=1}^N$ making up a connected subset of $\calC_2$. Forming the rectangular hull $\ivC$ of all $\ivC_i$, $i=1,\dots,N$ we prove that there must be a zero of $f$ inside $\ivC$ (and thus inside $\calC_2$) using the following topological theorem.

\begin{theorem}\label{thm:topological_zero}
Let $f\colon [x^-, x^+]\times [y^-, y^+]\to \R^2$ be a continuous function, and assume that the following holds:
\begin{itemize}
\item[1.] Both $f_1$ and $f_2$ are negative on the two sides $\{x^-\}\times [y^-, y^+]$ and $[x^-, x^+]\times\{y^-\}$.
\item[2.] Both $f_1$ and $f_2$ are positive at the upper-right corner $(x^+, y^+)$.
\item[3.] $\max\{x\in [x^-, x^+] \colon f_1(x,y^+) = 0\} < \min\{x\in [x^-, x^+] \colon f_2(x,y^+) = 0\}$.
\item[4.] $\max\{y\in [y^-, y^+] \colon f_2(x^+,y) = 0\} < \min\{y\in [y^-, y^+] \colon f_1(x^+,y) = 0\}$.
\end{itemize}
Then $f$ has (at least) one zero in $[x^-, x^+]\times [y^-, y^+]$.
\end{theorem}

Of course, the theorem remains true when we reverse all appearing signs, or interchange the function components. We can also relax assumption 1 and simply demand that $f_1$ is negative on $[x^-, x^+]\times\{y^-\}$, and $f_2$ is negative on $\{x^-\}\times [y^-, y^+]$.  We keep the current (stronger) assumptions as they actually hold for the problem at hand. A typical realization of these are illustrated in Figure~\ref{fig:bifurcation}.

\begin{figure}[h]
\begin{center}
\includegraphics[width=0.6\linewidth]{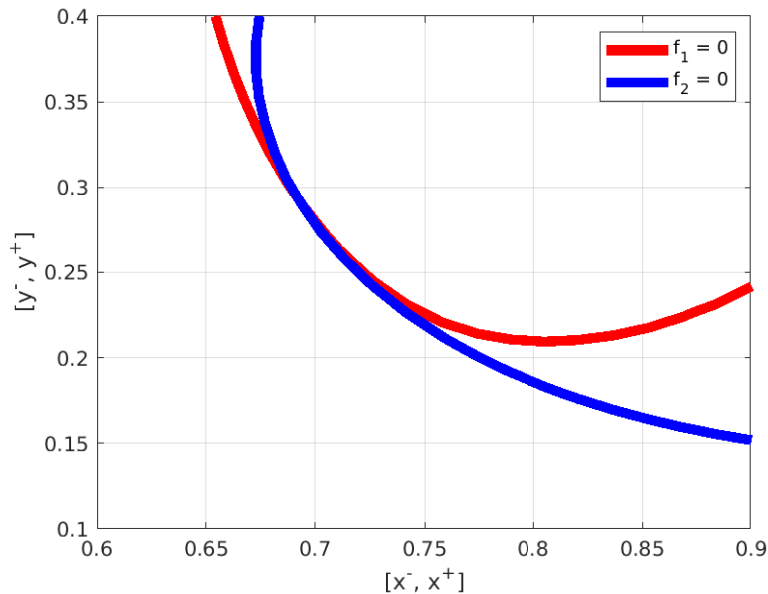}
\end{center}
\captionsetup{width=0.8\linewidth}
\caption{An illustration of how the (set-valued) zero-sets of $f_1$ and $f_2$ behave for typical parameters in {\cc s300List}.}
\label{fig:bifurcation}
\end{figure}

\begin{proof}
By the intermediate value theorem, assumptions 1 and 2 imply that both $f_1$ and $f_2$ change sign (and therefore vanish) somewhere along the two sides $[x^-, x^+]\times\{y^+\}$ and $\{x^+\}\times [y^-, y^+]$. Therefore the sets appearing in assumptions 3 and 4 are non-empty, and the inequalities can be checked. Now, let $x^\star$ satisfy $\max\{x\in [x^-, x^+] \colon f_1(x,y^+) = 0\} < x^\star < \min\{x\in [x^-, x^+] \colon f_2(x,y^+) = 0\}$. Similarly, let $y^\star$ satisfy $\max\{y\in [y^-, y^+] \colon f_2(x^+,y) = 0\} < y^\star < \min\{y\in [y^-, y^+] \colon f_1(x^+,y) = 0\}$. Then, by a continuous deformation, we can form a new rectangle with corners $(x^-,y^-), (x^\star, y^+), (x^+, y^+)$, and $(x^+, y^\star)$ on which we can directly apply the Poincar\'e-Miranda theorem, see \cite{Miranda40}. It follows that $f$ has a zero in the original rectangle.
\end{proof}

As all assumptions of the theorem are open, we can extend it to the case when $f$ depends on (set-valued) parameters. This is what we use for the parameters forming ${\cc s300List}$.

\subsection{Timings}\label{subsec_timings}
We end this section by reporting the timings of all computations. These we carried out sequencially on a latop, using a single thread on an Intel Core i7-7500U CPU running at 2.70GHz. The memory requirements are very low, and are not reported. For all computations we used the same splitting tolerance: ${\cc tol}  = 10^{-6}$. The total wall time amounts to 16m:29s. This can of course be massively reduced by a finer pre-splitting in parameter space, combined with a simple script for parallel execution.
\begin{table}
\centering % used for centering table
\begin{tabular}{c c c c c c}% centered columns (4 columns)
\hline\hline                        %inserts double horizontal lines
sMin & sMax & tMin & tMax & strategy & time \\ [0.5ex]% inserts table heading
\hline                                             % inserts single horizontal line
0.25 & 0.25 & 0.25 & 0.25 & 1 & 0m00s  \\
0.45 & 0.45 & 0.60 & 0.60 & 1 & 0m00s  \\
0.00 & 0.50 & 0.58 & 0.67 & 2 & 0m02s  \\          % inserting body of the table
0.00 & 1e-6 & 0.00 & 1e-6 & 2 & 0m01s  \\
1e-6 & 0.50 & 0.00 & 1e-6 & 2 & 1m30s  \\
0.00 & 1e-6 & 1e-6 & 0.55 & 2 & 7m40s  \\
1e-6 & 0.50 & 1e-6 & 0.55 & 2 & 1m38s  \\
0.00 & 0.20 & 0.55 & 0.58 & 3 & 0m02s  \\
0.20 & 0.25 & 0.55 & 0.58 & 3 & 4m53s  \\
0.25 & 0.50 & 0.55 & 0.58 & 3 & 0m43s  \\[1ex]      % [1ex] adds vertical space
\hline%inserts single line
\end{tabular}\label{tab:timings}% is used to refer this table in the text
\caption{Wall timings for the computer-assisted part of the proof.}% title of Table
\end{table}

\section{Conclusions and future work}

We have demonstrated a novel way to account for all relative equilibria in the planar, circular, restricted 4-body problem, and used it to give a new proof of the results by Barros and Leandro (Theorem~\ref{thm:main_result}). The novely of our approach is that it does not rely upon any algebraic considerations; it is purely analytic. As such it is completely insensitive to the exact shape of the gravitational potential, and generalizes to a wider range of problems. The main advantage, however, is that our method is amenable to set-valued computations, and as such can use the mature techniques and machinery available to solve non-linear equations with computer-assisted methods. This, in turn, gives us a realistic expectation that our approach is transferable to harder instances of the $n$-body problem, and our next challenge will be to work on the \textit{unrestricted} 4-body problem, where we expect the number of relative equilibria to range between 32 and 50 (depending on the masses of the four primaries).

\section{Acknowledgments}
We are very greatful to Professor Carles Sim\'o for bringing this problem to our attention and for fruitful discussions.
W. T. has been partially supported by VR grant 2013-4964.
P. Z. has been partially supported by the NCN grant 2019/35/B/ST1/00655.

\bibliographystyle{plain}
\bibliography{pcr4bp}

\appendix

\section{Bounding the number of small masses}\label{app:small_masses}

The goal of this appendix is to prove Theorem~\ref{thm:small-masses-no-bif}.

From the discussion in Section~\ref{subsec:the_configuration_space} 
(see (\ref{eq:configuration_space})) we know that
that no relative equilibrium exist close to $p_3$, while for 
small $m_1$ or $m_2$ the relative equilibria can exists arbitrarily close to $p_1$ 
and $p_2$ (this is proved in this appendix). If we relax the requirement 
that $m_1 \leq m_2$, then it is enough to investigate the neighborhood of
$p_1$ and the same conclusions will be valid by symmetry close to $p_2$.

We will investigate the neighborhood of $p_1$  when $m_1 \in (0,M]$ and 
$0<m_2<1$ with  $M=10^{-2}$. 
This will be split later in two cases:
when the second mass $m_2$ is away from zero ($m_2 \geq M =10^{-2}$) and when $m_2 \in [0,M]$. 
The difference between these two cases is that for the former
we can control all (four) the solutions in a neighbourhood of $p_1$, while in the latter
two of the four can escape the neighbourhood (while the other two stay always in it). 
Technically this is related to the behaviour
of the eigenvalues of Hessian of the nonsingular part of $V$ at $p_1$. 
In the first case both eigenvalues are separated from $0$ and
the second case the smaller eigenvalue approaches zero when $m_1,m_2 \to 0$.

\subsection{The separation  of eigenvalues at $p_1$ for non-singular part of $V$}
\label{subsec:sep-eigenvalues}
Let
\begin{equation}
  \widetilde{V}(x,y)= V(p_1 + (x,y)) -  \frac{m_1}{\|z-p_1\|}=\frac{1}{2}D\tilde{V}(p_1)((x,y),(x,y)) + O(\|(x,y)\|^3). \label{eq:Vtilde}
\end{equation}

The following straightforward lemma will be used during the entire discussion.
\begin{lemma}
\label{lem:sep-eigenval}
The matrix $D\tilde{V}(p_1)$ is positive definite. Moreover, if $m_1 \leq 1/2$
(in our case it
is true because $m_3  \geq m_1$ and $m_3 + m_1 \leq 1$), then its
eigenvalues $\lambda_1 > \lambda_2$ are given by
\begin{eqnarray}
\lambda_{1,2}&=&\frac{1}{2} \left( 2+m_2+m_3 \pm 3 \sqrt{m_2^2-m_2 m_3+m_3^2}\right)=  \nonumber \\
  &=& \frac{1}{2} \left( 2+m_2+m_3 \pm 3 \sqrt{(m_2+m_3)^2-3m_2 m_3}\right),  \label{eq:eigenval}
\end{eqnarray}
and satisfy the following inequalities
\begin{eqnarray*}
 m_1 + \frac{9 m_2 m_3}{2(m_2+m_3)} &>& \lambda_2 > m_1 + \frac{9m_2 m_3}{4(m_2+m_3)},
 \nonumber \\
  \lambda_1 &<& 3, \nonumber \\
  \frac{3}{4}&\leq& \lambda_1 - \lambda_2 < 3. \nonumber
\end{eqnarray*}
\end{lemma}

The proof of this lemma is straightforward. It is important to notice that
if $m_1,m_2 \to 0$, then $\lambda_2 \to 0$, but if $m_2>\epsilon$
then $\lambda_2$ stays away from $0$. Also that  $\lambda_1-\lambda_2$
has a positive lower bound independent of the size of $m_1$ and $m_2$. Moreover,
with the help of the code \textbf{Bounds\_eigenvalues} and under the assumption
that $m_2, m_3 \geq 10^{-2}$ we can obtain sharper bounds:
$\lambda_1-\lambda_2<2.95643$, $1.62287 < \lambda_1 < 2.97843$,
$0.02167 < \lambda_2 < 0.87695$.
Also, if we allow $m_2, m_3\in[0,1]$ then we have $\lambda_1 \in[1,3]$ and
$\lambda_2\in[0,1]$.

\subsection{Estimates for higher order terms in $V$} \label{subsec:highoredeestm}
We shift the coordinate origin to $p_1$. Then $V(x,y)$ can be written as follows 
\begin{equation}
  V(x,y)=\frac{m_1}{r} + \tilde{V}(0,0) + D\tilde{V}(x,y) + h(x,y),  \label{eq:V-h}
\end{equation}
where $r=\|(x,y)\|$ and $h$ is analytic with respect to all variables and is $O(\|(x,y)\|^3)$
uniformly with respect to the masses. In fact $h$ depends linearly on masses $m_2$ and $m_3$. The term  $V(0,0)$
will be ignored.

For any $R >0$ there exist $C_1=C_1(m_2, m_3, x, y)$ and $C_2=C_2(m_2,m_3, x, y)$ such that
\begin{eqnarray}
   \left\|\nabla_{(x,y)} h(x,y)\right\|&<& C_1 \|(x,y)\|^2, \label{eq:vr-estm1} \\
   \left\|D_{(x,y)}^2 h(x,y)\right\| &<& C_2 \|(x,y)\|. \label{eq:vr-estm2}
\end{eqnarray}
for   $\|(x,y)\| \leq R$.

We obtain quantitative instances of the upper bounds \eqref{eq:vr-estm1} and \eqref{eq:vr-estm2}
as follows: we set $R=2$ (from the discussion in Section~\ref{subsec:sep-eigenvalues} 
it follows that for any $i=1,2,3$ all relative equilibria are contained in $\overline{B}(p_i,R)$.) 
By performing automatic differentiation of the function $h(x,y)$
for all $(x,y)$ with $\|(x,y)\|\leq 10^{-4}$, the first one is three times the sum of all
Taylor coefficients of order 3, and the second one is six times the sum of all Taylor
coefficients of order 3.
Hence,
using the code \textbf{Bounds\_h} and under the assumption of
$0\leq m_2, m_3\leq 1$, and $\|(x,y)\|\leq 10^{-3}$, we obtain the bounds
\begin{equation}
C_1\leq 9.68850, \qquad
C_2\leq 19.3770.  \label{eq:C1C2R=2}
\end{equation}

\subsubsection{Estimates in polar coordinates}
\label{subsec:estm-polar-coord}

Using the polar coordinates we  obtain the following expressions for partial derivatives of $h(r,\varphi)$
\begin{eqnarray*}
  \frac{\partial h}{\partial r}(r,\varphi)&=& \frac{\partial h}{\partial x}(r \cos \varphi,r\sin \varphi) \cos \varphi +  \frac{\partial h}{\partial y}(r \cos \varphi,r\sin \varphi) \sin \varphi,  \\
   \frac{\partial h}{\partial \varphi}(r,\varphi)&=&r \left(-\frac{\partial h}{\partial x} \sin \varphi + \frac{\partial h}{\partial y} \cos \varphi  \right), \\
   \frac{\partial^2 h}{\partial r^2}(r,\varphi)&=& \frac{\partial^2 h}{\partial x^2} \cos^2 \varphi + 2 \frac{\partial^2 h}{\partial x \partial y} \sin \varphi \cos \varphi + \frac{\partial^2 h}{\partial y^2} \sin^2 \varphi, \\
   \frac{\partial^2 h}{\partial r \partial \varphi}(r,\varphi)&=& \left(-\frac{\partial h}{\partial x} \sin \varphi + \frac{\partial h}{\partial y} \cos \varphi  \right) + \\
   & & r \left( \left(-\frac{\partial^2 h}{\partial x^2} + \frac{\partial^2 h}{\partial y^2}\right) \sin \varphi \cos \varphi + \frac{\partial^2 h}{\partial x \partial y} (\cos^2 \varphi - \sin^2 \varphi) \right), \\
   \frac{\partial^2 h}{\partial \varphi^2}(r,\varphi)&=& -r\left(\frac{\partial h}{\partial x} \cos \varphi + \frac{\partial h}{\partial y} \sin \varphi  \right) + \\
   & & r^2 \left( \frac{\partial^2 h}{\partial x^2} \sin^2 \varphi - 2 \frac{\partial^2 h }{\partial x \partial y}\sin \varphi \cos \varphi +
   \frac{\partial^2 h}{\partial y^2} \cos^2 \varphi\right).
\end{eqnarray*}
Let
\begin{equation*}
 e=(\cos \varphi,\sin \varphi), \quad e^\bot=(-\sin \varphi,\cos \varphi).
\end{equation*}
In a more geometric way we can express the above partial
derivatives as follows
\begin{eqnarray*}
  \frac{\partial h}{\partial r}(r,\varphi)&=& (\nabla h,e), \\
  \frac{\partial h}{\partial \varphi}(r,\varphi)&=& r(\nabla h,e^\bot), \\
  \frac{\partial^2 h}{\partial r^2}(r,\varphi)&=& D^2h(e,e), \\
   \frac{\partial^2 h}{\partial r \partial \varphi}(r,\varphi)&=& (\nabla h,e^\bot) + r D^2h(e^\bot,e), \\
  \frac{\partial^2 h}{ \partial \varphi^2}(r,\varphi)&=&-r (\nabla h,e) + r^2 D^2
  h(e^\bot,e^\bot).
\end{eqnarray*}

Therefore we have the following estimates for the partial derivatives of $h$ with respect to polar coordinates.
\begin{lemma}
\label{lem:vr-par-polar-estm}
Assume that (\ref{eq:vr-estm1},\ref{eq:vr-estm2}) hold for $r \leq R$. Then
the following estimates are satisfied for $r \leq R$
\begin{eqnarray*}
 \left| \frac{\partial h}{\partial r}(r,\varphi) \right|&<& C_1 r^2, \\
 \left|  \frac{\partial h}{\partial \varphi}(r,\varphi) \right|&<& C_1r^3, \\
  \left| \frac{\partial^2 h}{\partial r^2}(r,\varphi) \right|&<& C_2 r, \\
  \left|  \frac{\partial^2 h}{\partial r \partial \varphi}(r,\varphi) \right|&<& C_1 r^2 + C_2 r^2=(C_1+C_2)r^2, \\
  \left| \frac{\partial^2 h}{ \partial \varphi^2}(r,\varphi) \right|&<&  C_1 r^3 + C_2
  r^3=(C_1+C_2)r^3.
\end{eqnarray*}
\end{lemma}

\begin{lemma}
\label{lem:vr-par-polar-tilde-estm} The same  assumptions as in
Lemma~\ref{lem:vr-par-polar-estm}. Let us set
\begin{eqnarray*}
\tilde{h}_\varphi(r,\varphi)&=& \frac{1}{r^3} \frac{\partial h}{\partial \varphi}(r,\varphi), \\
\tilde{h}_{\varphi \varphi}(r,\varphi)&=& \frac{1}{r^3} \frac{\partial^2 h}{\partial \varphi^2}(r,\varphi), \\
\tilde{h}_{\varphi r}(r,\varphi)&=& \frac{1}{r^2} \frac{\partial^2 h}{\partial \varphi \partial r}(r,\varphi), \\
\tilde{h}_r(r,\varphi)&=& \frac{1}{r^2} \frac{\partial h}{\partial r}(r,\varphi), \\
  \tilde{h}_{rr}(r,\varphi)&=& \frac{1}{r} \frac{\partial^2 h}{\partial r^2}(r,\varphi).
\end{eqnarray*}
Then we have for $r \leq R$
\begin{eqnarray}
  |\tilde{h}_\varphi(r,\varphi)| < C_1, \quad   |\tilde{h}_{\varphi \varphi}(r,\varphi)| < C_1 + C_2, \quad |\tilde{h}_{\varphi r}(r,\varphi)| < C_1 + C_2 \\
   |\tilde{h}_r(r,\varphi)| < C_1, \quad  |\tilde{h}_{rr}(r,\varphi)| < C_2
\end{eqnarray}
\end{lemma}

\subsection{Our system in polar coordinates near $p_1$}
By an orthogonal  change of variables $V$ we can diagonalize $D\widetilde{V}(0)$, hence $V$ can be written as follows (compare (\ref{eq:V-h}))
\begin{equation}
  V(x,y)=\frac{m_1}{r} +  \lambda \frac{x^2}{2} + a \frac{ y^2}{2} + h(x,y). \label{eq:V-rad-sepm}
\end{equation}
Observe that $\lambda$ and $a$ depend on mass parameters (see Lemma~\ref{lem:sep-eigenval}) and the same holds for the coordinate change. Bounds
on $h$ and its derivatives uniform with respect to masses have been obtained for $r<R$ in previous subsections.

In polar coordinates our potential (\ref{eq:V-rad-sepm}) becomes (we set $m=m_1$)
\begin{equation}
  V(r,\varphi)= \frac{m}{r} + \frac{a r^2}{2} + \frac{(\lambda -a)r^2 \cos^2 \varphi}{2} + h(r,\varphi)\label{eq:Vreal-rad}
\end{equation}
and
\begin{eqnarray}
  \frac{\partial V}{\partial r}(r,\varphi)&=&-\frac{m}{r^2} + ar + (\lambda - a)r \cos^2 \varphi + \frac{\partial h}{\partial r}(r,\varphi) = \nonumber \\
  & =&  -\frac{m}{r^2} + r \left( \frac{\lambda + a}{2} +  \frac{\lambda - a}{2} \cos(2 \varphi)\right) + r^2 \tilde{h}_r,
     \label{eq:vreal-fr} \\
  \frac{\partial V}{\partial \varphi}(r,\varphi)&=&-(\lambda-a)r^2  \sin (2 \varphi) / 2 + \frac{\partial h}{\partial
  \varphi}(r,\varphi)=  \label{eq:vreal-fp}\\
  &=& -(\lambda-a)r^2  \sin (2 \varphi) / 2 +  r^3 \tilde{h}_\varphi(r,\varphi) = \nonumber \\
 &=& r^2 \left(  -(\lambda-a)  \sin (2 \varphi) / 2 +  r \tilde{h}_\varphi(r,\varphi)
 \right), \nonumber \\
 \frac{\partial^2 V}{\partial r^2}(r,\varphi)&=& \frac{2m}{r^3} +  \frac{\lambda + a}{2} +  \frac{\lambda - a}{2} \cos(2 \varphi) + \frac{\partial^2 h}{\partial r^2}(r,\varphi)\\
 \frac{\partial^2 V}{\partial \varphi^2}(r,\varphi)&=& -(\lambda -a)r^2 \cos (2 \varphi) + \frac{\partial^2 h}{\partial \varphi^2}(r,\varphi) = \nonumber \\
   &=&  r^2 \left(-(\lambda -a) \cos (2 \varphi) + r \tilde{h}_{\varphi \varphi}(r,\varphi)\right) \label{eq:vrdfdp},\\
  \frac{\partial^2 V}{\partial \varphi \partial r}(r,\varphi)&=&-(\lambda -a)r \sin(2\varphi) + \frac{\partial^2 h}{ \partial r \partial \varphi}. \label{eq:d2rfi}
\end{eqnarray}

We will study the system
\begin{eqnarray}
  \frac{\partial V}{\partial \varphi}(r,\varphi)&=&0 ,  \label{eq:dVdvarphi} \\
   \frac{\partial V}{\partial r}(r,\varphi)&=&0.
\end{eqnarray}
It turns out that it is relatively easy to solve (\ref{eq:dVdvarphi}), i.e. to find four curves $\varphi(r)$ such that all solutions
of (\ref{eq:dVdvarphi}) are of the form $(r,\varphi(r))$.  This is done in the next subsection for $0<m_2<1$ and $m_1$ small enough. 

Then we study the equation  $\frac{\partial V}{\partial r}(r,\varphi(r))=0$ on each of these curves. This is the place where we will split our
considerations into two cases: $m_2$ bounded from below and both $m_1,m_2 \to 0$. The second case is much more subtle.

\subsection{Solving $\frac{\partial V}{\partial \varphi}(r,\varphi)=0$ for $\varphi(r)$.}
\label{subsec:vreal-sol-varphi(r)}

Our strategy to study the solutions of 
$\frac{\partial V}{\partial \varphi}(r,\varphi(r))=0$ 
in the set $\{(r,\varphi): r \leq R, \varphi \in [0,2\pi] \}$
is: we try to find possible large intervals such that the solution
is excluded, for this it is enough to have (see (\ref{eq:vreal-fp}) and Lemma~\ref{lem:vr-par-polar-tilde-estm})
\begin{equation}
  (\lambda-a)  |\sin (2 \varphi)| / 2 \geq  R C_1, \label{eq:fphi-no-sol}
\end{equation}
and then on the complementary  intervals we want $\frac{\partial^2 V}{\partial \varphi^2}$ to be either positive or negative on
the whole interval, this is guaranteed by the following inequality
(see (\ref{eq:vrdfdp}))  and Lemma~\ref{lem:vr-par-polar-tilde-estm})
\begin{equation}
  (\lambda - a) |\cos (2\varphi)| \geq R (C_1 + C_2).  \label{eq:fpphiphi-pos}
\end{equation}
\begin{rem}
Recall that all this computations must have $R\leq 2$ since this is the range where 
the existence of the constants $C_1, C_2$ have been proved.
\end{rem}

To deal with (\ref{eq:fphi-no-sol}) let us set
\begin{equation}
  \alpha =  \arcsin \left(\frac{2 R C_1}{\lambda - a} \right) \in (0,\pi/2).  \label{eq:vrel-sol-phi-def-alpha}
\end{equation}
For this to make sense we must have that
\begin{equation}
  \frac{2 R C_1}{\lambda - a} < 1,  \label{eq:fphi-cc1}
\end{equation}
which in our case ($C_1=9.68850$, see \eqref{eq:C1C2R=2}, and $\lambda-a > \frac{3}{4}$) 
it is satisfied when $R< 0.03952$.

We define the following sets (intervals) in $\varphi$
\begin{eqnarray}\label{eq:sectors}
   J^+_0&=& \left[-\frac{\alpha}{2},\frac{\alpha}{2}\right], \quad  J^-_0= \left[\frac{\pi - \alpha}{2},\frac{\pi+\alpha}{2}\right], \\
   J^+_1&=& \left[\frac{2\pi - \alpha}{2},\frac{2\pi + \alpha}{2}\right], \quad   J^-_1= \left[\frac{3\pi + \alpha}{2},\frac{4\pi-\alpha}{2}\right], \\
   N_0&=& \left[\frac{\alpha}{2},\frac{\pi - \alpha}{2}\right], \quad N_1= \left[\frac{\pi + \alpha}{2},\frac{2\pi-\alpha}{2}\right], \\
    N_2&=& \left[\frac{2\pi+\alpha}{2},\frac{3\pi -\alpha}{2}\right], \quad N_3= \left[\frac{3\pi+\alpha}{2},\frac{4\pi-\alpha}{2}\right].
\end{eqnarray}
These sectors are positioned as follows as we move in the direction of increasing $\varphi$: $J^+_0$, $N_0$, $J^-_0$, $N_1$, $J^+_1$,
 $N_2$, $J^-_1$, $N_3$, $J^+_0$, \dots.   The meaning of $\pm$ in $J$ symbol is the sign of $\cos(2 \varphi)$.

\begin{lemma}
\label{lem:vr-fp-int-estm}
Assume that assumptions of Lemma~\ref{lem:vr-par-polar-estm} are satisfied.
Assume that
\begin{equation}
 R^2 \left(4C_1^2 + (C_1 + C_2)^2 \right) < (\lambda - a)^2 \label{eq:main-rad-cond}
\end{equation}

Then  the equation
 \begin{equation}
 \frac{\partial V}{\partial \varphi}(r,\varphi(r))=0 \label{eq:dvdf=0}
 \end{equation}
 has a unique solution for $r \leq R$ in each of the 
intervals $J_{0,1}^\pm$. These are all solutions
 of (\ref{eq:dvdf=0}) with $r \leq R$.

 The solutions curves of (\ref{eq:dvdf=0}) satisfy
 \begin{equation}
   \varphi(r)= \{0,\pi/2,\pi,3\pi/2\} + \delta(r), \quad \delta(0)=0.
 \end{equation}
 where $\delta(r)$ is a different function for each branch and satisfies the following estimate
 \begin{equation}
   |\delta(r)| \leq \arcsin \left(\frac{2 r C_1}{\lambda - a} \right) 
\leq \frac{2C_1 r}{\lambda-a} + \left(\frac{2C_1}{\lambda-a}\right)^3 r^3/6 + O(r^4). 
\label{eq:dvdf-delta-estm}
\end{equation}
In particular, 
by assuming $C_1=9.68850$, $C_2=19.3770$ \textbf{(see (\ref{eq:C1C2R=2}))}, $m_1\leq 10^{-2}$, 
$m_2\in[0,1]$ and that
$\lambda-a>\frac{3}{4}$ we obtain that
$R$ must be less than $0.02193$ and that
$|\delta(r)|\leq 0.0259 $ for all $|r|\leq 10^{-3}$.
\end{lemma}
\textbf{Proof:}
Observe first that (\ref{eq:main-rad-cond}) implies condition (\ref{eq:fphi-cc1}) and hence each of the intervals  $J^+_0$, $N_0$, $J^-_0$, $N_1$, $J^+_1$,
 $N_2$, $J^-_1$, $N_3$ is nonempty and  we will show that the following conditions  are satisfied
\begin{eqnarray}
  \frac{\partial V}{\partial \varphi}(r,\varphi) &>&0 , \quad \mbox{for $r \leq R,\  \varphi \in N_1 \cup N_3$},  \\
  \frac{\partial V}{\partial \varphi}(r,\varphi) &<&0 , \quad \mbox{for $r \leq R,\  \varphi \in N_0 \cup N_2$}.
\end{eqnarray}
and
 \begin{eqnarray}
  \frac{\partial^2 V}{\partial \varphi^2} (r,\varphi) <0, \quad \mbox{for $\varphi \in J^+_0 \cup J^+_1$},  \label{eq:fpp-neg} \\
  \frac{\partial^2 V}{\partial \varphi^2} (r,\varphi) >0, \quad \mbox{for $\varphi \in J^-_0 \cup J^-_1$}. \label{eq:fpp-pos}
 \end{eqnarray}

To establish (\ref{eq:fpp-neg},\ref{eq:fpp-pos}) we  need condition (\ref{eq:fpphiphi-pos}) to hold for $\varphi \in J^+_0 \cup J^+_1 \cup J^-_0 \cup J^-_1 $.

It is easy to see that it will hold iff
\begin{equation}
 (\lambda - a) |\cos (2 \varphi)| \geq  (\lambda - a) |\cos (\alpha)| \geq R (C_1 + C_2). \label{eq:fphi-cc2}
\end{equation}
From (\ref{eq:vrel-sol-phi-def-alpha}) we know the value of  $\sin(\alpha)$, therefore (\ref{eq:fphi-cc2}) is equivalent to the following
chain of inequalities
\begin{eqnarray*}
  \cos^2(\alpha) &\geq&  \frac{R^2 (C_1 + C_2)^2}{(\lambda-a)^2} \\
  1 - \sin^2(\alpha) &\geq & \frac{R^2 (C_1 + C_2)^2}{(\lambda-a)^2} \\
  1- \frac{4R^2C_1^2}{(\lambda-a)^2}  &\geq & \frac{R^2 (C_1 + C_2)^2}{(\lambda-a)^2} \\
  (\lambda-a)^2 &\geq& 4R^2C_1^2 + R^2 (C_1 + C_2)^2 = R^2 \left(4C_1^2 + (C_1 + C_2)^2 \right).
\end{eqnarray*}
This is condition (\ref{eq:main-rad-cond}). Monotonicity in $J_{0,1}^\pm$ and opposite signs on the end points give us the existence of one branch of $\varphi(r)$ in each sector.

In fact in the above reasoning we can use any $r \leq R$ to define $\alpha=\alpha(r)=\arcsin\left(\frac{2r C_1}{\lambda-a}\right)$ and define the sectors
$J_{0,1}^\pm$ using this $\alpha$. In this way we obtain that
\begin{equation*}
  |\delta(r)| \leq \arcsin \left(\frac{2rC_1}{\lambda-a}\right).
\end{equation*}
From this and the series expansion of $\arcsin x$ we obtain  estimates of $\delta(r)$.
\qed

Now we estimate the derivative of $\varphi'(r)$.
\begin{lemma}
\label{lem:osm-der-varphi(r)}
The same assumptions as in Lemma~\ref{lem:vr-fp-int-estm}.
Then there exists $R_1 = 10^{-3}$, such that for all
$r\leq R_1$ and all branches of $\varphi(r)$ holds
\begin{eqnarray}
  |\varphi'(r)| \leq   \frac{3 C_1+C_2}{(\lambda -a) 
\sqrt{1-\left(\frac{2C_1r}{\lambda-a}\right)^2} 
-(C_1+C_2)r}
  \leq 67.2174.
\end{eqnarray}

\end{lemma}
\textbf{Proof:}
From (\ref{eq:vrdfdp}), (\ref{eq:d2rfi}) and Lemma~\ref{lem:vr-par-polar-estm} it follows that
\begin{eqnarray*}
  |\varphi'(r)|&=&\left|\frac{\frac{\partial^2 V}{\partial \varphi\partial r}}{\frac{\partial^2 V}{\partial \varphi^2} } \right|=
                \left|\frac{-(\lambda -a)r \sin(2\varphi) + \frac{\partial^2 h}{ \partial r \partial \varphi}}{-(\lambda -a)r^2 \cos (2 \varphi) + \frac{\partial^2 h}{\partial \varphi^2}(r,\varphi)}\right| \leq \\
               &\leq& \frac{(\lambda -a)r |\sin(2\varphi)| + (C_1+C_2)r^2}{(\lambda -a)r^2 |\cos (2 \varphi)| -(C_1+C_2)r^3}.
\end{eqnarray*}

Now, using the fact that $\left|\sin(2\varphi)\right|\leq \frac{2 C_1 R_1}{\lambda-a}$ and
$|\cos(2\varphi)|=\sqrt{1-\sin(2\varphi)^2}$
we obtain the desired inequality. Finally, by plugging in the same 
constants as in Lemma~\ref{lem:vr-fp-int-estm} we obtain the numerical upper bound.

\qed

\subsection{The case of $m_1\to 0$ with $m_2$ bounded from below. }

Here we treat the case
$0< m_1 \leq M$ and $(x,y) \in \overline{B}(0,R)$ (we will see below that 
it suffices $M=10^{-2}$ and $R=10^{-3}$)
in the configuration space with the goal to obtain
$R>0$ independent of $m_1 \in (0,M]$, where we know that there
is at most four solutions (central configurations) and all of them are non-degenerate.

We work with the representation of potential $V$ given by (\ref{eq:Vreal-rad}). In that setting 
\begin{equation*}
\lambda=\lambda_1 > a=\lambda_2
\end{equation*} 
where $\lambda_1$, $\lambda_2$ are as in Lemma~\ref{lem:sep-eigenval} and for $m_2 \geq 10^{-2}$ 
we have the bounds
\begin{equation}
  \lambda > 1.62, \qquad a > 0.02167.  
\end{equation}

\subsubsection{Solving $\frac{\partial V}{\partial r}(r,\varphi(r))=0$}
\label{subsec:osm-sol}

On curves $\varphi(r)$ in $J_{0,1}^+$ (i.e. $\cos(2\varphi)>0$) we have from
(\ref{eq:vreal-fr}) and Lemma~\ref{lem:vr-par-polar-tilde-estm}
\begin{eqnarray}
  \frac{\partial V}{\partial r}(r,\varphi(r)) &=& - \frac{m}{r^2} +
  r \left( \frac{\lambda+a}{2} +
  \frac{\lambda-a}{2}\cos(2\varphi(0)+2\delta(r)) \right) + r^2 \tilde{h}_r \label{eq:fr-on-pos-curve} \\
  &=& - \frac{m}{r^2} + \lambda r  - \frac{\lambda-a}{2} S(r) r^3 +
  r^2 \tilde{h}_r \nonumber = - \frac{m}{r^2} + \lambda r  + r^2 L(r),    \nonumber
\end{eqnarray}
with $S(r)=2 A(r) (\delta(r))^2$, where $A(r)\in[-1,1]$. (All these come from
expanding $\cos$ around $\varphi(0)$ up to order 2.)
Notice that
\begin{equation}
\label{eq: bound_L}
|L(r)|\leq\left|\frac{\lambda-a}2\right|2R
\left|\arcsin\left(\frac{2rC_1}{\lambda-a}\right)\right|^2+C_1,
\end{equation}
with $r\leq R$.

With the same assumptions and constants as the ones from Lemma~\ref{lem:vr-fp-int-estm} 
we obtain the numerical upper bound ${|L(r)|\leq 9.68851}$. So, 
\[
\frac{\partial V}{\partial r}(r,\varphi(r)) \in
- \frac{m}{r^2} + \lambda r+[-9.68851, 9.68851]r^2, \text{ for } R\leq 10^{-3}.
\]

For the derivative we obtain (we use (\ref{eq:vrdfdp}),(\ref{eq:d2rfi}),
Lemmas~\ref{lem:vr-par-polar-estm},~\ref{lem:vr-fp-int-estm} and~\ref{lem:osm-der-varphi(r)})
\begin{eqnarray}
  \frac{d}{dr}\left(\frac{\partial V}{\partial r}(r,\varphi(r))\right)
  &=& \frac{\partial^2 V}{\partial r^2} + \frac{\partial^2 V}{\partial r \partial \varphi}
  \varphi'(r) = \nonumber\\
  &=& \frac{2m}{r^3} +  \frac{\lambda + a}{2} +  \frac{\lambda - a}{2}
  \cos(2 \varphi) + \frac{\partial^2 h}{\partial r^2}(r,\varphi) + \nonumber \\
  &+&  \left(-(\lambda -a)r \sin(2\varphi) + \frac{\partial^2 h}{ \partial r \partial \varphi}
  \right) \varphi'(r) = \nonumber \\
  &=& \frac{2m}{r^3} + \lambda - \frac{\lambda-a}{2} S(r)r^2+\tilde{h}_{rr}r+\nonumber\\
  &+& r^2\varphi'(r)(\tilde{h}_{\varphi r}-r(\lambda-a)A(r)\delta(r))= \nonumber \\
  &=& \frac{2m}{r^3} + \lambda +r(-\frac{\lambda-a}{2} S(r)r+\tilde{h}_{rr}+\nonumber\\
  &+& r\varphi'(r)(\tilde{h}_{\varphi r}-r(\lambda-a)A(r)\delta(r)))= \nonumber \\
  &=&  \frac{2m}{r^3}  + \lambda  + rT(r). \label{eq:drfr-on-pos-curve}
\end{eqnarray}
with $A(r)\in[-1,1]$. So we have 
\begin{equation}
\label{eq:bound_T}
\left|T(r)\right|\leq \left|\frac{\lambda-a}{2}\right|R
\left|\delta(r)\right|+C_2
+R\left|\varphi'(r)\right|
(C_1+C_2+R(\lambda-a)\left|\delta(r)\right|)
\end{equation}
for $r\leq R$.

With the same assumptions and constants as the ones from Lemma~\ref{lem:vr-fp-int-estm} 
we obtain the numerical upper bound ${|T(r)|\leq 21.33076}$. So, 
\[
\frac{d}{dr}\left(\frac{\partial V}{\partial r}(r,\varphi(r))\right)\in
\frac{2m}{r^3} + \lambda +[-21.33076, 21.33076]r, \text{ for } r\leq 10^{-3}.
\]
Notice that this last interval expression is always positive because 
$\lambda>21.33076\cdot 10^{-3}$.

Similarly as before, on curves in $J_{0,1}^-$ (i.e. $\cos(2\varphi)<0$) we have from
(\ref{eq:vreal-fr}) and Lemma~\ref{lem:vr-par-polar-tilde-estm}
\begin{eqnarray}
  \frac{\partial V}{\partial r}(r,\varphi(r)) &=& - \frac{m}{r^2} +
  r \left( \frac{\lambda+a}{2} +
  \frac{\lambda-a}{2}\cos(2\varphi(0)+2\delta(r))
  \right) + r^2 \tilde{h}_r \label{eq:fr-on-neg-curve} \\
  &=& - \frac{m}{r^2} + a r  + \frac{\lambda-a}{2} S(r) r^3 +
  r^2 \tilde{h}_r \nonumber = - \frac{m}{r^2} + a r  + r^2 L(r),    \nonumber
\end{eqnarray}
with $S(r)=2 A(r) (\delta(r))^2$, where $A(r)\in[-1,1]$. (All these come from
expanding $\cos$ around $\varphi(0)$ up to order 2.)
Hence, the bound in \ref{eq: bound_L} is the same as in here.

With the same assumptions and constants as the ones from Lemma~\ref{lem:vr-fp-int-estm} 
we obtain the numerical upper bound ${|L(r)|\leq 9.68851}$. So, 
\[
\frac{\partial V}{\partial r}(r,\varphi(r)) \in
- \frac{m}{r^2} + a r+[-9.68851, 9.68851]r^2, \text{ for } r\leq 10^{-3}.
\]

For the derivative we obtain (see the derivation of (\ref{eq:drfr-on-pos-curve}))
\begin{eqnarray}
  \frac{d}{dr}\left(\frac{\partial V}{\partial r}(r,\varphi(r))\right)
  &=& \frac{\partial^2 V}{\partial r^2} + \frac{\partial^2 V}{\partial r \partial \varphi}
  \varphi'(r) = \nonumber\\
  &=& \frac{2m}{r^3} +  \frac{\lambda + a}{2} +  \frac{\lambda - a}{2}
  \cos(2 \varphi) + \frac{\partial^2 h}{\partial r^2}(r,\varphi) + \nonumber \\
  &+&  \left(-(\lambda -a)r \sin(2\varphi) + \frac{\partial^2 h}{ \partial r \partial \varphi}
  \right) \varphi'(r) = \nonumber \\
  &=& \frac{2m}{r^3} + a + \frac{\lambda-a}{2} S(r)r^2+\tilde{h}_{rr}r+\nonumber\\
  &+& r^2\varphi'(r)(\tilde{h}_{\varphi r}+r(\lambda-a)A(r)\delta(r))= \nonumber \\
  &=& \frac{2m}{r^3} + a + r(\frac{\lambda-a}{2} S(r)r+\tilde{h}_{rr}+\nonumber\\
  &+& r\varphi'(r)(\tilde{h}_{\varphi r}+r(\lambda-a)A(r)\delta(r)))= \nonumber \\
  &=&  \frac{2m}{r^3} + a + rT(r). \label{eq:drfr-on-neg-curve}
\end{eqnarray}
with $A(r)\in[-1,1]$. 
Hence, the same bound \ref{eq:bound_T} for $T$ works here.

With the same assumptions and constants as the ones from Lemma~\ref{lem:vr-fp-int-estm} 
we obtain the numerical upper bound ${|T(r)|\leq 21.33076}$. So, 
\[
\frac{d}{dr}\left(\frac{\partial V}{\partial r}(r,\varphi(r))\right)\in
\frac{2m}{r^3} + a +[-21.33076, 21.33076]r, \text{ for } r\leq 10^{-3}.
\]
In particular, since $a>0.02167$ we obtain that the interval expression above is positive 
since ${a>21.33076\cdot 10^{-3}}$.

\subsubsection{The solutions are non-degenerate}
\label{subsubsec:non-deg-first-case}

We have proven above that
\[ 
\frac{d}{dr}\left(\frac{\partial V}{\partial r}(r,\varphi(r))\right)
\]
does not vanish in the range $m_1 \in (0,10^{-2}]$, $m_2\geq [10^{-2},1)$ 
and $r\leq 10^{-3}$. This suffices to see that any critical point of $V$ is 
non-degenerate. In fact, 
this is equivalent to Hessian is full rank, since 
\begin{equation*}
  0<\frac{d}{dr}\left(\frac{\partial V}{\partial r} (r, \varphi(r)) \right)= 
\frac{\partial^2 V}{\partial r^2} +
\frac{\partial^2 V}{\partial r \partial \varphi}\varphi'(r)= 
\left(\frac{\partial^2 V}{\partial \varphi^2}\right)^{-1} \det D^2V (r,\varphi(r)).
\end{equation*}
Hence $\det D^2V (r,\varphi(r)) \neq 0$. 

\subsection{The case of $m_1, m_2\to 0$.}
\label{subsec:m1m1->0}

In this section we are interested in studying the critical points in a neighbourhood of $p_1$
and with masses $m_1, m_2\leq M$ ($M$ to be stated below, it is $M=10^{-2}$). For doing so,
we will use polar coordinates $(r, \varphi)$ centered at $p_1$. Hence, we will study the system of equations
\begin{eqnarray}
 \frac{\partial V}{\partial \varphi}(r,\varphi)=0,\label{eq: partial_V_partial_phi}\\
 \frac{\partial V}{\partial r}(r,\varphi)=0.\nonumber
\end{eqnarray}

In subsection~\ref{subsec:vreal-sol-varphi(r)} 
the equation  $\frac{\partial V}{\partial \varphi}(r,\varphi)=0$ 
for $\varphi$ was studied obtaining
four curves $(r,\varphi(r))$ on which will have uniform estimates over the 
whole range of $(m_1,m_2)$ including $(0,0)$. 
These bounds were produced in a coordinate system in which $D^2\tilde{V}(p_1)$ is diagonal.   

This time we work in another coordinate system
 (this choice simplifies the computations)
\begin{eqnarray}
  p_1=\left(1,0\right),
    \quad p_2=\left(\frac12, \frac{\sqrt 3}2\right), \quad
  p_3=\left(0, 0\right).
\end{eqnarray}
With this choice of axes we obtain that the Hessian matrix of $D^2\tilde{V}(p_1)$ in $(r,\varphi)$-coordinates has the following form for $m_1=m_2=0$
\begin{equation}
  D^2 \tilde{V}(p_1)=\left[\begin{array}{cc}
  \lambda>0 & 0 \\
  0 & 0  \end{array}\right]
\end{equation}
The second eigenvalue $a=\lambda_2=0$ is the reason why we cannot apply tools from the previous subsection. 
This is the cause that while studying the equation $\frac{\partial V}{\partial r}(r,\varphi(r))=0$ we will encounter that the number of solutions
is more subtle and, in some cases, it will depend on the ratio $\frac{m_1}{m_2}$.

\subsubsection{Some useful formulae}
\label{sec:two-small-masses}

In polar coordinates we have
\begin{eqnarray*}
  z&=&(1 + r \cos \varphi,r \sin \varphi), \\ \|z-c\|^2 &=& \left((1 - (m_1 +
  m_2/2))^2 + \frac{3 m_2^2}{4} \right) + r^2 + 2 r\cos \varphi + \\ & &m_1
  \left( -2 r \cos \varphi  \right) + m_2 \left( - r \cos \varphi - \sqrt{3} r
  \sin \varphi \right)
\end{eqnarray*}
Since we are interested in the stationary solutions of $\nabla V(r,\varphi)$,
in the following computations
we will drop the terms in $\|z-c\|^2$ which do not depend of
$(r,\varphi)$. Hence, the important part of $\|z-c\|^2$ is
\begin{eqnarray*}
r^2 + 2 r\cos \varphi + m_1 \left( -2 r \cos \varphi  \right) +
m_2 \left( - 2 r \cos (\varphi-\pi/3) \right).
\end{eqnarray*}

Let us set
\begin{eqnarray*}
  r_3&=& \|z - p_3\| = \left( 1+ 2r \cos \varphi + r^2 \right)^{1/2}, \\ r_2&=&
  \|z - p_2\| = \left( 1+ r \cos \varphi - r \sqrt{3} \sin \varphi + r^2
  \right)^{1/2} \\ &=&  \left( 1+ 2r \cos (\varphi+\pi/3) + r^2 \right)^{1/2}.
\end{eqnarray*}

Using $m_3=1-m_1-m_2$
we separate the potential $V$ into several parts as follows
\begin{equation}
V(r,\varphi;m_1,m_2)= V_0(r,\varphi) + m_1 V_1(r,\varphi) + m_2 V_2(r,\varphi) +
\frac{m_1}{r},  \label{eq:V}
\end{equation}
where
\begin{eqnarray}
  V_0(r,\varphi)&=&  \frac{r^2}{2} +  r\cos \varphi + \frac{1}{r_3},
  \label{eq:V0} \\ V_1(r,\varphi)&=& -r \cos \varphi   - \frac{1}{r_3},
  \label{eq:V1} \\ V_2(r,\varphi)&=& -r \cos (\varphi-\pi/3) + \frac{1}{r_2} -
  \frac{1}{r_3}   \label{eq:V2} \\ &=&- \frac{r}{2} \cos \varphi -
  \frac{\sqrt{3}}{2} r \sin \varphi + \frac{1}{r_2} - \frac{1}{r_3}. \nonumber
\end{eqnarray}

 It turns out  that  the point $p_1$ is
a critical point for the potential
$$\tilde{V}(r,\varphi;m_1,m_2)= V_0(r,\varphi) +
m_1 V_1(r,\varphi) + m_2 V_2(r,\varphi)$$ for any $m_1,m_2$, hence
obtaining that ($z$ here denotes cartesian coordinates)
\begin{equation}
  \nabla V_i (p_1+z) = O(|z|), \quad  D^2 V_i (p_1+z) = O(1), \qquad i=0,1,2.
  \label{eq:grad-Vi}
\end{equation}

\subsubsection{More useful formulae}

\begin{eqnarray}
  r_3(r,\varphi)&=&1+r \cos (\varphi)+\frac{1}{2} r^2 \sin
  ^2(\varphi)-\frac{1}{2} r^3 \left(\sin ^2(\varphi) \cos
  (\varphi)\right)+O\left(r^4\right), \nonumber \\
  1-\frac{1}{r_3^3(r,\varphi)}&=& 3 r \cos (\varphi)+r^2 \left(\frac{3 \sin
  ^2(\varphi)}{2}-6 \cos ^2(\varphi)\right)+ \nonumber \\ &+& r^3 \left(10 \cos
  ^3(\varphi)-\frac{15}{2} \sin ^2(\varphi) \cos
  (\varphi)\right)+O\left(r^4\right). \label{eq:r3-3exp}
\end{eqnarray}

To obtain a formula for $r_2(r,\varphi)$ or $1-\frac{1}{r_2^3}$ it is enough
to do the substitution $\varphi \to \varphi + \pi/3$ in the above expressions.

 From (\ref{eq:V0}) we obtain
\begin{eqnarray}
  \frac{\partial V_0}{\partial \varphi}&=& -r \sin \varphi - \frac{r \sin
  \varphi}{r_3^3}= -r \sin \varphi  \left(1 - \frac{1}{r_3^3} \right),
  \label{eq:dV0dfi} \\ \frac{\partial V_0}{\partial r}&=& r + \cos \varphi -
  \frac{r+\cos \varphi}{r_3^3}= (r + \cos \varphi) \left(1 - \frac{1}{r_3^3}
  \right), \label{eq:dV0dr}
\end{eqnarray}

From (\ref{eq:V1}) we have
\begin{eqnarray}
  \frac{\partial V_1}{\partial \varphi}&=& -\frac{\partial V_0}{\partial
  \varphi}= r \sin \varphi \left( 1 - \frac{1}{r_3^3}\right), \label{eq:dV1dfi}
  \\ \frac{\partial V_1}{\partial r}&=& - \cos \varphi + \frac{r+ \cos
  \varphi}{r_3^3}. \label{eq:dV1dr}
\end{eqnarray}

From (\ref{eq:V2}) we have
\begin{eqnarray*}
   \frac{\partial V_2}{\partial \varphi}&=&   \frac{r}{2} \sin \varphi -
   \frac{\sqrt{3}}{2} r \cos \varphi + \frac{r \sin \varphi + r \sqrt{3} \cos
   \varphi}{2 r^3_2} - \frac{r \sin \varphi}{r^3_3}= \\ &=& r \sin \varphi
   \left(1-\frac{1}{r_3^3} \right) - \frac{r}{2}\sin
   \varphi\left(1-\frac{1}{r_2^3} \right)- \frac{\sqrt{3} r}{2}\cos
   \varphi\left(1-\frac{1}{r_2^3} \right) = \\ &=& r \sin \varphi
   \left(1-\frac{1}{r_3^3} \right)  - r \sin(\varphi + \pi/3)
   \left(1-\frac{1}{r_2^3} \right), \\
    \frac{\partial V_2}{\partial r}&=&\frac{1}{2} \left(-\sqrt{3} \sin (\varphi)-\cos (\varphi)\right) +
   \frac{ \cos (\varphi)+r}{r_3^3} \\
   & & +\frac{\sqrt{3} \sin (\varphi)-\cos
   (\varphi)-2 r}{2 r_2^3} = \\
   &=& -\frac{\sqrt{3}}{2}\left(1 -
   \frac{1}{r_2^3} \right) \sin \varphi - \frac{\cos \varphi}{2} \left(1 +
   \frac{1}{r_2^3} \right) - \frac{r}{r_2^3}+  \frac{ \cos (\varphi)+r}{r_3^3}
\end{eqnarray*}

A nicer and better organized expression for $ \frac{\partial V_2}{\partial
\varphi}$ is
\begin{eqnarray}
    \frac{\partial V_2}{\partial \varphi}(r,\varphi)&=& -  \frac{\partial
    V_0}{\partial \varphi}(r,\varphi) + \frac{\partial V_0}{\partial
    \varphi}(r,\varphi+\pi/3).  \label{eq:dV2dfi-org}
\end{eqnarray}

We will also need second derivatives. From (\ref{eq:dV0dfi},\ref{eq:dV0dr}) we
obtain
\begin{eqnarray}
  \frac{\partial^2 V_0}{\partial \varphi^2}&=&- r \cos \varphi
  \left(1-\frac{1}{r_3^3} \right) + \frac{3 r^2 \sin^2 \varphi}{r_3^5}, \\
  \frac{\partial^2 V_0}{\partial r \partial \varphi}&=& - \sin \varphi \left(1
  - \frac{1}{r_3^{3}} \right) - r \sin \varphi \frac{\cos \varphi + r}{r_3^5}.
  \label{eq:d2V0drdfi}
\end{eqnarray}

Third derivatives:

\begin{eqnarray*}
\frac{\partial^3 V_0}{\partial r\partial^2\varphi} &=&
-\cos\varphi +\cos\varphi r_3^{-3}
+\left(3r\sin^2\varphi-r\cos\varphi(\cos\varphi+r)+r\sin^2\varphi\right)r_3^{-5} \\
&+& \left(5r^2\sin^2\varphi(\cos\varphi+r)\right)r_3^{-7}
\end{eqnarray*}

For the future use observe that we can factor $r^2$ from $\frac{\partial V_0}{\partial \varphi}$ as follows (see (\ref{eq:r3-3exp}))
\begin{eqnarray}
 \nonumber    -\frac{1}{r^2}\frac{\partial V_0}{\partial \varphi}&=&3 \cos \varphi \sin
    \varphi + \left(-6 \cos^2\varphi \sin\varphi + (3/2) \sin^3 \varphi
    \right)r + \sin \varphi O(r^2)= \\ &=& 3 \cos \varphi \sin \varphi +
    \left(-\frac{15}{2} \cos^2\varphi  + \frac{3}{2} \right)r \sin \varphi +
    \sin \varphi O(r^2). \label{eq:V_0dfi/r^2}
\end{eqnarray}

\subsubsection{Solving $\frac{\partial V}{\partial \varphi}(r,\varphi)=0$ for $\varphi(r)$. }
\label{sec:vreal-sol-varphi(r)}

As we will see in the next subsections, we can prove the existence of 
four continuous curves $\varphi(r)$ to 
the equation $\frac{\partial V}{\partial \varphi}(r,\varphi)=0$, each of these 
curves satisfying $\varphi(r)=\left\{0, \frac{\pi}{2}, \pi, \frac{3\pi}2\right\}+O(r)$. 
This is very similar to the result in lemma~\ref{lem:vr-fp-int-estm}. However, 
as we will see in the next subsections, we get this differently: it is enough 
to study this problem for  $\frac{\partial V_0}{\partial \varphi}(r,\varphi)=0$
since the other terms are perturbative in terms of the two small masses $m_1, m_2$.

\subsubsection{Equation $\frac{\partial V_0}{\partial \varphi}=0$}
\label{subsec:dv0dfi=0}

From (\ref{eq:dV0dfi}) we see that the solution of $\frac{\partial
V_0(r,\varphi)}{\partial \varphi}=0$ is given by $\sin \varphi =0$ which is
$\varphi=0$ or $\varphi=\pi$ (these are the solutions "collinear" with the
large body at $p_3$) or by
\begin{equation}
  1-\frac{1}{r_3^3}=0,
\end{equation}
which is equivalent to (we drop the solution $r=0$)
\begin{equation}
  \cos \varphi(r)=-\frac{r}{2}.  \label{eq:cosvarphir2}
\end{equation}
There are two branches of solutions of (\ref{eq:cosvarphir2}) , denoted by
$\varphi_0^\pm(r)$.  The series expansion for the first branch is
\begin{equation}
  \varphi_0^+(r)=\frac{\pi}{2} + \frac{r}{2} + \frac{r^3}{48} + O(r^4)
\end{equation}
and for the other branch
\begin{equation}
  \varphi_0^-(r)=\frac{3\pi}{2} - \frac{r}{2} -   \frac{r^3}{48} + O(r^4).
\end{equation}

Actually, we have the more accurate result.
\begin{lemma}
\label{rem:bounds-phi0}
From code \textbf{Bound\_1} we obtain for $r \in [0,R]$, $R=10^{-1}$, that
\begin{equation}
 \varphi_0^+(r)\in \frac{\pi}{2} + \frac{r}{2} + \frac{r^3}{48}
 + [0.,  0.0011841572]r^4
\end{equation}
\begin{equation}
  \varphi_0^-(r)\in \frac{3\pi}{2} - \frac{r}{2} -\frac{r^3}{48}
 - [0.,  0.0011841572]r^4
\end{equation}
\begin{equation}
 \frac{\partial \varphi_0^+}{\partial r}(r)
 \in \frac{1}{2} + \frac{r^2}{16} +
 [ 0.,  0.04366287]r^3
\end{equation}
\begin{equation}
 \frac{\partial \varphi_0^-}{\partial r}(r) \in -\frac{1}{2} - \frac{r^2}{16} -
 [ 0.,  0.04366287]r^3
\end{equation}
\end{lemma}
Observe that  $\varphi_0^+(r)-\varphi_0^+(0)=-(\varphi_0^-(r)-\varphi_0^-(0))$.
This is implied by   $(r,\varphi^\pm(r))$ being  just a  parametrisation of the circle $r_3=1$.

Since  $r_3=1$ on $(r,\varphi_0^\pm)$   from (\ref{eq:dV0dr}) we obtain
\begin{equation}
   \frac{\partial V_0}{\partial r}(r,\varphi_0^\pm(r))= 0. \label{eq:dV0dr0zero-phipm0}
\end{equation}

\subsubsection{Analysis of $\frac{\partial V}{\partial \varphi}=0$}

In this subsection we study the full problem of $\frac{\partial V}{\partial \varphi}=0$
treating it as the perturbation of the 
curves solving equation $\frac{\partial V_0}{\partial \varphi}=0$ 
considered in Section~\ref{subsec:dv0dfi=0}.

Observe that from (\ref{eq:V},\ref{eq:dV1dfi},\ref{eq:dV2dfi-org}) it follows that
\begin{eqnarray}
  \frac{\partial V}{\partial \varphi}(r,\varphi)&=&
  \frac{\partial V_0}{\partial \varphi}(r,\varphi)(1 - m_1 - m_2)+  m_2 \frac{\partial V_0}{\partial \varphi}(r,\varphi+\pi/3) = \nonumber \\
  & &  m_3 \frac{\partial V_0}{\partial \varphi}(r,\varphi) +
  m_2 \frac{\partial V_0}{\partial \varphi}(r,\varphi+\pi/3).   \label{eq:dVdfi-org}
\end{eqnarray}
Therefore we rewrite equation $\frac{\partial V}{\partial \varphi}=0$ as
\begin{equation*}
  \frac{\partial V_0}{\partial \varphi}(r,\varphi) +
  \frac{m_2}{m_3}  \frac{\partial V_0}{\partial \varphi}(r,\varphi + \pi/3)=0.  \label{eq:smass-ang-eq}
\end{equation*}
where $m_3=1-m_1-m_2$.

The implicit function theorem implies that
\begin{equation*}
  \varphi(r)=\varphi_0(r) + m_2 \delta(r,m_1,m_2), \label{eq:varph-delta}
\end{equation*}
where $\delta(r,m_1,m_2)$ is analytic and $\varphi_0(r)$ is any of the four curves satisfying
\begin{equation*}
  \frac{\partial V_0}{\partial \varphi}(r,\varphi_0(r))=0.
\end{equation*}
Below we develop constructive estimates.

\subsubsection{Bounds for $\delta(r,m_1,m_2)$}
\label{subsec:bdelta}
We want to solve Equation \eqref{eq: partial_V_partial_phi} for
$\varphi(r,m_1,m_2)=\varphi_0(r) + \Delta(r,m_1,m_2)$, where $\varphi_0(r)$ is any of the four branches
of solutions of $\frac{\partial V_0}{\partial \varphi}(r,\varphi_0(r))=0$.

Let us set (compare (\ref{eq:V_0dfi/r^2}))
\begin{equation}
  f_0(r,\varphi):=\frac{1}{r^2} \frac{\partial V_0}{\partial \varphi}(r,\varphi) \label{eq:f0def}
\end{equation}
and then
\begin{eqnarray}
  f(r,\varphi)&=&\frac{1}{r^2} \left(\frac{\partial V_0}{\partial \varphi}(r,\varphi) +
  \frac{m_2}{m_3}  \frac{\partial V_0}{\partial \varphi}(r,\varphi + \pi/3)\right) =   \label{eq:fsmass-ang-eq} \\
  &=& f_0(r,\varphi) + \frac{m_2}{m_3} f_0(r,\varphi + \pi/3) \nonumber
\end{eqnarray}

Equation (\ref{eq: partial_V_partial_phi}) becomes
\begin{equation}
  f(r,\varphi)=0.  \label{eq:f0r0}
\end{equation}

From Expansion (\ref{eq:r3-3exp}) we obtain the following result.
\begin{lemma}
\label{lem:estm-g-r33}
Let $R_0 <1$. Then
\begin{equation}
  \frac{1}{r} \left(1-\frac{1}{r_3^3(r,\varphi)} \right)= 3 \cos \varphi +
  r \left(\frac{3 \sin^2(\varphi)}{2}-6 \cos^2(\varphi)\right) + r^2 g(r,\varphi),
    \label{eq:1p3r33}
\end{equation}
where for $r \leq R_0$ it holds that
\begin{eqnarray*}
  |g(r,\varphi)| < D_1, \qquad 
\left|\frac{\partial g}{\partial \varphi}(r,\varphi)\right| < D_2, \quad
     \qquad \left|\frac{\partial^2 g}{\partial \varphi^2}(r,\varphi)\right| < D_3,
     \qquad \left|\frac{\partial g}{\partial r}(r,\varphi)\right| < D_4,\\
     \qquad \left|\frac{\partial^2 g}{\partial r\partial \varphi}(r,\varphi)\right| < D_5,
\end{eqnarray*}
for some positive constants $D_1$, $D_2$, $D_3$, $D_4$ and $D_5$.

In particular, for $R_0=10^{-3}$ we have constants \\
$D_1 = 10.0662, D_2 = 16.665, D_3 = 45.4482$, \\
 $D_4= 98.4105, D_5= 173.801$. 
The code
with its proof is named \textbf{LEMMA\_gs}.
\end{lemma}
\begin{proof}
The idea of the proof is that with the help of Automatic Differentiation and Interval arithmetics we can
compute the Taylor expansion of the left hand side of \eqref{eq:1p3r33} in any interval, call it $f(r,\varphi)$.
Then, since $f(r, \varphi) = r^3g(r, \phi)$, $g(r, \phi)\in f_{3,0}([0, R_0]\times[0, 2\pi])$ ($f_{3,0}$
is the coefficient $(3,0)$ of $f$). Also, $\partial_\varphi g\in f_{3,1}$ and $\partial_{\varphi,\varphi}g
\in 2f_{3,2}$. For the partial derivatives of $g$ with respect to $r$, one uses
$r \partial_r f-3f=r^4 \partial_r g$ and proceed as before replacing the computations on $f$ with
$r\partial_r f-3f$.
\end{proof}

We will need some expressions for the derivatives of $f_0$.
From (\ref{eq:f0def},\ref{eq:dV0dfi}) it follows that
\begin{equation}
f_0(r,\varphi)=-\frac{1}{r} \sin \varphi \left( 1- \frac{1}{r_3^3} \right). \label{eq:f0-def}
\end{equation}
From (\ref{eq:f0-def}) and (\ref{eq:1p3r33}) in Lemma~\ref{lem:estm-g-r33} we obtain
\begin{eqnarray}
f_0(r,\varphi)&=&-\sin \varphi \left(3 \cos \varphi + r \left(\frac{3 \sin^2(\varphi)}{2}
-6 \cos^2(\varphi)\right) + r^2 g(r,\varphi) \right)  \label{eq:f0exp} \\
&=&-\frac{3}{2} \sin (2\varphi) + r \sin \varphi \left(\frac{15}{2} \sin^2\varphi - 6 \right) - r^2 \sin \varphi g(r,\varphi) \nonumber \\
\frac{\partial f_0}{\partial \varphi} (r,\varphi)&=& -3 \cos(2\varphi) + \label{eq:df0dfiexp} \\
&+& r \cos \varphi \left(\frac{45}{2} \sin^2 \varphi - 6 \right)
- r^2 \left( \cos \varphi g(r,\varphi) + \sin \varphi \frac{\partial g}{\partial \varphi}(r,\varphi)\right), \nonumber \\
\frac{\partial^2 f_0}{\partial \varphi^2} (r,\varphi)&=&  6 \sin(2\varphi) +
r \left( \sin \varphi \left(45 \sin \varphi \cos \varphi-\frac{45}{2} \sin^2 \varphi + 6 \right)\right) +
\label{eq:d2f0dfi}\\
&-& r^2 \left( -\sin \varphi g(r,\varphi) + 2 \cos\varphi \frac{\partial g}{\partial \varphi} +
\sin \varphi \frac{\partial^2 g}{\partial \varphi^2} \right), \nonumber\\
\frac{\partial f_0}{\partial r}(r,\varphi)&=&\sin \varphi \left(\frac{15}{2} \sin^2\varphi - 6 \right) - 2 r \sin \varphi g(r,\varphi)
-r^2 \sin\varphi\frac{\partial g}{\partial r}(r, \varphi),\label{eq:df0dr} \\
\frac{\partial^2 f_0}{\partial r\partial \varphi}(r,\varphi)&=&
\cos \varphi \left(\frac{45}{2} \sin^2\varphi - 6 \right) - 2 r \cos \varphi g(r,\varphi)+\label{eq:d2f0drdphi} \\
&-&2r\sin\varphi\frac{\partial g}{\partial\varphi}(r,\varphi)
-r^2 \cos\varphi\frac{\partial g}{\partial r}(r, \varphi)
-r^2 \sin\varphi\frac{\partial^2 g}{\partial r\partial\varphi}(r, \varphi).
\nonumber
\end{eqnarray}

\subsubsection{Estimates on $\delta$}

Here we establish the following theorem, bounding the solution curves $\varphi(r)$ of the following equation ${\frac{\partial V}{\partial \varphi}(r,\varphi_0(r))=0}$. 

\begin{theorem}
\label{thm:estm-delta}
Consider equation (\ref{eq:f0r0}). 
Let $\varphi_0(r)$ be any of the four
corresponding branch of solutions of $\frac{\partial V_0}{\partial \varphi}=0$,
let $\alpha$ be a positive number and let $J$ be any of the intervals centered 
at $\left\{0, \frac\pi 2, \pi, \frac{3\pi}2\right\}$ with width $\frac\alpha 2$. 

With $R_0$, $m_1, m_2$, and $D_1$, $D_2$, $D_3$ 
be as in Lemma~\ref{lem:estm-g-r33} we obtain the bounds
\begin{eqnarray}
    Z_0(f_0,r)&=& \sup_{\varphi \in \mathbb{R}} |f_0(r,\varphi)| \leq  
\frac{3}{2} + \frac{5}{2}r + r^2 D_1, \label{eq:C0f0r} \\
  d_0(\alpha,r)&=& \inf_{\varphi \in J} 
\left|\frac{\partial f_0}{\partial \varphi}(r,\varphi)\right|
  \geq 3 \cos(\alpha) - 6 r  -r^2\sqrt{D_1^2 + D_2^2}, \label{eq:d0} \\
  d_f(r)&=& \sup_{\varphi \in \mathbb{R}} 
\left|\frac{\partial f_0}{\partial \varphi}(r,\varphi)\right|
  \leq 3 + 6 r + r^2 (D_1 + D_2). \label{eq:dfr}
\end{eqnarray}
Assume that $0<R_1 \leq R_0$ is such that
\begin{eqnarray}
  d_0(\alpha,R_1) &>& \frac{m_2}{m_3}d_f(R_1) \label{eq:d0gdfr}  \\
   |\varphi_0(0)- \varphi_0(R_1)| +  \frac{m_2}{m_3} \frac{Z_0(f_0,R_1)}{d_0(\alpha,R_1)}  &\leq& \alpha/2.   \label{eq:containedInGamma}
\end{eqnarray}
Then any solution  $\varphi(r)$ of (\ref{eq:f0r0}) is defined for $0 <r \leq R_1$ and satisfies
\begin{equation}
  \varphi(r)=\varphi_0(r) + m_2 \delta(r,m), 
\qquad |\delta(r,m)| \leq \frac{ Z_0(f_0,r)}{m_3  d_0(\alpha,r)}. \label{eq:delta-estm}
\end{equation}

In particular, this Theorem is true for $m_1, m_2\in [0, 10^{-2}]$, 
$R\leq 10^{-3}$, $\alpha=\frac{\pi}{4}$ and constants $D_1, D_2$ and
$D_3$ as in Lemma \ref{lem:estm-g-r33}. 
The proof is in the code \textbf{Theorem\_phis} and we obtain
\[|\delta(r, m_1,m_2)|=|\delta([0, 10^{-3}], [0,10^{-2}]^2)|\leq 0.724802.\]
\end{theorem}

\noindent
\textbf{Proof:}

The basic idea of the proof is to find $\Delta$, such that $f(r,\varphi_0(r)\pm \Delta)$ have opposite signs and  $\frac{\partial f}{\partial \varphi}(r,\varphi_0(r) + [-\Delta,\Delta])$ is positive or negative. We should also obtain that $\Delta=O(m_2)$.

By combining assumptions  (\ref{eq:d0},\ref{eq:dfr}) we obtain that for $\varphi \in J$,  and $r \leq R_0$ holds
\begin{equation}
  \left|\frac{\partial f}{\partial \varphi}(r,\varphi) \right| \geq d_0(\alpha,r) - \frac{m_2}{m_3}d_f(r). \label{eq:dfbelow}
\end{equation}
From assumption (\ref{eq:d0gdfr}) it follows that for a fixed $r \in (0,R_1]$ equation (\ref{eq:f0r0}) has at most one solution. This is in fact already contained in Lemma~\ref{lem:vr-fp-int-estm}, but here it is made explicit.

Let us fix $0 < \Delta \leq \alpha/2$.
From assumptions (\ref{eq:containedInGamma})  and  (\ref{eq:d0}) we see that
$f_0(r,\varphi_0(r)-\Delta)$ and $f_0(r,\varphi_0(r)+\Delta)$ have  opposite signs and
\begin{eqnarray}
  |f_0(r,\varphi_0(r)\pm \Delta)| \geq \Delta d_0(\alpha,r). \label{eq:f0Delta}
\end{eqnarray}
To prove that
$f(r,\varphi_0(r)-\Delta)$ and $f(r,\varphi_0(r)+\Delta)$ have opposite signs we require that (we use (\ref{eq:f0Delta}) and (\ref{eq:C0f0r}) )
\begin{equation}
  \Delta d_0(\alpha,r) \geq  \frac{m_2}{m_3} Z_0(f_0,r).
\end{equation}
Therefore we have proved that
\begin{equation}
  \varphi(r)=\varphi_0(r) + m_2 \delta(r,m), \qquad |\delta(r,m)| \leq \frac{ Z_0(f_0,r)}{m_3  d_0(\alpha,r)}
\end{equation}
 provided that it holds
\begin{equation}
  \varphi_0(r) \pm \frac{m_2}{m_3} \frac{Z_0(f_0,r)}{d_0(\alpha,r)} \subset J.
\end{equation}
\qed

\subsubsection{Estimates on $\frac{\partial }{\partial r}\delta(r,m)$}

We would like to find explicit bounds on $\frac{\partial \delta}{\partial r}(r,m)$, which will be $O(1)$ for $m_2 \to 0$.

We will work under assumptions of Theorem~\ref{thm:estm-delta}.

Let $\Delta(r;m)=m_2 \delta(r;m)$. We will use the notation $\delta'(r)=\frac{\partial \delta}{\partial r}(r,m)$
and  $\Delta'(r)=\frac{\partial \Delta}{\partial r}(r,m)$

Differentiating Equation \eqref{eq:f0r0} with respect to $r$ we obtain  (we write $\varphi(r)=\varphi_0(r) + \Delta(r)$)
\begin{eqnarray*}
 & & 0=\frac{\partial f_0}{\partial r}(r,\varphi(r)) + \frac{m_2}{m_3} \frac{\partial f_0}{\partial r}(r,\pi/3 + \varphi(r))+ \\
 & &   \left(\frac{\partial f_0}{\partial \varphi}(r,\varphi(r)) +
 \frac{m_2}{m_3} \frac{\partial f_0}{\partial \varphi}(r,\pi/3 + \varphi(r)) \right) \left( \varphi_0'(r) + \Delta'(r) \right)
\end{eqnarray*}
Hence
\begin{eqnarray}
  \Delta'(r)&=&  -\left(\frac{\partial f_0}{\partial \varphi}(r,\varphi(r)) +
  \frac{m_2}{m_3} \frac{\partial f_0}{\partial \varphi}(r,\pi/3 + \varphi(r)) \right)^{-1} \cdot  \label{eq:Delta'} \\
 & & \left( \frac{\partial f_0}{\partial r}(r,\varphi(r)) +
 \frac{m_2}{m_3} \frac{\partial f_0}{\partial r}(r,\pi/3 + \varphi(r)) \right) -\varphi_0'(r). \nonumber
\end{eqnarray}

We will now estimate various terms in \eqref{eq:Delta'} in order to show that it is $O(m_2)$ times a bounded function.

Since by the definition of curves $\varphi_0(r)$ it holds that
\begin{eqnarray*}
   f_0(r,\varphi_0(r))=0,
\end{eqnarray*}
we obtain that
\begin{eqnarray*}
 0 &=& \frac{\partial f_0}{\partial r}(r,\varphi_0(r)) +  \frac{\partial f_0}{\partial \varphi}(r,\varphi_0(r)) \varphi'_0(r),  \\
  \varphi'_0(r) &=& - \left( \frac{\partial f_0}{\partial \varphi}(r,\varphi_0(r)) \right)^{-1}
  \frac{\partial f_0}{\partial r}(r,\varphi_0(r)).
\end{eqnarray*}

With the help of the Taylor theorem we obtain that there exists $\theta_1(r) \in (0,1)$ such that
\begin{eqnarray*}
  \frac{\partial f_0}{\partial \varphi}(r,\varphi(r)) =
 \frac{\partial f_0}{\partial \varphi}(r,\varphi_0(r)) +
 m_2 \frac{\partial^2 f_0}{\partial \varphi^2}(r,\varphi_0(r) + \theta_1(r) \Delta(r))\delta(r),
\end{eqnarray*}
implying that
\begin{eqnarray}
  \frac{\partial f_0}{\partial \varphi}(r,\varphi(r)) +  \frac{m_2}{m_3} \frac{\partial f_0}{\partial \varphi}(r,\pi/3 + \varphi(r))=
  \frac{\partial f_0}{\partial \varphi}(r,\varphi_0(r))\left(1 +  m_2 h(r) \right),  \label{eq:delta'-term-inv}
\end{eqnarray}
where
\begin{eqnarray*}
   h(r)&=& \left( \frac{\partial^2 f_0}{\partial \varphi^2}(r,\varphi_0(r) + \theta_1(r) \Delta(r))\delta(r) +
  \frac{1}{m_3} \frac{\partial f_0}{\partial \varphi}(r,\pi/3 + \varphi(r)) \right) \cdot \\
    &\cdot& \left( \frac{\partial f_0}{\partial \varphi}(r,\varphi_0(r)) \right)^{-1}.
\end{eqnarray*}
Observe that $h(r)=O(1)$. Indeed from (\ref{eq:C0f0r},\ref{eq:d0}) we have
\begin{eqnarray}
  \left| \frac{\partial f_0}{\partial \varphi}(r,\pi/3 + \varphi_0(r) + \Delta(r))
  \right| &\leq& d_f(r), \quad 0<r \leq R_0,  \label{eq:2nd-term-h} \\
  \left|\frac{\partial f_0}{\partial \varphi}(r,\varphi_0(r))\right|^{-1} &\leq& \frac{1}{d_0(\alpha,r)}. \label{eq:in-h-inv}
\end{eqnarray}
From (\ref{eq:d2f0dfi}) we have for $r \leq R_1$ and $\varphi \in \mathbb{R}$
\begin{eqnarray*}
d_2(r):=  \left| \frac{\partial^2 f_0}{\partial \varphi^2}(r,\varphi)\right| < 6 + 27 r + r^2(C_1 + 2C_2 + C_3).\\
\end{eqnarray*}

From this and the estimate (\ref{eq:delta-estm}) for $\delta(r)$ we obtain
\begin{eqnarray}
   \left|\frac{\partial^2 f_0}{\partial \varphi^2}(r,\varphi_0(r) +
   \theta_1(r) \Delta(r))\delta(r)\right| < \frac{d_2(r) Z_0(f_0,r)}{m_3  d_0(\alpha,r)}. \label{eq:1st-term-h}
\end{eqnarray}
Combining (\ref{eq:1st-term-h},\ref{eq:2nd-term-h},\ref{eq:in-h-inv}) we obtain
\begin{eqnarray}
  |h(r)| < \left(\frac{d_2(r) Z_0(f_0,r)}{m_3  d_0(\alpha,r)} + \frac{d_f(r)}{m_3} \right)
  \frac{1}{d_0(\alpha,r)}. \label{eq:delta'-h-estm}
\end{eqnarray}

Therefore from (\ref{eq:delta'-term-inv}) and the above estimates we infer that
\begin{eqnarray*}
  \left(\frac{\partial f_0}{\partial \varphi}(r,\varphi(r)) +  \frac{m_2}{m_3} \frac{\partial f_0}{\partial \varphi}(r,\pi/3 + \varphi(r))\right)^{-1}= \\
  \left(\frac{\partial f_0}{\partial \varphi}(r,\varphi_0(r))\right)^{-1} +  m_2 h_2(r),
\end{eqnarray*}
where
\begin{equation}
  h_2(r)=\left(\frac{\partial f_0}{\partial \varphi}(r,\varphi_0(r))\right)^{-1}\frac{(1+m_2 h(r))^{-1}-1}{m_2}
\end{equation}
Observe that
\begin{equation*}
(1+m_2 h(r))^{-1} - 1= \sum_{k=1}^\infty (-1)^k (m_2 h(r))^k=\frac{-m_2 h(r)}{1+m_2 h(r)}
\end{equation*}
hence $h_2(r)=O(1)$ and, more explicitly,
\[
|h_2(r)|\leq \frac{1}{d_0(\alpha, r)}\frac{|h(r)|}{1-m_2 |h(r)|}.
\]

We have for some $\theta_2(r) \in (0,1)$
\begin{eqnarray}
  \frac{\partial f_0}{\partial r}(r,\varphi(r)) = \frac{\partial f_0}{\partial r}(r,\varphi_0(r)) +
m_2 \frac{\partial^2 f_0}{\partial r \partial \varphi}(r,\varphi_0(r)+ \theta_2(r) \Delta(r)) \delta(r). \label{eq:estm-dfo-dr}
\end{eqnarray}

We have
\begin{eqnarray*}
\Delta'(r)&=&  \left(\frac{\partial f_0}{\partial \varphi}(r,\varphi_0(r))\right)^{-1}
\frac{\partial f_0}{\partial r}(r,\varphi_0(r)) + \\
 &-& \left( \left(\frac{\partial f_0}{\partial \varphi}(r,\varphi_0(r))\right)^{-1} +  m_2 h_2(r)  \right) \cdot \\
 &\cdot& \left(\frac{\partial f_0}{\partial r}(r,\varphi_0(r)) +
m_2 \frac{\partial^2 f_0}{\partial r \partial \varphi}(r,\varphi_0(r)+ \theta_2(r) \Delta(r)) \delta(r) \right. \\
& & + \left. \frac{m_2}{m_3} \frac{\partial f_0}{\partial r}(r,\pi/3 + \varphi_0(r) + \Delta(r))  \right)= \\
 &=& - \left( \left(\frac{\partial f_0}{\partial \varphi}(r,\varphi_0(r))\right)^{-1} +  m_2 h_2(r)  \right) \cdot \\
 &\cdot& \left(m_2 \frac{\partial^2 f_0}{\partial r \partial \varphi}(r,\varphi_0(r)+ \theta_2(r) \Delta(r)) \delta(r)
+\frac{m_2}{m_3} \frac{\partial f_0}{\partial r}(r,\pi/3 + \varphi(r))  \right) + \\
&-& m_2 h_2(r) \frac{\partial f_0}{\partial r}(r,\varphi_0(r))
\end{eqnarray*}
It is clear that $\Delta'(r)=m_2 O(1)$. More explicitly,
\[
|\Delta'(r)|\leq \frac{1}{d_0(\alpha, r)}\left(1+m_2\frac{|h(r)|}{1-m_2|h(r)|}\right)
\left(m_2 d_4(r)+\frac{m_2}{m_3}d_3(r)\right)+m_2\frac{d_3(r)}{d_0(\alpha, r)}\frac{|h(r)|}{1-m_2|h(r)|},
\]
where
\[
d_3(r):=\sup_{\varphi\in\mathbb R}\left|\frac{\partial f_0}{\partial r}(r, \varphi)\right|\leq
\frac{5}{2}+2rC_1+r^2C_4
\]
and
\[
d_4(r):=\sup_{\varphi\in\mathbb R}\left|\frac{\partial^2 f_0}{\partial r\partial \varphi}(r, \varphi)\right|\leq
6+2r(C_1+C_2)+r^2(C_4+C_5).
\]

Hence, we have that
\[
|\delta'(r; m1, m2)|=|\delta'([0, 10^{-3}]; [0, 10^{-2}]^2)|\leq
8.56695
\]
\[
|\Delta'(r; m1, m2)|=|\Delta'([0, 10^{-3}]; [0, 10^{-2}]^2)|\leq
0.0856695
\]

\subsubsection{Non-degeneracy of the solutions for $m_1, m_2\leq 10^{-2}$ and 
$R\leq 10^{-3}$.}

The following theorem summarizes all the expansions for the four solutions, proving 
in particular the non-degeneracy 
($\frac{d}{dr}\left(\frac{\partial V}{\partial r}(r, \varphi(r))\right)\neq 0$).

Moreover, all solutions are nondegenerate.
\begin{theorem}
From the codes \textbf{Existence\_results} and \textbf{Existence\_results\_derivative},
and for masses $m_1, m_2\in[0, 10^{-2}]$ and $r\leq 10^{-3}$ we have that
\begin{itemize}
\item
\begin{eqnarray*}
\varphi(r)&\in &\frac\pi{2}+\frac{r}{2}+\frac{r^3}{48}+[ 0,
0.00119]r^4+m_2\delta(r), \\
\frac{\partial V}{\partial r}(r, \varphi(r))&\in&
m_1 r+m_2(\frac{9}{4}r+[  1.8369662715,  2.7184275039]r^2) \\
&+& m_2^2([ -1.3099756791,  2.3608417928]r) \\
&+&  m_1m_2([ -0.0242767833,  0.0242767833]r)-
\frac{m_1}{r^2} \in\\
&&m_1 r+[  2.2366574753,  2.2765696133]m_2 r-\frac{m_1}{r^2},\\
\frac{d}{dr}\left(\frac{\partial V}{\partial r}(r, \varphi(r))\right)&\in &
[  0.9997275169,  1.0000428593]m_1+\\
&&[  2.1720852328,  2.3307223616] m_2+2m_1/r^3
\end{eqnarray*}
\item
\begin{eqnarray*}
\varphi(r)&\in&\frac{3\pi}{2}-\frac{r}{2}-\frac{r^3}{48}-[ 0,
0.0001171998]r^4+m_2\delta(r), \\
\frac{\partial V}{\partial r}(r, \varphi(r))&\in&
m_1 r+m_2(\frac{9}{4}r+[ -2.7094865798, -1.8290229479]r^2)\\
& & m_2^2([ -1.3055557029,  2.3564218165]r)\\
& &m_1m_2([ -0.0242767833,  0.0242767833]r)-
\frac{m_1}{r^2} \in\\
&&m_1 r+[  2.2339921885,  2.2738069860]m_2 r-\frac{m_1}{r^2},\\
\frac{d}{dr}\left(\frac{\partial V}{\partial r}(r, \varphi(r))\right)&\in&
[  0.9997275169,  1.0000428593]m_1+\\
&&[  2.1720852328,  2.3307223616]m_2+2m_1/r^3
\end{eqnarray*}

\item
\begin{eqnarray*}
\varphi(r)&=&0+m_2\delta(r) \\
\frac{\partial V}{\partial r}(r, \varphi(r))&=&
3r+[ -3.0990751276, -2.8895899185]r^2 \\
&+&
m_1[ -2.0120180121, -1.9819212587]r \\
&+&
m_2[ -2.3047529557, -2.1865555824]r
-\frac{m_1}{r^2}\\
&\in&
[  2.9969009248,  3.0000000000]r \\
&& +[ -2.0120180121, -1.9819212587]m_1 r\\
&&+[ -2.3047529557, -2.1865555824]m_2 r-\frac{m_1}{r^2}\\
\frac{d}{dr}\left(\frac{\partial V}{\partial r}(r, \varphi(r))\right)&\in&
[  2.9419980924,  3.0090010494]+2m_1/r^3
\end{eqnarray*}

\item
\begin{eqnarray*}
\varphi(r)&=&\pi+m_2\delta(r) \\
\frac{\partial V}{\partial r}(r, \varphi(r))&=&
3r+[  2.9009801393,  3.1116899497]r\\
&+&
m_1[ -2.0180903512, -1.9878489103] r \\
&+&
m_2[ -2.3135901858, -2.1950040940] r
-\frac{m_1}{r^2}\\
&\in&[  3.0000000000,  3.0031116900] r+
\\ && +[ -2.0180903512, -1.9878489103]m_1 r+\\
&&+[ -2.3135901858, -2.1950040940]m_2 r-\frac{m_1}{r^2}\\
\frac{d}{dr}\left(\frac{\partial V}{\partial r}(r, \varphi(r))\right)&\in&
[  2.9477930888,  3.0150582104]+2m_1/r^3
\end{eqnarray*}
\end{itemize}

\end{theorem}

\end{document}